\documentclass[12pt]{amsart}
\allowdisplaybreaks
\usepackage[english]{babel} 
\usepackage{comment}
\usepackage[pdftex,textwidth=400pt,marginratio=1:1]{geometry}
\usepackage{amsfonts}
\usepackage[dvips]{graphics}
\usepackage[colorinlistoftodos]{todonotes}
\usepackage{amsmath}
\usepackage{amsthm}
\usepackage{amssymb}
\usepackage{bbm}
\usepackage{cancel}
\usepackage{color}
\usepackage[curve]{xypic}
\usepackage{graphicx}
\usepackage{mathabx}
\usepackage{tikz-cd}
\newtheorem{theorem}{Theorem}[section]

\newtheorem{proposition}[theorem]{Proposition}

\newtheorem{conjecture}[theorem]{Conjecture}

\theoremstyle{definition}
\newtheorem{definition}[theorem]{Definition}

\newtheorem{remark}[theorem]{Remark}

\theoremstyle{property}

\usepackage{graphicx}

\DeclareFontFamily{OT1}{rsfs}{}
\DeclareFontShape{OT1}{rsfs}{n}{it}{<-> rsfs10}{}
\DeclareMathAlphabet{\curly}{OT1}{rsfs}{n}{it}

\newcommand\I{\mathcal I}

\renewcommand\L{\mathcal L}

\renewcommand\O{\mathcal O}

\newcommand\cM{\mathcal M}

\newcommand\cE{\mathcal E}
\newcommand\EE{\mathbb E}

\newcommand\T{\mathbb T}

\newcommand\E{\mathbb E}
\newcommand\C{\mathbb C}

\newcommand\sfZ{\mathsf Z}

\newcommand\Q{\mathbb Q}

\newcommand\R{\mathbb R}

\newcommand\Z{\mathbb Z}

\newcommand\bfbeta{{\boldsymbol{\beta}}}
\newcommand\bfn{{\boldsymbol{n}}}

\renewcommand\t{\mathfrak t}

\newcommand\SU{\mathrm{SU}}

\newcommand\vd{\mathrm{vd}}
\newcommand\pt{\mathrm{pt}}
\newcommand\vir{\mathrm{vir}}

\newcommand\SW{\mathrm{SW}}

\newcommand\td{\mathrm{td}}
\newcommand\rk{\operatorname{rk}}

\newcommand\tr{\operatorname{tr}}

\newcommand\ch{\operatorname{ch}}

\newcommand\Hom{\operatorname{Hom}}
\renewcommand\hom{\mathcal{H}{\it{om}}}

\newcommand\Pic{\operatorname{Pic}}

\newcommand\Spec{\operatorname{Spec}\,}
\newcommand\Hilb{\operatorname{Hilb}}

\newcommand\mdot{{\scriptscriptstyle\bullet}}

\newcommand\INTO{\ar@{^{(}->}[r]}

\DeclareRobustCommand{\SkipTocEntry}[4]{}

\begin{document}
\title[Vafa-Witten invariants, continued fractions, cosmic strings]{$\SU(r)$ Vafa-Witten invariants, Ramanujan's continued fractions, and cosmic strings}
\author[G\"ottsche, Kool, Laarakker]{L.~G\"ottsche, M.~Kool, and T.~Laarakker}
\maketitle

\vspace{-1cm}

\begin{abstract}
We conjecture a structure formula for the $\SU(r)$ Vafa-Witten partition function for surfaces with holomorphic 2-form. The conjecture is based on $S$-duality and a structure formula for the \emph{vertical} contribution
previously derived by the third-named author using Gholampour-Thomas's theory of virtual degeneracy loci.

For ranks $r=2,3$, conjectural expressions for the partition function in terms of the theta functions of $A_{r-1}, A_{r-1}^{\vee}$ and Seiberg-Witten invariants were known. We conjecture new expressions for $r=4,5$ in terms of the theta functions of $A_{r-1}, A_{r-1}^{\vee}$, Seiberg-Witten invariants, \emph{and} continued fractions studied by Ramanujan. The vertical part of our conjectures is proved for low virtual dimensions by calculations on nested Hilbert schemes.

The \emph{horizontal} part of our conjectures gives predictions for virtual Euler characteristics of Gieseker-Maruyama moduli spaces of stable sheaves. In this case, our formulae are sums of universal functions with coefficients in Galois extensions of $\Q$. The universal functions, corresponding to different quantum vacua, are permuted under the action of the Galois group. 

For $r=6, 7$ we also find relations with Hauptmoduln of $\Gamma_0(r)$. We present $K$-theoretic refinements for $r=2,3,4$ involving weak Jacobi forms.
\end{abstract}

\section{Introduction}

\subsection{Tanaka-Thomas theory} \label{sec:TT}

C.~Vafa and E.~Witten's seminal paper on $S$-duality from 1994 \cite{VW} has recently attracted renewed attention in algebraic geometry. In 2017, Y.~Tanaka and R.~P.~Thomas \cite{TT1} proposed a mathematical definition for the $\SU(r)$ Vafa-Witten invariants, which we briefly review. 

Let $(S,H)$ be a complex smooth polarized surface satisfying $H_1(S,\Z) = 0$. Let $r \in \Z_{>0}$ and $c_1 \in H^2(S,\Z)$. A \emph{Higgs pair} $(\cE,\phi)$ on $S$ consists of a torsion free sheaf $\cE$ on $S$ and a morphism $\phi : \cE \to \cE \otimes K_S$ satisfying $\tr(\phi) = 0$.\footnote{The trace-free condition reflects the fact that we are interested in gauge group $\SU(r)$.} It is called $H$-stable when it satisfies Gieseker's stability condition for all $\phi$-invariant subsheaves. For any $c_2 \in H^4(S,\Z) \cong \Z$, we denote by
$$
N=N_S^H(r,c_1,c_2)
$$
the moduli space of (isomorphism classes of) rank $r$ $H$-stable Higgs pairs $(\cE,\phi)$ on $S$ satisfying $c_1(\cE) = c_1$ and $c_2(\cE) = c_2$. The main result of Tanaka-Thomas
is that $N$ admits a \emph{symmetric} perfect obstruction theory. We denote its virtual tangent bundle by $T_N^{\vir}$. The moduli space $N$ is non-compact due to the scaling action
$$
\C^* \times N \rightarrow N, \quad (t,(\cE,\phi)) \mapsto (\cE,t \cdot \phi).
$$
Assume there are no rank $r$ strictly $H$-semistable Higgs pairs on $S$ with Chern classes $c_1$, $c_2$. Then the fixed locus $N^{\C^*}$ is compact. It can be decomposed
into ``open and closed'' components (unions of connected components)
$$
N^{\C^*} = \bigsqcup_{\mu} N^{\C^*}_{\mu},
$$
where $\mu$ runs over all sequences of positive integers $\mu = (\mu_0, \ldots, \mu_{\ell})$ satisfying $\mu_0+\cdots+\mu_\ell=r$. Here $N^{\C^*}_\mu$ contains the $\C^*$-fixed Higgs pairs $(\cE,\phi)$ with weight decomposition
$$
\cE = \bigoplus_{i=0}^{\ell} \cE_i \otimes \mathfrak{t}^{-i}, \quad \rk(\cE_i) = \mu_i,
$$
where $\t$ denotes a primitive character for the $\C^*$-action. Two extreme cases are $\mu=(r)$ and $\mu = (1,\ldots, 1)=:(1^r)$. We refer to $N^{\C^*}_{(r)}$ and $N^{\C^*}_{(1^r)}$ as the \emph{horizontal} and \emph{vertical} component of $N^{\C^*}$ respectively. We denote by
$$
M=M_S^H(r,c_1,c_2)
$$
the Gieseker-Maruyama moduli space of rank $r$ $H$-stable torsion free sheaves on $S$ with Chern classes $c_1,c_2$ \cite{HL}. Note that $M$ coincides with the horizontal
component $N^{\C^*}_{(r)}$.

Tanaka-Thomas define invariants by the virtual localization formula
\begin{equation} \label{def1}
\int_{[N]^{\vir}} 1 = \int_{[N^{\C^*}]^{\vir}} \frac{1}{e^{\C^*}(\nu^{\vir})} \in \Q,
\end{equation}
where $\nu^{\vir}$ denotes the virtual normal bundle, i.e.~the moving part of $T_N^{\vir}|_{N^{\C^*}}$, and $e^{\C^*}(\cdot)$ denotes equivariant Euler class.\footnote{Independence of the invariant from the $\C^*$-equivariant parameter follows from the fact that the perfect obstruction theory on $N$ is symmetric.} Tanaka-Thomas observe that the $\C^*$-fixed part of $T_N^{\vir}|_{M}$ is equal to the virtual tangent bundle $T_M^{\vir}$ of the natural perfect obstruction theory on $M$ previously studied in T.~Mochizuki's monograph \cite{Moc}. It is (point-wise) given by
$$
T_M^{\vir}|_{[\cE]} \cong R^\mdot\Hom(\cE,\cE)_0[1],
$$
where $(\cdot)_0$ denotes trace-free part. Therefore $M$ has virtual dimension
\begin{equation*} \label{vd}
\vd = \vd(r,c_1,c_2) = 2rc_2 - (r-1)c_1^2 - (r^2-1) \chi(\O_S).
\end{equation*}
The contribution of $M$ to \eqref{def1} is the (signed) \emph{virtual Euler characteristic} \cite{TT1}
\begin{equation} \label{hor}
\int_{[M]^{\vir}} \frac{1}{e^{\C^*}(\nu^{\vir})} = (-1)^{\vd} \int_{[M]^{\vir}} c_{\vd}(T_M^{\vir}) = (-1)^{\vd} e^{\vir}(M).
\end{equation}

Keeping $S,H,r,c_1$ fixed and varying $c_2$ (still assuming there are no strictly $H$-semistable objects), we define\footnote{The factor in front of the sum is required for modularity later.} the $\SU(r)$ Vafa-Witten partition function by
$$
\sfZ_{S,H,c_1}^{\SU(r)}(q) = r^{-1} q^{-\frac{\chi(\O_S)}{2r} + \frac{r K_S^2}{24}} \sum_{c_2 \in \Z} q^{\frac{1}{2r} \vd(r,c_1,c_2)} (-1)^{\vd(r,c_1,c_2)}  \int_{[N_S^H(r,c_1,c_2)]^{\vir}} 1.
$$
Since the integral is defined by localization to fixed loci, this can be written as
\begin{equation} \label{def2}
\sfZ_{S,H,c_1}^{\SU(r)}(q) = r^{-1} \sum_{\mu} \sfZ_{S,H,c_1}^{\mu}(q),
\end{equation}
where $\sfZ_{S,H,c_1}^{\mu}(q)$ is the contribution of all components indexed by $\mu$.

\begin{remark} 
If $H K_S \leq 0$ (still assuming there are no strictly $H$-semistable objects), then only the horizontal components contribute. If in addition $H K_S < 0$ or $K_S \cong \O_S$, then the horizontal components are smooth of expected dimension and \eqref{hor} implies \cite{TT1}
$$
\sfZ_{S,H,c_1}^{\SU(r)}(q) =  r^{-1} q^{-\frac{\chi(\O_S)}{2r} + \frac{r K_S^2}{24}} \sum_{c_2 \in \Z} q^{\frac{1}{2r} \vd(r,c_1,c_2)} \, e(M_S^H(r,c_1,c_2)),
$$
where $e(\cdot)$ denotes topological Euler characteristic. In this case, the $\SU(r)$ Vafa-Witten partition function has been studied extensively since the 1990s. An (incomplete) list of references can be found in \cite{GK5}.
\end{remark}

For arbitrary $S,H,r,c_1$, there may be strictly $H$-semistable objects and one should replace $N_S^H(r,c_1,c_2)$ by the moduli space of \emph{Joyce-Song Higgs pairs} $P_S^H(r,c_1,c_2)$ \cite{TT2}. As before, these have a $\C^*$-action and components of fixed loci are indexed by $\mu$. 
The $\SU(r)$ Vafa-Witten partition function can be defined similarly and is also of the form \eqref{def2}. The definition in loc.~cit.~depends on a conjecture \cite[Conj.~1.2]{TT2}, which is proved in many cases (e.g.~when $HK_S < 0$, when $S$ is a K3 surface \cite{MT}, when $p_g(S)>0$ and $r$ is prime \cite{Laa2}).

Suppose $p_g(S)>0$. Then the Gieseker-Maruyama moduli spaces may be singular and the  fixed locus may have non-horizontal contributions. Picking a holomorphic 2-form $\theta \in H^0(S,K_S) \setminus \{0\}$, Thomas constructs a cosection of the $\C^*$-fixed obstruction theory on $N^{\C^*}$, which he uses to prove that, for $r$ \emph{prime}, only horizontal and vertical components contribute \cite[Thm.~5.23]{Tho}
$$
\sfZ_{S,H,c_1}^{\SU(r)}(q) = r^{-1} \sfZ_{S,H,c_1}^{(r)}(q) + r^{-1} \sfZ_{S,H,c_1}^{(1^r)}(q).
$$

The invariants in this section can also be viewed as \emph{reduced} Donaldson-Thomas invariants of the non-compact Calabi-Yau 3-fold $X = K_S$. This approach was pursued by A.~Gholampour, A.~Sheshmani, and S.-T.~Yau \cite{GSY}.

\subsection{Vertical contribution}

Let $(S,H)$ be a smooth polarized surface satisfying $H_1(S,\Z) = 0$. Fix $r > 0$ (not necessarily prime) and $c_1 \in H^2(S,\Z)$. As we discussed above, when there are no strictly $H$-semistable objects, Higgs pairs in the vertical component decompose into rank 1 eigensheaves. A rank 1 torsion free sheaf on $S$ is of the form $I_Z \otimes \mathcal{L}$, where $Z \subset S$ is 0-dimensional with ideal sheaf $I_Z \subset \O_S$ and $\L \in \Pic(S)$. Consequently, the vertical components can be realized as \emph{nested Hilbert schemes} of points and curves on $S$. 

We denote by $S^{[n]}$ the Hilbert scheme of $n$ points on $S$ and, for an algebraic class $\beta \in H^2(S,\Z)$, we write $|\beta|$ for the linear system of effective divisors on $S$ with class $\beta$. 
For any $\bfn = (n_0,\ldots, n_{r-1}) \in \Z_{\geq 0}^r$ and effective classes $\bfbeta = (\beta_1, \ldots, \beta_{r-1}) \in H^2(S,\Z)^{r-1}$, we define 
\begin{align*}
S^{[\bfn]} =  \prod_{i=0}^{r-1} S^{[n_i]}, \quad |\bfbeta| = \prod_{i=1}^{r-1} |\beta_i|.
\end{align*}
We are interested in the incidence locus
\begin{align*}
S_{\bfbeta}^{[\bfn]} &= \Big\{(Z_0, \ldots, Z_{r-1},C_1, \ldots, C_{r-1}) \, : \, I_{Z_{i-1}}(-C_i) \subset I_{Z_i} \,   \forall i  \Big\} \subset S^{[\bfn]} \times |\bfbeta|.
\end{align*}
When there are no strictly $H$-semistable objects, then 
$$
\bigcup_{c_2} N_S^H(r,c_1,c_2)_{(1^r)}^{\C^*}
$$ 
is isomorphic to a union of incidence loci $S_{\bfbeta}^{[\bfn]}$ for certain $\bfn$, $\bfbeta$. In Section \ref{sec:GT}, we give an expression for $\bfn$, $\bfbeta$, in terms of $r,c_1,c_2$.

In \cite{GT1,GT2}, A.~Gholampour and Thomas show that (roughly speaking) $S_{\bfbeta}^{[\bfn]}$ can be realized as the \emph{degeneracy locus} of a morphism of vector bundles on $S^{[\bfn]} \times |\bfbeta|$ and that this structure naturally endows $S_{\bfbeta}^{[\bfn]}$ with a perfect obstruction theory. We describe this perfect obstruction theory in Section \ref{sec:GT}. They show that the virtual class from this ``degeneracy'' perfect obstruction theory coincides with the one obtained from the virtual $\C^*$-localization in Section \ref{sec:TT}. 

Using \cite{GT1,GT2}, the third-named author derived a structure result for $\sfZ_{S,H,c_1}^{(1^r)}$ \cite{Laa1,Laa2}, which we now recall. Denote by
\begin{align*}
\eta(q) &= q^{\frac{1}{24}} \prod_{n=1}^{\infty}(1-q^n), \quad \Delta(q) = q \prod_{n=1}^{\infty} (1-q^n)^{24} \\
\Theta_{A_{r},\ell}(q) &= \sum_{v \in \mathbb{Z}^r} q^{\frac{1}{2} \langle v - \ell \lambda, v - \ell \lambda \rangle}, \quad \ell \in \Z, \quad \lambda = \frac{1}{r+1}(r,r-1, \ldots, 1) 
\end{align*}
the Dedekind eta function, the discriminant modular form, and the theta functions of the lattice $A_{r}$. More precisely, denoting the standard basis of $\Z^r$ by $(e_1, \ldots, e_r)$, the symmetric bilinear form of the lattice $A_{r}$ is determined by
$$
\langle e_i, e_j \rangle = \left\{ \begin{array}{cc} 2 & \textrm{if  } i=j \\ -1 & \textrm{if } |i-j|=1 \\ 0 & \textrm{otherwise.} \end{array}\right.
$$
In Appendix \ref{sec:theta}, we collect some properties of theta functions used in this paper. We also define
\begin{align} \label{def:delta}
t_{A_{r},\ell}(q) = \frac{\Theta_{A_{r},0}(q)}{\Theta_{A_{r},\ell}(q)}, \quad \delta_{a,b} = \left\{\begin{array}{cc} 1 & \mathrm{if \, } a-b \in rH^2(S,\mathbb{Z}) \\ 0 & \mathrm{otherwise.} \end{array} \right.
\end{align}
When $p_g(S)>0$, we denote the \emph{Seiberg-Witten invariant} of $\beta \in H^2(S,\Z)$ by $\SW(\beta) \in \Z$. For an algebraic class $\beta \in H^2(S,\Z)$, the linear system $|\beta|$ has a perfect obstruction theory and virtual class $|\beta|^{\vir}$ in degree $\beta(\beta-K_S) / 2$. If $|\beta|^{\vir} \neq 0$, then $\beta^2 = \beta K_S$ and $\SW(\beta) = \deg(|\beta|^{\vir})$ \cite[Prop.~6.3.1]{Moc} (see also \cite{DKO}, which treats surfaces with $H_1(S,\Z) \neq 0$).\footnote{Equality of algebro- and differential geometric SW invariants was established in \cite{CK}.} We refer to $\beta \in H^2(S,\Z)$ such that $\SW(\beta) \neq 0$ as \emph{Seiberg-Witten basic classes}. 

\begin{theorem}[Laarakker] \label{thm:Laarakker}
For any $r>1$, there exist $C_0$, $\{C_{ij}\}_{1 \leq i \leq j \leq r-1} \in \Q(\!(q^{\frac{1}{2r}})\!)$ with the following property.\footnote{The universal functions $C_0,C_{ij}$ only depend on $r$. When we want to stress this dependence, we write $C_0^{(r)}, C_{ij}^{(r)}$, though we usually suppress this dependence.} For any smooth polarized surface $(S,H)$ satisfying $H_1(S,\Z) = 0$, $p_g(S)>0$, and $c_1 \in H^2(S,\Z)$, we have
\begin{align*}
\sfZ_{S,H,c_1}^{(1^r)}(q) = &\Bigg( \frac{(-1)^{r-1}}{ r \Delta(q^r)^{\frac{1}{2}}} \Bigg)^{\chi(\O_S)} \Bigg( \frac{\Theta_{A_{r-1},0}(q)}{\eta(q)^r} \Bigg)^{-K_S^2} \\
&\times  C_0(q)^{K_S^2} \sum_{\bfbeta \in H^2(S,\Z)^{r-1}}  \delta_{c_1,\sum_i i \beta_i} \prod_{i} \SW(\beta_i) \prod_{i \leq j} C_{ij}(q)^{\beta_i \beta_j}.
\end{align*}
\end{theorem}

Up to an explicit normalization term, discussed in Section \ref{sec:cob}, the universal series $C_0,C_{ij}$ actually lie in $1+ q \, \Q[\![q]\!]$.

The previous theorem holds in the presence of strictly $H$-semistable objects \cite{Laa2}. It also holds, with a minor modification, for surfaces with $H_1(S,\Z) \neq 0$ (however, we include this assumption for later when dealing with $S$-duality). The assumption $p_g(S)>0$ is essential. The universal functions $C_0, C_{ij}$ can be expressed in terms of integrals over products of Hilbert schemes $S^{[n]}$ as discussed in Section \ref{sec:Laar}. This allows for explicit calculation of $C_0, C_{ij}$. After normalizing the leading term to 1, $C_0, C_{ij}$ were determined by the third-named author modulo $q^{15}$ for $r=2$, and modulo $q^{11}$ for $r=3$ \cite{Laa3}. In this paper, we additionally determine them for $r=4,5,6,7$ modulo $q^{13}$ (Appendix \ref{sec:data}).

\begin{definition}
Let $S$ be a smooth projective surface satisfying $H_1(S,\Z) = 0$, $p_g(S)>0$, and let $c_1 \in H^2(S,\Z)$. For any $r>1$, we define
$$
\Phi_{r,S,c_1}(q) = C_0(q)^{K_S^2} \sum_{\bfbeta \in H^2(S,\Z)^{r-1}}  \delta_{c_1,\sum_i i \beta_i} \prod_{i} \SW(\beta_i) \prod_{i \leq j} C_{ij}(q)^{\beta_i \beta_j}.
$$
\end{definition}

The data in Appendix \ref{sec:data} leads to conjectures for $\Phi_{r,S,c_1}$ with intricate modular structure. For notational simplicity, in this introduction we present our results when $S$ is minimal of general type. Then the only Seiberg-Witten basic classes are $0, K_S$ with Seiberg-Witten invariants $1, (-1)^{\chi(\O_S)}$ \cite[Thm.~7.4.1]{Mor}. The general case is presented in Section \ref{sec:verconj}.  The conjectures are confirmed up to the orders for which we determined the $C_0, C_{ij}$ as mentioned above.

\subsection{Rank 2}

The following conjecture was made by Vafa-Witten \cite{VW}, for surfaces with smooth canonical curve $C \in |K_S|$ and, more generally, by R.~Dijkgraaf, J.-S.~Park, and B.~J.~Schroers in \cite{DPS}. Astonishingly, this formula was derived from purely physical considerations (involving superconducting cosmic strings\footnote{Known by the less exciting name ``canonical curves'' to algebraic geometers.}, quantum vacua, and mass gaps) more than two decades before the mathematical definition was found \cite{TT1, TT2}, and mathematical verifications could be done \cite{Laa1}.
\begin{conjecture}[Vafa-Witten] \label{intro:conjver:rk2}
For any minimal surface $S$ of general type satisfying $H_1(S,\Z) = 0$, $p_g(S)>0$, and $c_1 \in H^2(S,\Z)$, we have
$$
\Phi_{2,S,c_1} = \delta_{c_1,0} + \delta_{c_1,K_S} (-1)^{\chi(\O_S)} t_{A_1,1}^{K_S^2}.
$$
\end{conjecture}

\subsection{Rank 3}

The following conjecture was made in \cite{GK3}. 
\begin{conjecture}[G\"ottsche-Kool] \label{intro:conjver:rk3}
For any minimal surface $S$ of general type satisfying $H_1(S,\Z) = 0$, $p_g(S)>0$, and $c_1 \in H^2(S,\Z)$, we have
$$
\Phi_{3,S,c_1} = \delta_{c_1,0} \, t_{A_2,1}^{K_S^2}  \, \big(X_{+}^{K_S^2} + X_{-}^{K_S^2} \big) + \big( \delta_{c_1,K_S} +\delta_{c_1,-K_S} \big) \, (-1)^{\chi(\O_S)} \, t_{A_2,1}^{K_S^2}, 
$$
where $X_{\pm}$ are the two roots of 
$$
X^2 - 4  t_{A_2,1}^{2} \, X + 4 t_{A_2,1} = 0.
$$
\end{conjecture}

\subsection{Rank 4}

Consider \emph{Ramanujan's octic continued fraction} \cite{Duk}
\begin{equation} \label{defu}
u(q)=\frac{\sqrt{2} q^{\frac{1}{8}}}{1+\frac{q}{1+q+ \frac{q^2}{1+q^2 + \cdots}}} = \frac{\sqrt{2} \, \eta(q) \eta(q^4)^2 }{\eta(q^2)^3}.
\end{equation}

\begin{conjecture} \label{intro:conjver:rk4}
For any minimal surface $S$ of general type satisfying $H_1(S,\Z) = 0$, $p_g(S)>0$, and $c_1 \in H^2(S,\Z)$, we have
\begin{align*}
\Phi_{4,S,c_1} = &\ \delta_{c_1,0} \Big\{ \Big( \frac{Z - Z^{-1}}{t_{A_3,2}^{-1} u(q^2)^{-4} - Z^{-1}} \Big)^{K_S^2} + \Big( \frac{Z^{-1} - Z}{t_{A_3,2}^{-1} u(q^2)^{-4} - Z} \Big)^{K_S^2} \Big\} \\
&+ \delta_{c_1,2K_S}(-1)^{\chi(\O_S)}  \Big\{ \Big( \frac{Z - Z^{-1}}{t_{A_3,2}^{-1} Z - u(q^2)^{4}} \Big)^{K_S^2} + \Big( \frac{Z^{-1} - Z}{t_{A_3,2}^{-1} Z^{-1} - u(q^2)^{4}} \Big)^{K_S^2} \Big\} \\
&+(\delta_{c_1,K_S} + \delta_{c_1,-K_S}) \, t_{A_3,1}^{K_S^2} \Big\{\big( 1 + u(q^2)^{-4} \big)^{K_S^2} + (-1)^{\chi(\O_S)}\big( u(q^2)^4+1 \big)^{K_S^2} \Big\},
\end{align*}
where $Z$ is a root of
\begin{align*}
Z -3( u(q^2)^{-2} + u(q^2)^{2}) + Z^{-1} = 0.
\end{align*}
\end{conjecture}

\subsection{Rank 5}

For rank 5, we encounter the \emph{Rogers-Ramanujan continued fraction} appearing in S.~Ramanujan's famous letter to G.~H.~Hardy \cite{Ram}
$$
r(q) = \frac{q^{\frac{1}{5}}}{1+\frac{q}{1+\frac{q^2}{1+\cdots}}}.
$$
We define auxiliary functions 
\begin{align*}
\beta_1 =  \frac{1}{25} t_{A_4,1} \big( 3 r^{-5} + 2 -8 r^5 \big), \quad \beta_2 = \frac{1}{25} t_{A_4,2} \big( 8 r^{-5}  + 2  -3 r^{5} \big). 
\end{align*}

\begin{conjecture} \label{intro:conjver:rk5}
For any minimal surface $S$ of general type satisfying $H_1(S,\Z) = 0$, $p_g(S)>0$, and $c_1 \in H^2(S,\Z)$, we have
\begin{align*}
\Phi_{5,S,c_1} = \, &\delta_{c_1,0} \Bigg\{ \Bigg( \frac{Z X_+^2Y_+^2}{\beta_1\beta_2}   \Bigg)^{\frac{K_S^2}{2}} + \Bigg( \frac{X_+^2Y_-^2}{Z\beta_1\beta_2}   \Bigg)^{\frac{K_S^2}{2}} + \Bigg( \frac{X_-^2Y_+^2}{Z\beta_1\beta_2}   \Bigg)^{\frac{K_S^2}{2}} + \Bigg(  \frac{Z  X_-^2Y_-^2}{\beta_1\beta_2}   \Bigg)^{\frac{K_S^2}{2}} \Bigg\} \\
&+\big(\delta_{c_1,K_S} + \delta_{c_1,-K_S} \big) \Big\{ \beta_1^{K_S^2} + (-1)^{\chi(\O_S)}\big(X_+^{K_S^2} + X_-^{K_S^2}\big) \Big\}  \\
&+ \big(\delta_{c_1,2K_S} + \delta_{c_1,-2K_S} \big) \Big\{ \beta_2^{K_S^2} + (-1)^{\chi(\O_S)}\big(Y_+^{K_S^2} + Y_-^{K_S^2}\big) \Big\}, 
\end{align*}
where $X_{\pm}$, $Y_{\pm}$, $Z$ are the roots of
\begin{align*}
&X^2 - \frac{4}{5} \beta_1 \big(\beta_1 t_{A_4,1}^{-1}  - 1 \big) \big(3r^{-5}+1\big) X + \frac{4}{5}\beta_1^2 \big(3r^{-5} +1\big) = 0, \\
&Y^2 - \frac{4}{5} \beta_2  \big(\beta_2  t_{A_4,2}^{-1}  - 1 \big) \big(1-3r^{5} \big) Y + \frac{4}{5}\beta_2^2 \big(1-3r^{5} \big) = 0,  \\
&Z - \frac{6}{25} \big( 8r^{-5} -13 - 8r^5 \big) + Z^{-1} = 0. 
\end{align*}
\end{conjecture}

\subsection{Hauptmoduln}

The functions $t_{A_1,1}, t_{A_2,1}$, encountered in the rank 2 and 3 cases, are related to continued fractions as follows \cite{Cha}:
\begin{align*}
u^{-2} =  t_{A_1,1}, \quad c^{-1} + 4c^2 = 3  t_{A_2,1},
\end{align*}
where $c(q)$ is \emph{Ramanujan's cubic continued fraction}
$$
c(q)=\frac{q^{\frac{1}{3}}}{1+\frac{q+q^2}{1+ \frac{q^2+q^4}{1+ \cdots}}}.
$$
Denote by $\Gamma_0(r) \leq \mathrm{SL}(2,\Z)$ the Hecke congruence subgroup of level $r$. Let $\mathfrak{h} \subset \C$ denote the upper half plane. The quotient $\Gamma_0(r) \setminus (\mathfrak{h} \cup \{\textrm{cusps}\})$ is known as the \emph{classical modular curve} and is denoted by $X_0(r)$. It has genus zero precisely when $r=1,2,\ldots, 10,12,13,16,18,25$. In these cases elements of $\C(X_0(r))$ can be expressed as rational functions in terms of a single modular function called a \emph{Hauptmodul} for $\Gamma_0(r)$. The following are Hauptmoduln for $\Gamma_0(r)$ for $r=3,4,5$
$$
j_3(q) = \frac{\eta(q)^{12}}{\eta(q^3)^{12}}, \quad j_4(q) = \frac{\eta(q)^8}{\eta(q^4)^8}, \quad j_5(q) = \frac{\eta(q)^6}{\eta(q^5)^6}, \quad q = e^{2 \pi i \tau}.
$$
The continued fractions $c,u,r$ are related to $j_3,j_4,j_5$ \cite{Duk,Cha}
\begin{align*} \label{contdvsHaupt}
c^{-3} -15 + 48 c^3 + 64 c^6 = j_3, \quad 16 u^{-8} - 16 = j_4, \quad r^{-5} - 11 - r^5 = j_5.
\end{align*}
The last equation in Conjecture \ref{intro:conjver:rk5} can be rewritten as $Z - (\frac{48}{25} j_5 +18) + Z^{-1}=0$.

\begin{remark}
For $r=6,7$, we also find interesting (conjectural) relations between the universal functions $C_0, C_{ij}$ and a Hauptmodul for $\Gamma_0(r)$ (Conjectures \ref{conjver:rk6}, \ref{conjver:rk7}), but no complete solution. We speculate that, for any rank $r$, $C_0, C_{ij}$ are given by algebraic expressions in terms of $t_{A_{r-1},\ell}$ and elements of $\C(X_0(r))$.
\end{remark}

\subsection{$S$-duality}

Let us return to the full Vafa-Witten partition function $\sfZ_{S,H,c_1}^{\SU(r)}$ with $r$ prime. The \emph{Langlands dual} of $\SU(r)$ is $^{L}\SU(r) = \SU(r) / \mu_r$. 
Naively, the Langlands dual partition function is defined by
\begin{equation*} 
\mathsf{Z}^{^{L}\SU(r)}_{S,H,c_1}(q) := r \sum_{w \in H^2(S,\Z_r)} \epsilon_r^{w \cdot c_1} \, \sfZ_{S,H,w}^{\SU(r)}(q),
\end{equation*}
where $\epsilon_r = \exp(2 \pi i /r)$ and $H^2(S,\Z_r) \cong H^2(S,\Z) / rH^2(S,\Z)$. Firstly, this formula is only well-defined when $\sfZ_{S,H,c_1}^{\SU(r)}$ is invariant under $c_1 \mapsto c_1 + r \gamma$ for $\gamma \in H^2(S,\Z)$. 
This property holds in many cases but is not proved in general (see the discussion in \cite{JK}). Secondly, if $w \in H^2(S,\Z_r)$ cannot be represented by an algebraic class, then $\sfZ_{S,H,w}^{\SU(r)}$, as defined in Section \ref{sec:TT}, is obviously zero. In this case the correct definition for $\sfZ_{S,H,w}^{\SU(r)}$ in terms of moduli spaces of \emph{twisted sheaves} was recently given by Y.~Jiang and the second-named author \cite{JK} (see also \cite{Jia}). With these definitions in place, the mathematical formulation of the \emph{$S$-duality conjecture} is as follows:
\begin{conjecture}[Vafa-Witten] \label{intro:Sdualconj} 
Let $(S,H)$ be a smooth polarized  surface satisfying $H_1(S,\Z) = 0$ and $p_g(S)>0$. Let $r$ be prime and $c_1 \in H^2(S,\Z)$ algebraic. Then $\sfZ_{S,H,c_1}^{\SU(r)}(q)$ and $\sfZ_{S,H,c_1}^{^{L}\SU(r)}(q)$ are Fourier expansions in $q = \exp(2 \pi i \tau)$ of meromorphic functions $\sfZ_{S,H,c_1}^{\SU(r)}(\tau)$ and $\sfZ_{S,H,c_1}^{^{L}\SU(r)}(\tau)$ on $\mathfrak{H}$ satisfying
\begin{equation*} 
\mathsf{Z}^{\SU(r)}_{S,H,c_1}(-1/\tau) = (-1)^{(r-1)\chi(\O_S)} \Big( \frac{ r \tau}{i} \Big)^{-\frac{e(S)}{2}} \mathsf{Z}^{^{L}\SU(r)}_{S,H,c_1}(\tau).
\end{equation*}
\end{conjecture}

Building on \cite{TT2,MT}, this conjecture was proved for K3 surfaces in \cite{JK}. $S$-duality should also hold for non-prime $r$ and surfaces with $p_g(S) = 0$ or $H_1(S,\Z) \neq 0$, but the formulation should be adapted (e.g.~see \cite{Wit}).

\subsection{Horizontal contribution}

Starting from the vertical generating series $\sfZ_{S,H,c_1}^{(1^r)}$ and applying $\tau \mapsto -1/\tau$ gives a prediction for the horizontal generating series $\sfZ_{S,H,c_1}^{(r)}$. We make this precise in Section \ref{sec:S-dual}. This strategy was applied for $r=2$ in physics in \cite{DPS}. We then obtain conjectures for virtual Euler characteristics of Gieseker-Maruyama moduli spaces, which we now present. 

The $A_r^\vee$ lattice is defined as $\Z^r$ with bilinear pairing $\langle \cdot, \cdot \rangle^\vee$ determined by the inverse of the matrix defining $\langle \cdot, \cdot \rangle$. 
See Appendix \ref{sec:theta} for details on lattice theta functions. 
Define
\begin{align*}
\Theta_{A_{r}^\vee,\ell}(q) = \sum_{v \in \mathbb{Z}^r} e^{2 \pi i \langle v, \ell (1,0,\ldots, 0) \rangle^{\vee} } q^{\frac{1}{2} \langle v , v  \rangle^{\vee}}, \quad t_{A_r^\vee,\ell}(q) = \frac{\Theta_{A_{r}^\vee,0}(q)}{\Theta_{A_{r}^\vee,\ell}(q)}, \quad \ell \in \Z.
\end{align*}
\begin{conjecture}  \label{intro:conjhor}
For any $r>1$, there exist $D_0$, $\{D_{ij}\}_{1 \leq i \leq j \leq r-1} \in \C[\![q^{\frac{1}{2r}}]\!]$ with the following property. For any smooth polarized surface $(S,H)$ satisfying $b_1(S) = 0$, $p_g(S)>0$, $c_1 \in  H^2(S,\Z)$, and $c_2 \in H^4(S,\Z)$ such that there are no rank $r$ strictly $H$-semistable sheaves on $S$ with Chern classes $c_1,c_2$, $e^{\mathrm{vir}}(M_S^H(r,c_1,c_2))$ equals the coefficient of $q^{c_2 - \frac{r-1}{2r} c_1^2 - \frac{r}{2} \chi(\O_S) + \frac{r}{24} K_S^2}$ of
\begin{align*}
&r^{2+K_S^2 - \chi(\O_S)} \Bigg( \frac{1}{\Delta(q^{\frac{1}{r}})^{\frac{1}{2}}} \Bigg)^{\chi(\O_S)} \Bigg( \frac{\Theta_{{A^\vee_{r-1}},0}(q)   }{\eta(q)^r} \Bigg)^{-K_S^2}  \\
&\quad\quad\quad\quad \times D_0(q)^{K_S^2} \sum_{\bfbeta \in H^2(S,\Z)^{r-1}}   \prod_{i} \epsilon_r^{i \beta_i c_1} \, \SW(\beta_i) \prod_{i \leq j} D_{ij}(q)^{\beta_i \beta_j}.
\end{align*}
\end{conjecture}

\begin{remark}
This formula shares a similar structure to conjectural formulae in \cite{GK1}--\cite{GK4}, \cite{GKW, Got4} for other virtual invariants (virtual $\chi_y$-genera, elliptic genera, cobordism classes, Segre numbers, Verlinde numbers, Donaldson invariants). In Appendix \ref{sec:mult}, we present a generalization of all these cases: a universal formula for \emph{any} set of ``multiplicative'' virtual invariants of Gieseker-Maruyama moduli spaces $\{M_S^H(r,c_1,c_2)\}_{c_2 \in \Z}$. Using \emph{Mochizuki's formula} \cite{Moc}, we prove a weak version of this conjecture (Theorem \ref{thm:multinst}).
\end{remark}

Parallel to the vertical case, we define:
\begin{definition} \label{def:Psi}
Let $S$ be a smooth projective surface satisfying $H_1(S,\Z) = 0$, $p_g(S)>0$, and let $c_1 \in H^2(S,\Z)$. For any $r>1$, we define
$$
\Psi_{r,S,c_1}(q) = D_0(q)^{K_S^2} \sum_{\bfbeta \in H^2(S,\Z)^{r-1}}  \prod_{i} \epsilon_r^{i \beta_i c_1} \SW(\beta_i) \prod_{i \leq j} D_{ij}(q)^{\beta_i \beta_j}.
$$
\end{definition}

Motivated by $S$-duality, we conjecture the following relation between the universal functions $C_0, C_{ij}$ and $D_0, D_{ij}$ (which implies part of the $S$-duality transformation, Conjecture \ref{intro:Sdualconj}, as we will show in Proposition \ref{prop:Sdualpart})
\begin{conjecture} \label{intro:conjCD}
For any $r>1$, $C_0(q)$, $C_{ij}(q)$, $D_0(q)$, $D_{ij}(q)$ are Fourier expansions in $q = \exp(2 \pi i \tau)$ of meromorphic functions $C_0(\tau)$, $C_{ij}(\tau)$, $D_0(\tau)$, $D_{ij}(\tau)$ on $\mathfrak{H}$ satisfying
$$
D_0(\tau) = C_0(-1/\tau), \quad D_{ij}(\tau) = C_{ij}(-1/\tau).
$$
\end{conjecture}

For notational simplicity, we state the following consequence when $S$ is minimal of general type. See Section \ref{sec:horconj} for the general case. 
\begin{proposition} \label{intro:hor:rk2-5}
Assume Conjecture \ref{intro:conjCD} holds for $r \in \{2,3,4,5\}$. For any minimal surface $S$ of general type with $H_1(S,\Z) = 0$, $p_g(S)>0$, and $c_1 \in H^2(S,\Z)$, $\Psi_{r,S,c_1}(q)$ is given by the formula for $\Phi_{r,S,c_1}(q)$ of Conjectures \ref{intro:conjver:rk2}, \ref{intro:conjver:rk3}, \ref{intro:conjver:rk4}, \ref{intro:conjver:rk5} after making the following replacements:
\begin{enumerate}
\item Replace $\delta_{c_1,\ell K_S}$ by $\epsilon_r^{\ell K_S c_1}$ for all $\ell$.
\item Replace $t_{A_{r-1},\ell}(q)$ by $t_{A_{r-1}^\vee,\ell}(q)$ for all $\ell$.
\item For $r=4$, replace $u(q^2)$ by
$$
\frac{\eta(q^{\frac{1}{2}}) \eta(q^\frac{1}{8})^2 }{\eta(q^\frac{1}{4})^3}.
$$
\item For $r=5$, replace $r(q)$ by
$$
\frac{1 - \varphi \, r(q)}{\varphi + r(q)},
$$
where $\varphi = \frac{1+ \sqrt{5}}{2}$ denotes the \emph{golden ratio}.
\end{enumerate}
\end{proposition}

For the prime ranks $r=2,3,5$, we therefore have complete conjectural formulae for $\sfZ_{S,H}^{\SU(r)}$.\footnote{There exists a formula in the physics literature for $\SU(r)$ Vafa-Witten invariants for any prime $r>2$ \cite{LL}. This formula is incorrect (see \cite[Rem.~1.8]{GK3}).} In Proposition \ref{prop:Sdualcheck}, we show that these expressions satisfy $S$-duality. For $r=2$, this was shown in \cite{VW, DPS} and for $r=3$ in \cite{GK3}. The case $r=5$ is new and involves properties of the Rogers-Ramanujan continued fraction. 

\begin{remark}
The rank 2 and 3 cases of Proposition \ref{intro:hor:rk2-5} were studied in \cite{GK1, GK3}. The strategy in loc.~cit.~is to write $e^{\mathrm{vir}}(M_S^H(r,c_1,c_2))$ in terms of (descendent) \emph{Donaldson invariants} \cite{Moc}. Using Mochizuki's formula, the latter can be reduced to intersection numbers on products of Hilbert schemes, which led to direct verifications in many examples.
\end{remark}

Since virtual Euler characteristics are integers, one expects from Conjecture \ref{intro:conjhor} that $r^{2 + K_S^2 - \chi(\O_S)}\Psi_{r,S,c_1}$ has \emph{integer} coefficients. This appears to be a non-trivial fact that we can only verify experimentally in examples. For $r>2$, it is not even obvious that $\Psi_{r,S,c_1}$ has \emph{rational} coefficients. For $r=3,4,5$ the coefficients (a priori) lie in $K = \Q(\sqrt{3},i)$, $\Q(\sqrt{2},i)$, $\Q(\epsilon_5,i)$ respectively. We show that $\Psi_{r,S,c_1}$ is indeed invariant under the action of the Galois groups $\mathrm{Gal}(K/\Q)\cong \Z_2 \times \Z_2$, $\Z_2 \times \Z_2$, $\Z_4 \times \Z_2$ respectively (Proposition \ref{prop:Gal}).  

The leading term of Conjecture \ref{intro:conjhor} can be seen as a \emph{Donaldson invariant} without insertions
$$
\int_{[M_S^H(r,c_1,c_2)]^{\vir}} 1, \quad \textrm{where} \quad \vd(M_S^H(r,c_1,c_2)) = 0.
$$
For $r=2,3,4,5$, we compute these degrees explicitly (Proposition \ref{prop:Don}). The resulting expressions can be seen as instances of the \emph{Mari\~no-Moore conjecture} \cite{MM1,LM}, and they are consistent with expressions obtained by different means in \cite{GNY1} ($r=2$, i.e.~the Witten conjecture), \cite{GK4} ($r=3,4$), \cite{Got4} ($r=5$).

\subsection{Refinement and weak Jacobi forms}

In \cite{Tho}, Thomas defines a $K$-theoretic refinement of Vafa-Witten invariants. For $(S,H)$ a smooth polarized surface, this leads to a $K$-theoretic $\SU(r)$ Vafa-Witten partition function
$$
\sfZ_{S,H,c_1}^{\SU(r)}(q,y),
$$
which, for $y=1$, reduces to the (unrefined) $\SU(r)$ Vafa-Witten partition function. We recall its definition and basic properties in Appendix \ref{sec:ref}. 

Suppose $H_1(S,\Z) = 0$, $p_g(S)>0$. The third-named author proved an analogue of Theorem \ref{thm:Laarakker} for the $K$-theoretic $\SU(r)$ Vafa-Witten partition function \cite{Laa1} (Theorem \ref{thm2:Laarakker}). Conjectural formulae for $\sfZ_{S,H,c_1}^{\SU(r)}(q,y)$ for $r=2,3$ were given by the first- and second-named author \cite{GK1, GK3} and verified in examples up to some order in $q$ in \cite{GK1, GK3, Laa1}. We recall some of these results in Appendix \ref{sec:ref}. In \cite{GK3}, a $K$-theoretic enhancement of the $S$-duality conjecture was proposed. See \cite{AMP} for recent related developments in physics. 

The main result of Appendix \ref{sec:ref} is a (surprising) $K$-theoretic refinement of Conjecture \ref{intro:conjver:rk4}. It involves the function
$$
J(q,y) = \frac{u(q^2)^4}{4} \frac{\phi_{-2,1}(q^2,y^2)\phi_{-2,1}(q^8,y^4)^2}{\phi_{-2,1}(q^4,y^2)^2 \phi_{-2,1}(q^4,y^4)},
$$
where $\phi_{-2,1}(q,y)$ is the following Fourier expansion of a weak Jacobi form of weight $-2$ and index 1
$$
\phi_{-2,1}(q,y) = (y^{\frac{1}{2}} - y^{-\frac{1}{2}})^2 \prod_{n=1}^{\infty} \frac{(1- y q^n)^2 (1-y^{-1} q^n)^2}{(1-q^n)^4}.
$$
Note that $J(q,1) = u(q^2)^4$. We also provide refinements of Conjectures \ref{intro:conjhor} and \ref{intro:conjCD}. We were unable to find a formula for the refinement for rank 5.

It is also interesting to include $\mu$-insertions. This has been explored recently in depth in the rank 2 case (for arbitrary differentiable 4-manifolds) in the physics literature by J.~Manschot and G.~W.~Moore \cite{MM2}. See also \cite[App.~C]{GK1} and \cite{GKW} for a related mathematical discussion. \\  

\noindent \textbf{Acknowledgments.} We warmly thank A.~Gholampour, Y.~Jiang, J.~Manschot, G.~W.~Moore, H.~Nakajima, A.~Sheshmani, Y.~Tanaka, R.~P.~Thomas, and K.~Yoshioka for stimulating discussions on Vafa-Witten theory. M.K.~is supported by NWO grant VI.Vidi.192.012. T.L.~is supported by EPSRC grant EP/R013349/1.

\section{Virtual degeneracy loci} \label{sec:GT}

This section is mostly a review of \cite{GT1, GT2, Laa1}. Our main goal is to explain how $C_0$, $C_{ij}$ can be expressed in terms of integrals over products of Hilbert schemes of points and how this can be used to derive the expansions for these universal functions listed in Appendix \ref{sec:data}.

\subsection{Gholampour-Thomas theory}

Let $(S,H)$ be a smooth polarized surface with $H_1(S,\Z) = 0$, $p_g(S)>0$, and let $r \in \Z_{>1}$ (not necessarily prime)
As mentioned in the introduction, for any effective algebraic $\bfbeta \in H_2(S,\Z)^{r-1}, \bfn \in \Z_{\geq 0}^{r}$, Gholampour-Thomas constructed a virtual class for $S_{\bfbeta}^{[\bfn]}$ by realizing it (roughly speaking) as a degeneracy locus of a map between vector bundles on the smooth ambient space $ S^{[\bfn]} \times |\bfbeta|$ \cite[Thm.~5.6]{GT2}. Denote by $\I_i$ the pull-back to $S \times S^{[\bfn]} \times |\bfbeta|$ of the universal ideal sheaf on $S \times S^{[n_i]}$. We also write $\iota : S_{\bfbeta}^{[\bfn]} \hookrightarrow S^{[\bfn]} \times |\bfbeta|$ for the closed embedding, $\pi : S \times S^{[\bfn]} \times |\bfbeta| \to S^{[\bfn]} \times |\bfbeta|$ for the projection, and $\pt = \mathrm{Spec}(\C)$. 
\begin{theorem}[Gholampour-Thomas]
The push-forward 
$$
\iota_* [S_{\bfbeta}^{[\bfn]}]^{\vir} \in H_{2n_0+2n_{r-1}}(S^{[\bfn]} \times |\bfbeta|,\Z)
$$ 
equals
\begin{align*}
\prod_{i=1}^{r-1} \SW(\beta_i) \cdot e\Big( R^\mdot \Gamma(S,\O_S(\beta_i)) \otimes \O - R^\mdot \hom_{\pi}(\I_{i-1},\I_i(\beta_i)) \Big) \cap [S^{[\bfn]} \times \pt \times \cdots \times \pt],
\end{align*}
where $\pt \times \cdots \times \pt \in |\bfbeta| = |\beta_1| \times \cdots \times |\beta_{r-1}|$.
\end{theorem}
In this theorem, the Euler classes in the product are well-defined by the generalized Carlson-Okounkov vanishing \cite[Thm.~3]{GT1}.

For any effective algebraic $\bfbeta \in H_2(S,\Z)^{r-1}$, and identifying $\beta_i$ with its Poincar\'e dual in $H^2(S,\Z)$, we consider (isomorphism classes of) line bundles $L_0, \ldots, L_{r-1}$ on $S$ satisfying
$$
K_S - \beta_i = c_1(L_{i-1} \otimes L_i^*).
$$
We fix the remaining indeterminacy by requiring
$$
\sum_{i=0}^{r-1} c_1(L_i) = c_1,
$$
for some fixed $c_1 \in H^2(S,\Z)$. Consider the tautological line bundle $\O_{|\beta_i|}(1)$ on the complete linear system $|\beta_i|$. We then define the following line bundles on $S \times S^{[\bfn]} \times |\bfbeta|$
\begin{align*}
\L_0 &= L_0, \\ 
\L_1 &= L_1 \boxtimes \O_{|\beta_1|}(1), \\
&\cdots \\
\L_{r-1} &= L_{r-1} \boxtimes \O_{|\beta_{1}|}(1) \boxtimes \cdots \boxtimes \O_{|\beta_{r-1}|}(1).
\end{align*}

The line bundles $\O_S(\beta_i) \boxtimes \O_{|\beta_i|}(1)$ on $S \times |\beta_i|$ have tautological sections, inducing maps $\L_{i-1} \to \L_i \otimes K_S$, which combine into a $\C^*$-equivariant Higgs pair on $S \times S^{[\bfn]} \times |\bfbeta|$
$$
\Phi_{\L} : \mathbb{E}_{\L} \to \mathbb{E}_{\L} \otimes K_S \otimes \mathfrak{t}, \quad \mathbb{E}_{\L} := \bigoplus_{i=0}^{r-1} \L_i \otimes \mathfrak{t}^{-i},
$$
where $\mathfrak{t}$ is the primitive character corresponding to the $\C^*$-scaling action. Define
\begin{equation} \label{defE}
\mathbb{E} := \bigoplus_{i=0}^{r-1} \I_i \otimes \L_i \otimes \mathfrak{t}^{-i}.
\end{equation}
Over the incidence locus $S_{\bfbeta}^{[\bfn]} \subset S^{[\bfn]} \times |\bfbeta|$, the tautological Higgs field $\Phi_{\L}$ factors through $\mathbb{E}$.\footnote{$S_{\bfbeta}^{[\bfn]}$ can be defined as the maximal closed subscheme over which $\Phi_{\L}$ factors through $\mathbb{E}$.} We stress that this construction works for any $\bfbeta, \bfn$ and we have not yet discussed whether the elements of $S_{\bfbeta}^{[\bfn]}$, viewed as Higgs pairs, are actually $H$-stable or not.

\subsection{Vertical generating series} \label{sec:Laar}

As in the previous section, let $(S,H)$ be a smooth polarized surface with $H_1(S,\Z) = 0$, $p_g(S)>0$, and let $r \in \Z_{>1}$ (not necessarily prime). Define the following quadratic form
$$
Q : H^2(S,\Z)^{r-1} \rightarrow \Q, \quad Q(a_1, \ldots, a_{r-1}) = - \sum_{i<j} \frac{i(r-j)}{r} a_ia_j - \sum_{i=1}^{r-1} \frac{i(r-i)}{2r} a_i^2.
$$

Let $r,c_1,c_2$ be chosen such that there are no rank $r$ strictly $H$-semistable Higgs pairs $(\mathcal{E},\phi)$ on $S$ with $c_1(\mathcal{E}) = c_1$ and $c_2(\mathcal{E}) = c_2$, and let $N:=N_S^H(r,c_1,c_2)$. Then Gholampour-Thomas \cite{GT1,GT2} proved that $N_{(1^r)}^{\C^*}$ is isomorphic to the union of incidence loci
$$
\iota : S_{\bfbeta}^{[\bfn]} \subset S^{[\bfn]} \times |\bfbeta|
$$
with $\bfbeta = (\beta_1, \ldots, \beta_{r-1}) \in H_2(S,\Z)^{r-1}$ and $\bfn = (n_0, \ldots, n_{r-1}) \in \Z_{\geq 0}^{r}$ such that all elements of $S_{\bfbeta}^{[\bfn]}$ are $H$-stable, and 
\begin{align}
\begin{split} \label{Cherneq}
&c_1 \equiv \sum_{i=1}^{r-1} i(K_S - \beta_i) \mod rH^2(S,\Z), \\
&c_2 = |\bfn| + \frac{r-1}{2r} c_1^2 + Q(K_S - \beta_1, \ldots, K_S - \beta_{r-1}),
\end{split}
\end{align}
where $|\bfn| := \sum_{i} n_i$. The above equations are derived in \cite[Lem.~2.6]{Laa1}. The contribution of such a component $S_{\bfbeta}^{[\bfn]} \subset N_{(1^r)}^{\C^*}$ to the (vertical) Vafa-Witten invariant is
$$
\int_{[S_{\bfbeta}^{[\bfn]}]^{\vir}} \frac{1}{e(\nu^{\vir}|_{S_{\bfbeta}^{[\bfn]}})} \in \Q.
$$
Recall that $\nu^{\vir}|_{S_{\bfbeta}^{[\bfn]}}$ is the moving part of the virtual tangent bundle $T_N^{\vir}|_{S_{\bfbeta}^{[\bfn]}}$. As expected, the $K$-theory class $\nu^{\vir}|_{S_{\bfbeta}^{[\bfn]}} \in K^{\C^*}_0(S_{\bfbeta}^{[\bfn]})$ can be written as the restriction of a class on $S^{[\bfn]} \times |\bfbeta|$. The result is the following \cite[Sect.~4]{Laa1}
$$
T_N^{\vir}|_{S_{\bfbeta}^{[\bfn]}} = \Big( R^\mdot \hom_\pi(\mathbb{E}, \mathbb{E} \otimes K_S \otimes \mathfrak{t})_0 - R^\mdot \hom_\pi(\mathbb{E}, \mathbb{E})_0 \Big)\Big|_{S_{\bfbeta}^{[\bfn]}}.
$$

Conveniently, incidence loci $S_{\bfbeta}^{[\bfn]}$ containing $H$-unstable elements do not contribute by the following observation \cite[Prop.~3.5]{Laa1}:
\begin{proposition}[Laarakker]
If $S_{\bfbeta}^{[\bfn]}$ contains an $H$-unstable element, then $\iota_* [S_{\bfbeta}^{[\bfn]}]^{\vir} = 0$.
\end{proposition}
Therefore, we can sum \emph{all} $\bfbeta, \bfn$ satisfying \eqref{Cherneq}.

Now suppose $r,c_1$ are fixed such that there are no rank $r$ strictly $H$-semistable Higgs pairs $(\cE,\phi)$ on $S$ with $c_1(\cE) = c_1$. Combining all results so far leads to the following expression for the vertical generating series \cite[Eqn.~(4.4), (4.3)]{Laa1} \footnote{Here we set $a_i = K_S - \beta_i$ and use the well-known identity $\SW(a_i) = (-1)^{\chi(\O_S)} \SW(\beta_i)$.}
\begin{align}
\begin{split} \label{eqn1}
\sfZ_{S,H,c_1}^{(1^r)}(q) = q^{-\frac{r \chi(\O_S)}{2} + \frac{r K_S^2}{24}} \sum_{\bfn \in \Z_{\geq 0}^r, \boldsymbol{a} \in H^2(S,\Z)^{r-1}} \delta_{c_1,\sum i a_i} q^{Q(\boldsymbol{a}) + |\bfn|}  \prod_{i=1}^{r-1} \SW(a_i) \int_{S^{[\boldsymbol{n}]}} \Upsilon(\boldsymbol{a},\boldsymbol{n},t)
\end{split}
\end{align}
where $\delta_{a,b}$ was defined in \eqref{def:delta} and
\begin{align*}
\Upsilon(\boldsymbol{a},\boldsymbol{n},t) := (-1)^{(r-1) \sum_i i a_i^2 } \frac{\prod_{i=1}^{r-1} e\big(R^\mdot \Gamma(S,\O_S(K_S - a_i)) \otimes \O - R^\mdot \hom_{\pi}(\I_{i-1}, \I_i(K_S - a_i))\big)}{    e\big(   [ R^\mdot \hom_{\pi}( \E,\E \otimes K_S \otimes \mathfrak{t})_0 - R^\mdot \hom_{\pi}( \E,\E)_0 ]^{m}\big)},
\end{align*}
where $\E$ was defined in \eqref{defE} and $[\cdot]^m$ extracts the $\C^*$-moving part. We note that the expressions in this section differ slightly from \cite{Laa1}, because loc.~cit.~does not include the factor
$$
q^{-\frac{\chi(\O_S)}{2r} + \frac{r K_S^2}{24}} q^{-\frac{(r^2-1) \chi(\O_S)}{2r}} (-1)^{\vd(r,c_1,c_2)} = q^{-\frac{r\chi(\O_S)}{2} + \frac{r K_S^2}{24}} (-1)^{\vd(r,c_1,c_2)}. 
$$

\subsection{Cobordism argument} \label{sec:cob}

Now let $S$ be \emph{any} smooth projective surface (not necessarily satisfying $H_1(S,\Z)=0$ or $p_g(S)>0$). Take any algebraic classes $\boldsymbol{a} = (a_1, \ldots, a_{r-1}) \in H^2(S,\Z)^{r-1}$. Then we can consider the expression
\begin{align*}
\sum_{\bfn \in \Z_{\geq 0}^r} q^{|\bfn|}  \int_{S^{[\boldsymbol{n}]}} \Upsilon(\boldsymbol{a},\boldsymbol{n},t).
\end{align*}
The constant term of this generating series, corresponding to $n_0 = \cdots = n_{r-1} = 0$, is non-trivial and denoted by $\Upsilon(\boldsymbol{a},\boldsymbol{0},t)$. 
We define the normalized generating series:
\begin{align} \label{eqn:defGS}
\mathsf{G}_{S,\boldsymbol{a}}(q) = \frac{1}{\Upsilon(\boldsymbol{a},\boldsymbol{0},t)} \sum_{\bfn \in \Z_{\geq 0}^r} q^{|\bfn|}  \int_{S^{[\boldsymbol{n}]}} \Upsilon(\boldsymbol{a},\boldsymbol{n},t).
\end{align}
In the case all $a_i$ are Seiberg-Witten basic classes, this constant term is determined in \cite[Prop.~6.3]{Laa1}
\begin{equation} \label{eqn:Upsilonzero}
\Upsilon(\boldsymbol{a},\boldsymbol{0},t) = \Bigg(\frac{(-1)^{r-1}}{r} \Bigg)^{\chi(\O_S)} \prod_{i=1}^{r-1} \binom{r}{i}^{-a_i^2} \prod_{1 \leq i < j \leq r-1} \Bigg( \frac{j(r-i)}{(j-i)r} \Bigg)^{a_i a_j}.
\end{equation}

The generating series $\mathsf{G}_{S,\boldsymbol{a}}(q)$ is defined for any (possibly disconnected) smooth projective surface $S$ and any algebraic classes $\boldsymbol{a} \in H^2(S,\Z)^{r-1}$. For a disconnected surface $S = S' \sqcup S''$ we have
$$
S^{[n]} \cong \bigsqcup_{n = n' + n''} S^{\prime [n']} \times S^{\prime \prime [n'']}.
$$
Using this property, one can show the following \cite[proof of Prop.~7.2]{Laa1}:
\begin{proposition}[Laarakker]
For any $S = S' \sqcup S''$ and $\boldsymbol{a} = \boldsymbol{a}' \oplus \boldsymbol{a}''$, we have
$$
\mathsf{G}_{S' \sqcup S'',\boldsymbol{a}}(q) = \mathsf{G}_{S',\boldsymbol{a}'}(q) \mathsf{G}_{S'',\boldsymbol{a}''}(q).
$$
\end{proposition}

This property is very powerful when combined with the universality result of G.~Ellingsrud, M.~Lehn, and the first-named author \cite[Thm.~4.1]{EGL}. The universality theorem of \cite[Thm.~4.1]{EGL} states (roughly speaking) that tautological integrals over $\Hilb^n(S)$ are determined by universal polynomials in the Chern numbers. See also \cite{Got2}. Applied to our current setting, and combined with the previous proposition, this leads to the following result. There exist universal power series $A(q)$, $B(q)$, $\{E_i(q)\}_{i=1}^{r-1}$, $\{E_{ij}(q)\}_{1 \leq i \leq j \leq r-1}$, only depending on $r$, such that
\begin{equation} \label{eqn2}
q^{-\frac{r\chi(\O_S)}{2} + \frac{r K_S^2}{24} + Q(\boldsymbol{a})} \Upsilon(\boldsymbol{a}, \boldsymbol{0}, t) \mathsf{G}_{S,\boldsymbol{a}}(q) = A(q)^{\chi(\O_S)} B(q)^{K_S^2} \prod_i E_i(q)^{a_i K_S} \prod_{i \leq j} E_{ij}(q)^{a_i a_j}.
\end{equation}
 By specializing to a K3 surface, the universal series $A(q)$ can be explicitly determined \cite{TT2}, \cite[Thm.~C]{Laa1}
\begin{equation} \label{eqn:A}
A(q) = \frac{(-1)^{r-1}}{ r \Delta(q^r)^{\frac{1}{2}}}.
\end{equation}
Note that $\mathsf{G}_{S,\boldsymbol{a}}(q)$ only contributes to the vertical generating series \eqref{eqn1} when $a_1, \ldots, a_{r-1}$ are Seiberg-Witten basic classes. In this case, $a_i K_S = a_i^2$. We therefore define
\begin{align} \label{def:C0Cij}
C_{ii}(q)= E_i(q) E_{ii}(q), \quad C_{ij}(q) = E_{ij}(q), \quad C_0(q) = \frac{\Theta_{A_{r-1},0}(q)}{\eta(q)^r} B(q),
\end{align}
for all $i$ and $j>i$. The normalization by $\Theta_{A_{r-1},0}(q) / \eta(q)^r$ is chosen for later convenience (and in anticipation of blow-up formulae). Combining \eqref{eqn1} and \eqref{eqn2} yields Theorem \ref{thm:Laarakker} in the introduction, and it was proved by the third-named author in \cite[Thm.~A]{Laa1}.

\begin{remark}
In this section, we assumed that there are no rank $r$ strictly $H$-semistable Higgs pairs $(\cE,\phi)$ on $S$ with $c_1(\cE) = c_1$. However, Theorem \ref{thm:Laarakker} also holds without this assumption. In this case, the definition of the Vafa-Witten invariants involves Joyce-Song Higgs pairs \cite{TT2} and the derivation of Theorem \ref{thm:Laarakker} was given in \cite{Laa2}. Crucially for us, the universal functions $A$, $C_0$, $C_{ij}$ are \emph{the same} in the strictly semistable case and are therefore also determined by \eqref{eqn2}.
\end{remark}

The universal series $C_0, C_{ij}$ are elements of $\Q(\!(q^{\frac{1}{2r}})\!)$. Their leading terms can be easily calculated using \eqref{eqn:Upsilonzero}
\begin{align*}
&C_0 : \quad 1 \\
&C_{ii} : \quad \binom{r}{i}^{-1} q^{\frac{i(i-r)}{2r}} \\
&C_{ij} : \quad \frac{j(r-i)}{(j-i)r} q^{\frac{i(j-r)}{r}}, \quad \forall i< j.
\end{align*}
The leading term of $B$ is $q^{\frac{r}{24}}$.
After normalizing by these factors, $C_0,C_{ij}, B$ are elements of $1+q \, \Q[\![q]\!]$. We denote the normalized series by $\overline{C}_0$, $\overline{C}_{ij}, \overline{B}$ and similarly we denote $\overline{A}=r(-1)^{r-1}q^{\frac{r}{2}}A$.

\subsection{Localization} \label{sec:loc}

The expression for $\mathsf{G}_{S,\boldsymbol{a}}(q)$ in terms of integrals over products of Hilbert schemes $S^{[\boldsymbol{n}]}$ can be used to determine $C_0$, $C_{ij}$ for fixed $r$ and up to some order in $q$. 

By \eqref{eqn2}, the universal functions $A,B,E_i, E_{ij}$, and therefore $A$, $C_0$, $C_{ij}$, are determined by any choice of\footnote{Although $A$ is already known, including it allows us to perform useful consistency checks in our calculations.} 
$$
\Big\{(S^{(\alpha)}, \boldsymbol{a}^{(\alpha)}) \Big\}_{\alpha=1}^{\frac{1}{2}r(r+1)+1}
$$ 
such that
$$
\Big(\chi(\O_{S^{(\alpha)}}), (K_{S^{(\alpha)}})^2, \Big\{K_{S^{(\alpha)}} a_i^{(\alpha)} \Big\}_{i=1}^{r-1}, \Big\{a_i^{(\alpha)} a_j^{(\alpha)} \Big\}_{1 \leq i \leq j \leq r-1} \Big) \in \Q^{\frac{1}{2}r(r+1)+1}
$$
are $\Q$-independent. A particularly convenient basis is to take all $S^{(\alpha)}$ equal to toric surfaces and all $\boldsymbol{a}^{(\alpha)} \in H^2(S,\Z)^{r-1}$ equal to (Poincar\'e duals of) $T$-equivariant divisors.

Suppose $S$ is a smooth projective toric surface with torus $T$ and let $\boldsymbol{a} \in H^2(S,\Z)^{r-1}$ be (Poincar\'e dual to) $T$-equivariant divisors. Then the action of $T$ lifts to the Hilbert scheme $S^{[n]}$ and the product of Hilbert schemes $S^{[\boldsymbol{n}]}$. Moreover, the fixed point loci of $S^{[n]}$, $S^{[\boldsymbol{n}]}$ consist of isolated reduced points. For any maximal $T$-invariant affine open subset $U \subset S$ and $Z \in (S^{[n]})^T$, the ideal $I_Z|_U \subset \O_U \cong \C[x,y]$ is monomial with finite colength and corresponds to a 2-dimensional partition (Young diagram). Since the number of such maximal $T$-invariant affine open subsets is $e(S)$, the Euler characteristic of $S$, the elements of $(S^{[n]})^T$ correspond to sequences of 2-dimensional partitions $\{\lambda_i\}_{i=1}^{e(S)}$ of total size $n$. Similarly, elements of $(S^{[\boldsymbol{n}]})^T$ correspond to $r e(S)$-tuples of 2-dimensional partitions of total size $|\boldsymbol{n}|$. 

By the Atiyah-Bott localization formula \cite{AB}, we have
$$
 \int_{S^{[\boldsymbol{n}]}} \Upsilon(\boldsymbol{a},\boldsymbol{n},t) = \sum_{P \in (S^{[\boldsymbol{n}]})^T} \frac{ \Upsilon(\boldsymbol{a},\boldsymbol{n},t)|_P}{e(T_{S^{[\boldsymbol{n}]}}|_P)}.
$$
Interestingly, the integral \emph{simplifies} after applying the Atiyah-Bott localization formula:
\begin{align}
\begin{split} \label{eqn:GSafterAB}
\mathsf{G}_{S,\boldsymbol{a}}(q) &=  \sum_{\boldsymbol{n} \in \Z_{\geq 0}^{r}} q^{|\boldsymbol{n}|} \sum_{P \in (S^{[\bfn]})^T} e(-T^{[\bfn]}_0|_P), \\
T^{[\bfn]}_0 &:= T^{[\bfn]} - T^{[\boldsymbol{0}]}, \quad \mathbb{E}^{[\bfn]} := \bigoplus_{i=0}^{r-1} \I_i \boxtimes L_i \otimes \mathfrak{t}^{-i} \\
T^{[\bfn]} &:=R^\mdot \hom_\pi(\mathbb{E}^{[\bfn]}, \mathbb{E}^{[\bfn]} \otimes K_S \otimes \mathfrak{t})_0 - R^\mdot \hom_\pi(\mathbb{E}^{[\bfn]}, \mathbb{E}^{[\bfn]})_0, 
\end{split}
\end{align}
where $\pi : S \times S^{[\bfn]} \to S^{[\bfn]}$ denotes the projection.
This, in principle, reduces $\mathsf{G}_{S,\boldsymbol{a}}(q)$ to a purely combinatorial expression. The calculation of tautological integrals on Hilbert schemes of points on toric surfaces is a well-documented technique (for instance, see \cite{Laa1,GK5} for more details in this setting). 

Using an implementation into SAGE and Pari/GP, we determined the normalized series 
$$
\overline{A}, \overline{B}, \overline{C}_{ij} \in 1+q\Q[\![q]\!]
$$
for $r \leq 7$ in the following cases. For $r=2$ modulo $q^{15}$ and $r=3$ modulo $q^{11}$ (previously done in \cite{Laa3}). For $r=4,5,6,7$ modulo $q^{13}$. Recall that $C_0, B$ are related by \eqref{def:C0Cij} and we provided explicit expressions for the normalization terms in Section \ref{sec:cob}. The coefficients for $A$ indeed match \eqref{eqn:A} and this can be seen as a consistency check. In Appendix \ref{sec:data} we list our data for $r \leq 5$.

\section{Vertical conjectures} \label{sec:verconj}

For any rank $r>1$, consider the universal functions $C_0$, $C_{ij}$ for the vertical Vafa-Witten generating function (Theorem  \ref{thm:Laarakker}). The goal of this section is to use the expansions of $C_0$, $C_{ij}$, produced in the previous section, to find closed formulae for them. 

We introduce the following notation. For any, possibly empty, subset $I \subset [r-1]:=\{1, \ldots, r-1\}$, define
\begin{align*} 
C_{I} = C_0 \prod_{i \leq j \in I} C_{ij}.
\end{align*} 
In particular, $C_{\varnothing} = C_0$. Note that the universal functions $C_I$ determine the universal functions $C_0, C_{ij}$ (and vice versa). We also define
$$
|I| = \sum_{i \in I} 1, \quad |\!|I|\!| = \sum_{i \in I} i.
$$

\begin{remark}
We obtained the conjectures presented in this section using the following approach. Although $C_0$, $C_{ij}$ have rational coefficients, certain Laurent polynomials in $C_0$, $C_{ij}$ have integer coefficients for which we can guess closed formulae. We used the computer to find some of these Laurent polynomials in $C_0$, $C_{ij}$. Once we obtained sufficiently many closed formulae in this way, we solve for $C_0$, $C_{ij}$. Examples of such Laurent polynomials in $C_0, C_{ij}$ with integer coefficients  are $C_0^{-1}(1+C_{11}^{-1})$ (for $r=2$), $C_0^{-1}(1+C_{11}^{-1}+C_{22}^{-1}+C_{11}^{-1}C_{22}^{-1}C_{12}^{-1})$ and $C_{11}C_{22}^{-1}$ (for $r=3$), $2(C_{12}+C_{12}^{-1})$ and $4C_{12}C_{22}$ (for $r=4$), $25 C_{14} C_{23}$ and $5 C_{12} C_{13}^{-1}C_{23}$ (for $r=5$). Some of these relations are obtained from certain relations which we conjecture to hold in any rank (Conjecture \ref{conj:symm} below). The others are obtained by brute force ---i.e.~testing for many Laurent polynomial expressions, with the help of a computer, whether they yield integer coefficients.
\end{remark}

The conjectures presented in this section are always verified up to the orders for which we determined the universal functions $C_0, C_{ij}$ (as stated at the end of the previous section).

\subsection{Ranks 2--5} \label{sec:rk2rk3rk4rk5}

We make the following conjectures (which, for $S$ minimal of general type, imply Conjectures \ref{intro:conjver:rk2}--\ref{intro:conjver:rk5} from the introduction):
\begin{conjecture}[Vafa-Witten] \label{conjver:rk2}
For $r=2$, we have
$$
C_{\varnothing} = 1, \quad C_{\{1\}} = t_{A_1,1}.
$$
\end{conjecture}

\begin{conjecture}[G\"ottsche-Kool] \label{conjver:rk3}
For $r=3$, we have
$$
C_{\varnothing} = t_{A_2,1} X_-, \quad C_{\{1\}} = C_{\{2\}} = t_{A_2,1}, \quad C_{\{1,2\}} = t_{A_2,1} X_+,
$$
where $X_{\pm}$ are the solutions\footnote{Whenever we have a quadratic equation $aX^2 + bX + c = 0$, with $a \neq 0,$ then $X_{+}$ denotes the solution $\frac{1}{2a}(-b + \sqrt{b^2 - 4ac})$ and $X_-$ denotes the other solution.} of
$$
X^2 - 4 t_{A_2,1}^{2} X + 4 t_{A_2,1} = 0.
$$
\end{conjecture}
We note the following symmetry: permuting the roots $X_+ \leftrightarrow X_-$ corresponds to the operation
$$
C_I \leftrightarrow C_{\{1,2\} \setminus I}.
$$

Recall from the introduction that we denote Ramanujan's octic continued fraction by $u(q)$.
\begin{conjecture} \label{conjver:rk4}
For $r=4$, we have
\begin{align*}
C_{\varnothing} &= \frac{Z - Z^{-1}}{t_{A_3,2}^{-1} u(q^2)^{-4} - Z^{-1}}, \quad C_{\{1,3\}} = \frac{Z^{-1} - Z}{t_{A_3,2}^{-1} u(q^2)^{-4} - Z}, \\
C_{\{1\}} &= C_{\{3\}} = ( u(q^2)^4 + 1)t_{A_3,1}, \quad C_{\{1,2\}} = C_{\{2,3\}} = ( 1+u(q^2)^{-4} )t_{A_3,1}, \\
C_{\{2\}} &= \frac{Z - Z^{-1}}{t_{A_3,2}^{-1} Z - u(q^2)^{4}}, \quad C_{\{1,2,3\}} = \frac{Z^{-1} - Z}{t_{A_3,2}^{-1} Z^{-1} - u(q^2)^{4}},
\end{align*}
where $Z$ is a root\footnote{More precisely, $Z$ denotes the root $\frac{3}{2}(u(q^2)^{-2}+u(q^2)^{2} + \sqrt{ u(q^2)^{-4} + \frac{14}{9} + u(q^2)^{4} })$.} of
\begin{align*}
Z -3( u(q^2)^{-2} + u(q^2)^{2}) + Z^{-1} = 0.
\end{align*}
\end{conjecture}

Denoting the Rogers-Ramanujan continued fraction by $r(q)$, we define
\begin{align*}
\beta_1 =  \frac{1}{25} t_{A_4,1} \big( 3 r^{-5} + 2 -8 r^5 \big), \quad \beta_2 = \frac{1}{25} t_{A_4,2} \big( 8 r^{-5}  + 2  -3 r^{5} \big). 
\end{align*}
\begin{conjecture} \label{conjver:rk5}
For $r=5$, we have
\begin{align*}
&C_{\varnothing} = \Big(\frac{Z X_-^2 Y_-^2}{\beta_1\beta_2} \Big)^{\frac{1}{2}}, \quad C_{\{1\}} = C_{\{4\}} = X_-, \quad C_{\{2\}} = C_{\{3\}} = Y_-, \quad C_{\{1,2\}} = C_{\{3,4\}} = \beta_2, \\ 
&C_{\{1,3\}} = C_{\{2,4\}} = \beta_1, \quad C_{\{2,3\}} = \Big(\frac{X_+^2 Y_-^2}{Z \beta_1\beta_2} \Big)^{\frac{1}{2}}, \quad C_{\{1,4\}} = \Big( \frac{X_-^2 Y_+^2}{Z \beta_1\beta_2} \Big)^{\frac{1}{2}}, \\
&C_{\{1,2,3\}} = C_{\{2,3,4\}} = X_+, \quad C_{\{1,2,4\}} = C_{\{1,3,4\}} = Y_+, \quad C_{\{1,2,3,4\}} = \Big( \frac{Z X_+^2 Y_+^2}{\beta_1\beta_2} \Big)^{\frac{1}{2}},
\end{align*}
where $X_{\pm}$, $Y_{\pm}$ are, respectively, the solutions of
\begin{align*}
&X^2 - \frac{4}{5} \beta_1 \big(\beta_1 t_{A_4,1}^{-1}  - 1 \big) \big(3r^{-5}+1\big) X + \frac{4}{5}\beta_1^2 \big(3r^{-5} +1\big) = 0, \\
&Y^2 - \frac{4}{5} \beta_2  \big(\beta_2  t_{A_4,2}^{-1}  - 1 \big) \big(1-3r^{5} \big) Y + \frac{4}{5}\beta_2^2 \big(1-3r^{5} \big) = 0,
\end{align*}
and $Z$ is a root\footnote{More precisely, $Z = \frac{1}{25}(24 r^{-5}-39-24r^5+4 \sqrt{36 r^{-10} -117r^{-5} -16+117r^5+36r^{10}} )$.} of
$$
Z - \frac{6}{25} \big( 8r^{-5} -13 - 8r^5 \big) + Z^{-1} = 0.
$$
\end{conjecture}

As pointed out in the introduction, the last equation can also be written in terms of a Hauptmodul for $\Gamma_0(5)$ 
$$
Z - \Big(\frac{48}{25} j_5 +18\Big) + Z^{-1}=0, \quad  j_5 = \frac{\eta(q)^6}{\eta(q^5)^6}.
$$
Similar to the rank 3 case, we see that simultaneously permuting the roots $X_+ \leftrightarrow X_-$, $Y_+ \leftrightarrow Y_-$ corresponds to the operation
$$
C_I \leftrightarrow C_{\{1,2,3,4\} \setminus I}.
$$

It is easy to check that the universal functions $C_0, C_{ij}$ of the previous conjectures satisfy the following relations, which we conjecture to be true for any rank $r$.
\begin{conjecture} \label{conj:symm}
For any $r, \ell$, and $i \leq j$, we have  
\begin{align*}
C_{ij} = C_{r-j,r-i}, \quad \sum_{I \subset [r-1]} \epsilon_r^{\ell |\!|I|\!|} C_I^{-1} = \frac{\Theta_{A_{r-1}^{\vee},\ell}}{\Theta_{A_{r-1},0}}.
\end{align*}
\end{conjecture}
Without loss of generality, we can take $\ell = 0, \ldots, \lfloor  \frac{r}{2} \rfloor$. Indeed, we clearly do not get new equations when replacing $\ell$ by $\ell + k r$, and replacing $\ell$ by $-\ell$ results in conjugation of the equation. As we will discuss in Section \ref{sec:blowup}, the second formula follows from the (conjectural) blow-up formula for virtual Euler characteristics (via $S$-duality).

\subsection{Ranks 6 and 7} 

We present the following conjecture for the vertical part of the $\SU(6)$ Vafa-Witten partition function. 
\begin{conjecture}  \label{conjver:rk6}
For $r=6$, we have the following 16 equations\footnote{It is easy to solve for all but $C_0, C_{11}, C_{22}, C_{15}$. The equations for $\ell = 0,1,2,3$ then provide four (complicated) equations in $C_0, C_{11}, C_{22}, C_{15}$.} for the 16 universal functions $C_0, C_{ij}$
\begin{align*}
&C_{ij} = C_{6-j,6-i}, \quad \forall i \leq j, \quad \sum_{I \subset [5]} \epsilon_6^{\ell |\!|I|\!|} C_I^{-1} = \frac{\Theta_{A_{5}^{\vee},\ell}}{\Theta_{A_{5},0}}, \quad \ell=0,1,2,3 \\
&C_{24} = 4  t_{A_2,1}(q^2), \\
&C_{13}^{-1} C_{23}^{-1} C_{33}^{-1} = \frac{8\eta(q^3)^6\eta(q^{12})^{12}}{\eta(q^{6})^{18}}, \\
&C_{23}+C_{23}^{-1}=2j_6+14,\\
&C_{12}C_{14}+(C_{12}C_{14})^{-1}=\frac{50j_6^3 + 1206j_6^2 + 9504j_6 + 24192}{27(j_6+8)^2},\\
&C_{12}C_{14}^{-1}C_{24}=2j_6+16,\\
&C_{13}+C_{13}^{-1}=\frac{(5j_6^2+84j_6+360)}{4(j_6+9)^{\frac{3}{2}}},
\end{align*}
where $j_6 = \frac{\eta(q)^5\eta(q^3)}{\eta(q^2)\eta(q^6)^5}$ is a Hauptmodul for $\Gamma_0(6)$. 
\end{conjecture}

\begin{remark} \label{differentrks:mon}
The equation for $C_{24}$ in Conjecture \ref{conjver:rk6} is part of a pattern that persists in higher rank. Usually we suppress the dependence of $C_{ij}$ on the rank $r$. When we want to keep track of the rank, we denote by $C_{ij}^{(r)}(q)$ the universal functions of Theorem \ref{thm:Laarakker}. Then we conjecture
$$
C_{is,js}^{(rs)}(q) = C_{ij}^{(r)}(q^s),
$$
for all $r,s>1$, and $i<j$.\footnote{For this paper, we only calculated the universal functions $C_{ij}$, up to some order in $q$, for $r \leq 7$. However, the third-named author did some additional calculations for $r=8,9$, up to low order in $q$, which are compatible with this conjecture.} In Remark \ref{differentrks:inst}, we conjecture a similar property for the universal functions of \emph{any} set of multiplicative invariants of Gieseker-Maruyama moduli spaces. For Vafa-Witten invariants, the property in this remark and Remark \ref{differentrks:inst} are related by Conjecture \ref{intro:conjCD}. 
\end{remark}

Finally, we give the following conjecture for the vertical part of the $\SU(7)$ Vafa-Witten partition function. 
\begin{conjecture} \label{conjver:rk7}
For $r=7$, we have the following 16 relations for the 22 universal functions $C_0, C_{ij}$ 
\begin{align*}
&C_{ij} = C_{7-j,7-i}, \quad \forall i \leq j, \quad \sum_{I \subset [6]} \epsilon_7^{\ell |\!|I|\!|} C_I^{-1} = \frac{\Theta_{A_{6}^{\vee},\ell}}{\Theta_{A_{6},0}}, \quad \ell=0,1,2,3 \\
&C_{\{124\}} = \frac{1}{7^6} ( 8j_7+7^2)^2(j_7^2+7^2j_7+7^3), \\
&C_{16}C_{25}C_{34}=\frac{960}{7^3}j_7^2+\frac{192}{7}j_7+64,\\
&C_{12}C_{13}^{-1}C_{14}C_{23}^{-1}C_{15}^{-1}C_{24}=\frac{(8j_7+7^2)^3}{(8j_7+7^2)^3+28j_7^3}, 
\end{align*}
where $j_7 = \frac{\eta(q)^4}{\eta(q^7)^4}$ is a Hauptmodul of $\Gamma_0(7)$.
\end{conjecture}

\section{Horizontal conjectures} \label{sec:horconj}

Let $(S,H)$ be a smooth polarized surface with $b_1(S) = 0$ and $p_g(S)>0$. Fix $r,c_1,c_2$ such that there are no rank $r$ strictly $H$-semistable torsion free sheaves on $S$ with Chern classes $c_1,c_2$. Then the Gieseker-Maruyama moduli space $M:=M_S^H(r,c_1,c_2)$ is a projective scheme with perfect obstruction theory. As discussed in the introduction, the virtual tangent bundle and dimension satisfy
\begin{align*}
T_M^{\vir}|_{[\cE]} &\cong R^\mdot\Hom(\cE,\cE)_0[1], \\
 \vd &= \vd(r,c_1,c_2) := 2rc_2 - (r-1)c_1^2 - (r^2-1) \chi(\O_S).
\end{align*}
In Section \ref{sec:evir}, we provide a conjectural structure formula, in terms of universal functions (depending only on $r$), for the virtual Euler characteristics $e^{\vir}(M)$. At this level of generality, we can say little about the universal functions themselves. Motivated by $S$-duality, we use Conjectures \ref{conjver:rk2}--\ref{conjver:rk5} to find conjectural formulae for the universal functions for ranks $r=2$--$5$. The formula for $r=2$ first appeared in \cite{VW, DPS} (see also \cite{GK1}) and for $r=3$ in \cite{GK3}.

The following list of virtual invariants of $M$ were studied in the series of papers \cite{GK1}--\cite{GK5, GKW}:
\begin{itemize}
\item Virtual Euler characteristic, Hirzebruch $\chi_{-y}$-genus, elliptic genus: $e^{\vir}(M)$ ($r=2,3$), $\chi_{-y}^{\vir}(M)$ ($r=2,3$), $Ell^{\vir}(M)$ ($r=2$).
\item Vafa-Witten invariants with $\mu$-insertions (with H.~Nakajima): \\ $\int_{[M]^{\vir}} c(T_M^{\vir}) \exp(\mu(L) + \mu(\mathrm{pt}) \, u)$ ($r=2$).
\item $K$-theoretic Donaldson invariants (with R.~Williams): $\chi_{-y}^{\vir}(M, \mu(L))$ ($r=2$).
\item Virtual cobordism class: $\pi_* [M]_{\Omega_*}^{\vir}$ ($r=2$).
\item Virtual Segre numbers: $\int_{[M]^{\vir}} c(\alpha_M) \exp(\mu(L) + \mu(\mathrm{pt}) \, u)$ (any $r$).
\item Virtual Verlinde numbers: $\chi^{\vir}(M, \mu(L) \otimes E^r)$ (any $r$).
\end{itemize}
In each case, the invariants are governed by a conjectural structure formula, in terms of universal functions, similar to the case of virtual Euler characteristics. In Appendix \ref{sec:mult}, we generalize this universal structure formula further to \emph{any} collection of \emph{multiplicative} insertions on Gieseker-Maruyama moduli spaces covering all of these cases (Conjecture \ref{conj:multinst}).\footnote{Of course the explicit determination of the universal functions themselves is another matter and depends on the case.} We also prove a weak version of this general conjecture by using Mochizuki's formula (Theorem \ref{thm:multinst}). The reader who is only interested in Vafa-Witten theory can skip Appendix \ref{sec:mult}.

\subsection{Virtual Euler characteristics} \label{sec:evir}

Recall Conjecture \ref{intro:conjhor}:
\begin{conjecture}  \label{conjhor}
For any $r>1$, there exist $D_0$, $\{D_{ij}\}_{1 \leq i \leq j \leq r-1} \in \C[\![q^{\frac{1}{2r}}]\!]$ with the following property. For any smooth polarized surface $(S,H)$ satisfying $b_1(S) = 0$, $p_g(S)>0$, $c_1 \in  H^2(S,\Z)$, and $c_2 \in H^4(S,\Z)$ such that there are no rank $r$ strictly $H$-semistable sheaves on $S$ with Chern classes $c_1,c_2$, $e^{\mathrm{vir}}(M_S^H(r,c_1,c_2))$ equals the coefficient of $q^{c_2 - \frac{r-1}{2r} c_1^2 - \frac{r}{2} \chi(\O_S) + \frac{r}{24} K_S^2}$ of
\begin{align*}
&r^{2+K_S^2 - \chi(\O_S)} \Bigg( \frac{1}{\Delta(q^{\frac{1}{r}})^{\frac{1}{2}}} \Bigg)^{\chi(\O_S)} \Bigg( \frac{\Theta_{{A^\vee_{r-1}},0}(q)   }{\eta(q)^r} \Bigg)^{-K_S^2}  \\
&\quad\quad\quad\quad \times D_0(q)^{K_S^2} \sum_{\bfbeta \in H^2(S,\Z)^{r-1}}   \prod_{i} \epsilon_r^{i \beta_i c_1} \, \SW(\beta_i) \prod_{i \leq j} D_{ij}(q)^{\beta_i \beta_j}.
\end{align*}
\end{conjecture}
This is a special case of Conjecture \ref{conj:multinst}.\footnote{Note that $( \Theta_{{A^\vee_{r-1}},0}(q)   / \eta(q)^r)^{-K_S^2}$ is just a convenient normalization, which becomes useful in Section \ref{sec:blowup} when we discuss blow-up formulae.}  In order to see this, for $M:=M_S^H(r,c_1,c_2)$ as in Conjecture \ref{conjhor}, one first applies the virtual Poincar\'e-Hopf index theorem \cite{CFK,FG}
$$
e^{\vir}(M) = \int_{[M]^{\vir}} c_{\vd}(T_M^{\vir}),
$$
where $T_M^{\vir}$ is given by
$$
T_M^{\vir} = R^{\mdot}\hom_{\pi}(\EE,\EE)_0[1].
$$
Here $\EE$ is a universal sheaf on $S \times M$ and $\pi : S \times M \to M$ denotes projection.\footnote{Although $\E$ may only exist \'etale locally, $R^{\mdot}\hom_{\pi}(\E,\E)_0$ exists globally \cite[Thm.~2.2.4]{Cal}, see also \cite[Sect.~10.2]{HL}.}  
Denote by $\cM:=\cM_S^H(r,c_1,c_2)$ the Deligne-Mumford stack of rank $r$ oriented $H$-stable sheaves on $S$ with Chern classes $c_1,c_2$ (Appendix \ref{sec:mult}). Using the degree $\tfrac{1}{r} : 1$ \'etale morphism $\rho : \cM \to M$, we have
$$
T_{\cM}^{\vir} = \rho^* T_M^{\vir}, \quad \rho_*[\cM]^{\vir} = \frac{1}{r} [M]^{\vir},
$$
and therefore $e^{\vir}(M) = r e^{\vir}(\cM)$ and we can apply Conjecture \ref{conj:multinst}. In order to deduce Conjecture \ref{conjhor}, it remains to show that the universal function corresponding to $\chi(\O_S)$ is $\Delta(q^{\frac{1}{r}})^{-\frac{1}{2}}$.

If $S$ is a K3 surface, then $M$ is smooth of expected dimension and deformation equivalent to a Hilbert scheme of points of the same dimension \cite{OG,Huy1,Yos2}, i.e.~$S^{[\vd/2]}$. Since Euler characteristics of Hilbert schemes of points are well-known \cite{Got1}, we obtain that $e(M)$ equals the coefficient of $q^{c_2 - \frac{r-1}{2r} c_1^2 - r}$ of
\begin{align*}
\Delta(q^{\frac{1}{r}})^{-1}.
\end{align*}
Hence the universal function corresponding to $\chi(\O_S)$ is $\Delta(q^{\frac{1}{r}})^{-\frac{1}{2}}$.\footnote{One requires a small argument in order to show that every coefficient of $\Delta(q^{\frac{1}{r}})^{-\frac{1}{2}}$ indeed arises as some $e(M)$. This is readily established using \cite[Thm.~2.7]{Huy2}.}

The main conjecture of this section is the following (Conjecture \ref{intro:conjCD} in the introduction):
\begin{conjecture} \label{conjCD}
For any $r>1$, $C_0(q)$, $C_{ij}(q)$, $D_0(q)$, $D_{ij}(q)$ are Fourier expansions in $q = e^{2 \pi i \tau}$ of meromorphic functions $C_0(\tau)$, $C_{ij}(\tau)$, $D_0(\tau)$, $D_{ij}(\tau)$ on $\mathfrak{H}$ satisfying
$$
D_0(\tau) = C_0(-1/\tau), \quad D_{ij}(\tau) = C_{ij}(-1/\tau).
$$
\end{conjecture}

Combining with the conjectural formulae for $C_0, C_{ij}$ from Section \ref{sec:rk2rk3rk4rk5} provides expressions for $D_0, D_{ij}$, as we we will discuss in detail in the next section.

The main motivation for this conjecture is the $S$-duality transformation (Conjecture \ref{intro:Sdualconj}). Roughly speaking, the $S$-duality transformation swaps the vertical generating series with (part of) the horizontal generating series. We will discuss this in detail in Section \ref{sec:S-dual}.

In \cite{GK1, GK3}, the first- and second-named author conjectured formulae for $D_0, D_{ij}$ for $r=2,3$ and verified that the resulting generating series of virtual Euler characteristics are correct for a list of surfaces up to certain virtual dimensions. The calculation in loc.~cit.~used Mochizuki's formula for descendent Donaldson invariants \cite{Moc}. The conjectures in \cite{GK1, GK3} are consistent with Conjecture \ref{conjCD}. For $r=4,5$, our computer programs are not powerful enough to effectively calculate with Mochizuki's formula. 

In the next section, we perform the following two additional consistency checks on Conjecture \ref{conjCD}:
\begin{itemize}
\item Although the coefficients of $C_0, C_{ij}$ are rational, the coefficients of $D_0, D_{ij}$ generally lie in a Galois extension of $\Q$. Let $S$ be a smooth projective surface with $H_1(S,\Z) = 0$, $p_g(S)>0$, and such that the only Seiberg-Witten basic classes are $0, K_S \neq 0$ (e.g.~$S$ is minimal of general type \cite[Thm.~7.4.1]{Mor}). Then we prove that for $r=2$--$5$ and any $c_1 \in H^2(S,\Z)$, the coefficients of 
$$
\Psi_{r,S,c_1}(q) = D_0(q)^{K_S^2} \sum_{\bfbeta \in H^2(S,\Z)^{r-1}} \prod_{i} \epsilon_r^{i \beta_i c_1} \, \SW(\beta_i) \prod_{i \leq j} D_{ij}(q)^{\beta_i \beta_j}
$$
are rational. Since virtual Euler characteristics are integers, one also expects from Conjecture \ref{conjhor} that $$r^{2 + K_S^2 - \chi(\O_S)}\Psi_{r,S,c_1}(q)$$ has integer coefficients. This appears to be a non-trivial fact that we can only check in examples.
\item For $r=2$--$5$, the leading term of  $\sfZ_{S,H,c_1}^{(r)}$ produces certain expressions for Donaldson invariants without insertions (Proposition \ref{prop:Don}). These expressions are consistent with \cite{GK4} ($r=2$--$4$) and \cite{Got4} ($r=5$), where these expressions were obtained by completely different methods.
\end{itemize}

\subsection{Galois actions and Donaldson invariants}

Let $(S,H)$ be a smooth polarized surface with $H_1(S,\Z) = 0$ and $p_g(S)>0$. Applying Conjecture \ref{conjCD} to the conjectural formulae for $C_0, C_{ij}$ from Section \ref{sec:rk2rk3rk4rk5} amounts to the following replacements (\eqref{maintrans} in Appendix \ref{sec:theta} and \cite[Prop.~2]{Duk})
\begin{align}
\begin{split} \label{def:s}
t_{A_{r-1},\ell}(q)|_{\tau \mapsto -\frac{1}{\tau}} &= t_{A_{r-1}^\vee,\ell}(q), \\
u(q^2)|_{\tau \mapsto -\frac{1}{\tau}} &=  \frac{\eta(q^{\frac{1}{2}}) \eta(q^\frac{1}{8})^2 }{\eta(q^\frac{1}{4})^3} =: v(q), \\
r(q)|_{\tau \mapsto -\frac{1}{\tau}} &= \frac{1 - \varphi \, r(q)}{\varphi + r(q)} =: s(q),
\end{split}
\end{align}
where $q = \exp(2 \pi i \tau)$, $u(q), r(q)$ are the octic continued fraction and Rogers-Ramanujan continued fraction, and $\varphi = \frac{1+ \sqrt{5}}{2}$ denotes the golden ratio. In particular, we deduce Proposition \ref{intro:hor:rk2-5} from the introduction. 

\begin{remark}
The conjectural formulae for $C_0, C_{ij}, D_0, D_{ij}$ thus obtained, together with Theorem \ref{thm:Laarakker} and Conjecture \ref{conjhor}, therefore completely determine the horizontal and vertical contribution to
$$
\sfZ_{S,H}^{\SU(r)}(q), \quad r=2,3,4,5
$$
for any smooth polarized surface $(S,H)$ with $H_1(S,\Z) = 0$ and $p_g(S)>0$. Recall that for $r=2,3,5$ prime, there are only horizontal and vertical contributions (Thomas vanishing, Section \ref{sec:TT}) and we therefore (conjecturally) know the full $\SU(r)$ Vafa-Witten partition function.  
\end{remark}

Next, we suppose the only Seiberg-Witten basic classes of $S$ are $0$ and $K_S \neq 0$ (e.g.~$S$ is minimal of general type), so the corresponding Seiberg-Witten invariants are $1$ and $(-1)^{\chi(\O_S)}$ \cite[Thm.~7.4.1]{Mor}. Consider the expression $\Psi_{r,S,c_1}(q)$ (Definition \ref{def:Psi}).  For $r=2$, its coefficients are manifestly rational. For $r=3$, we have
$$
\Psi_{3,S,c_1} = t_{A_2^{\vee},1}^{K_S^2}  \, \big(X_{+}^{K_S^2} + X_{-}^{K_S^2} \big) + \big( \epsilon_3^{K_S c_1} + \epsilon_3^{-K_S c_1} \big) \, (-1)^{\chi(\O_S)} \, t_{A_2^{\vee},1}^{K_S^2}, 
$$
where $X_{\pm}$ are the roots of the quadratic equation
$$
X^2 - 4  t_{A_2^{\vee},1}^{2} \, X + 4 t_{A_2^{\vee},1} = 0.
$$
Note that $t_{A_2^\vee,1}$ has coefficients in $\Q(\epsilon_3)$. Writing this quadratic equation as $aX^2+bX+c=0$, the leading term of the discriminant $D = b^2 - 4ac$ is $432 = 2^4 3^3$. Therefore 
$$
\Psi_{3,S,c_1} \in \Q(\epsilon_3,\sqrt{3})[[q^{\frac{1}{6}}]] = \Q(\sqrt{3},i)[[q^{\frac{1}{6}}]].
$$

For $r=4$, we have
\begin{align*}
\Psi_{4,S,c_1} = &\ \Big\{ \Big( \frac{Z - Z^{-1}}{t_{A_3^\vee,2}^{-1} v^{-4} - Z^{-1}} \Big)^{K_S^2} + \Big( \frac{Z^{-1} - Z}{t_{A_3^\vee,2}^{-1} v^{-4} - Z} \Big)^{K_S^2} \Big\} \\
&+ (-1)^{\chi(\O_S) + K_S c_1}  \Big\{ \Big( \frac{Z - Z^{-1}}{t_{A_3^\vee,2}^{-1} Z - v^{4}} \Big)^{K_S^2} + \Big( \frac{Z^{-1} - Z}{t_{A^\vee_3,2}^{-1} Z^{-1} - v^{4}} \Big)^{K_S^2} \Big\} \\
&+\big(i^{K_S c_1} + i^{-K_S c_1}\big) \, t_{A_3^\vee,1}^{K_S^2} \Big\{\big( 1 + v^{-4} \big)^{K_S^2} + (-1)^{\chi(\O_S)}\big( v^4+1 \big)^{K_S^2} \Big\},
\end{align*}
where $Z$ is a solution of the equation
$$
Z -3( v^{-2} + v^{2}) + Z^{-1} = 0.
$$
The leading term of the discriminant is $32 = 2^5$. Therefore 
$$
\Psi_{4,S,c_1} \in \Q(\sqrt{2},i)[[q^{\frac{1}{8}}]].
$$

For $r=5$, we define auxiliary functions
\begin{align*}
\gamma_1 :=  \frac{1}{25} t_{A_4^{\vee},1} \big( 3 s^{-5} + 2 -8 s^5 \big), \quad \gamma_2 := \frac{1}{25} t_{A_4^{\vee},2} \big( 8 s^{-5}  + 2  -3 s^{5} \big). 
\end{align*}
Then
\begin{align*}
\Psi_{5,S,c_1} = & \Bigg\{ \Bigg( \frac{Z X_+^2Y_+^2}{\gamma_1\gamma_2}   \Bigg)^{\frac{K_S^2}{2}} + \Bigg( \frac{X_+^2Y_-^2}{Z\gamma_1\gamma_2}   \Bigg)^{\frac{K_S^2}{2}} + \Bigg( \frac{X_-^2Y_+^2}{Z\gamma_1\gamma_2}   \Bigg)^{\frac{K_S^2}{2}} + \Bigg(  \frac{Z  X_-^2Y_-^2}{\gamma_1\gamma_2}   \Bigg)^{\frac{K_S^2}{2}} \Bigg\} \\
&+\big( \epsilon_5^{K_S c_1} + \epsilon_5^{-K_S c_1} \big) \Big\{ \gamma_1^{K_S^2} + (-1)^{\chi(\O_S)}\big(X_+^{K_S^2} + X_-^{K_S^2}\big) \Big\}  \\
&+ \big( \epsilon_5^{2K_S c_1} + \epsilon_5^{-2K_S c_1} \big) \Big\{ \gamma_2^{K_S^2} + (-1)^{\chi(\O_S)}\big(Y_+^{K_S^2} + Y_-^{K_S^2}\big) \Big\}, 
\end{align*}
where $X_{\pm}, Y_{\pm}, Z$ are, respectively, the solutions of the equations
\begin{align}
\begin{split} \label{quadrXYZ}
&X^2 - \frac{4}{5} \gamma_1 \big(\gamma_1 t_{A_4^\vee,1}^{-1}  - 1 \big) \big(3s^{-5}+1\big) X + \frac{4}{5}\gamma_1^2 \big(3s^{-5} +1\big) = 0, \\
&Y^2 - \frac{4}{5} \gamma_2  \big(\gamma_2  t_{A_4^\vee,2}^{-1}  - 1 \big) \big(1-3s^{5} \big) Y + \frac{4}{5}\gamma_2^2 \big(1-3s^{5} \big) = 0,  \\
&Z - \frac{6}{25} \big( 8s^{-5} -13 - 8s^5 \big) + Z^{-1} = 0. 
\end{split}
\end{align}
Multiplying the third equation by $Z$, the leading terms of the discriminant of the quadratic equations are respectively given by
$$
2000(5 + 2 \sqrt{5}), \quad 2000(5- 2 \sqrt{5}), \quad 320,
$$
where $2000 = 2^4 5^3$ and $320 = 2^6 5$. We note the identities
\begin{align}
\begin{split} \label{sqrtsqrt}
&\sqrt{5} = -2 (\epsilon_5 - \epsilon_5^{-1})^2 - 5, \\
&\sqrt{5 + 2 \sqrt{5}} = -\frac{i}{2} (\epsilon_5 - \epsilon_5^{-1})(1  + \sqrt{5}), \\
&\sqrt{5 - 2 \sqrt{5}} = \frac{i}{2} (\epsilon_5^2 - \epsilon_5^{-2})(1  - \sqrt{5}),
\end{split}
\end{align}
which are all elements of $\Q(\epsilon_5,i)$. Furthermore, the coefficients of the leading terms of $Z X_+^2Y_+^2 / (\gamma_1\gamma_2)$, $X_+^2Y_-^2 / (Z \gamma_1\gamma_2)$, $X_-^2Y_+^2 / (Z \gamma_1\gamma_2)$, $Z X_-^2Y_-^2 / (\gamma_1\gamma_2)$ are squares in $\Q(\sqrt{5}) \subset \Q(\epsilon_5,i)$. Therefore
$$
\Psi_{5,S,c_1} \in \Q(\epsilon_5,i)[[q^{\frac{1}{10}}]].
$$
\begin{proposition} \label{prop:Gal}
Let $S$ be a smooth projective surface with $H_1(S,\Z) = 0$, $p_g(S)>0$, and such that the only Seiberg-Witten basic classes are $0, K_S \neq 0$. Then for any $c_1 \in H^2(S,\Z)$ we have
\begin{itemize}
\item $\Psi_{3,S,c_1}$ is invariant under the action of $\mathrm{Gal}(\Q(\sqrt{3},i)/\Q) \cong \Z_2 \times \Z_2$.
\item $\Psi_{4,S,c_1}$ is invariant under the action of $\mathrm{Gal}(\Q(\sqrt{2},i)/\Q) \cong \Z_2 \times \Z_2$.
\item $\Psi_{5,S,c_1}$ is invariant under the action of $\mathrm{Gal}(\Q(\epsilon_5,i)/\Q) \cong \Z_4 \times \Z_2$.
\end{itemize}
In particular, $\Psi_{3,S,c_1}, \Psi_{4,S,c_1}, \Psi_{5,S,c_1}$ have rational coefficients.
\end{proposition}
\begin{proof}
The Galois group $\mathrm{Gal}(\Q(\sqrt{3},i)/\Q)$ is generated by
\begin{align*}
&\sigma : \sqrt{3} \mapsto -\sqrt{3}, \quad i \mapsto i, \\
&\tau : \sqrt{3} \mapsto \sqrt{3}, \quad i \mapsto -i.
\end{align*}
We denote the action of an element $\nu \in \mathrm{Gal}(\Q(\sqrt{3},i)/\Q)$ on the coefficients of a formal Laurent series $F \in \Q(\sqrt{3},i)(\!(x)\!)$ by $\nu(F)$. Then
\begin{align*}
\sigma(t_{A_2^\vee,1}) = t_{A_2^\vee,1}, \quad \tau(t_{A_2^\vee,1}) = t_{A_2^\vee,1}, 
\end{align*}
where we used $t_{A_2^\vee,-1} = t_{A_2^\vee,1}$ (\eqref{eqn:Adual:ltominl} in Appendix \ref{sec:theta}). Therefore
$$
\sigma(X_+) = X_-, \quad \sigma(X_-) = X_+, \quad \tau(X_+) = X_+, \quad \tau(X_-) = X_-,
$$
and, consequently, $\sigma(\Psi_{3,S,c_1}) = \tau(\Psi_{3,S,c_1}) = \Psi_{3,S,c_1}$.

The Galois group $\mathrm{Gal}(\Q(\sqrt{2},i)/\Q)$ is generated by
\begin{align*}
&\sigma : \sqrt{2} \mapsto -\sqrt{2}, \quad i \mapsto i, \\
&\tau : \sqrt{2} \mapsto \sqrt{2}, \quad i \mapsto -i.
\end{align*}
Therefore
\begin{align*}
\sigma(t_{A_3^\vee,\ell}) = t_{A_3^\vee,\ell}, \quad \tau(t_{A_3^\vee,\ell}) = t_{A_3^\vee,\ell}, 
\end{align*}
where we used $t_{A_3^\vee,-\ell} = t_{A_3^\vee,\ell}$ (\eqref{eqn:Adual:ltominl} in Appendix \ref{sec:theta}).
Since $v(q)$ has rational coefficients, we see that $\sigma, \tau$ leave the quadratic equation for $Z$ invariant. However, $\sigma$ acts non-trivially on $Z$ due to the leading term $2^5$ of the discriminant. We obtain
$$
\sigma(Z) = Z^{-1}, \quad \sigma(Z^{-1}) = Z, \quad \tau(Z) = Z, \quad \tau(Z^{-1}) = Z^{-1}
$$
and, hence, $\sigma(\Psi_{4,S,c_1}) = \tau(\Psi_{4,S,c_1}) = \Psi_{4,S,c_1}$.

The Galois group $\mathrm{Gal}(\Q(\epsilon_5,i)/\Q)$ is generated by
\begin{align*}
&\sigma : \epsilon_5 \mapsto \epsilon_5^2, \quad i \mapsto i, \\
&\tau : \epsilon_5 \mapsto \epsilon_5, \quad i \mapsto -i.
\end{align*}
These elements have order 4 and 2 respectively and satisfy the relation $\tau \sigma = \sigma \tau$. We find
\begin{align*}
&\tau(t_{A_4^\vee,1}) = t_{A_4^\vee,1}, \quad \tau(t_{A_4^\vee,2}) = t_{A_4^\vee,2}, \quad \tau(s) = s, \\
&\sigma(t_{A_4^\vee,1}) = t_{A_4^\vee,2}, \quad \sigma(t_{A_4^\vee,2}) = t_{A_4^\vee,1}, \quad \sigma(s) = - s^{-1}, 
\end{align*}
where we used $t_{A_4^\vee,-\ell} = t_{A_4^\vee,\ell}$ (\eqref{eqn:Adual:ltominl} in Appendix \ref{sec:theta}) and $\sigma(\sqrt{5}) = -\sqrt{5}$ (equation \eqref{sqrtsqrt}). Therefore, the quadratic equations  \eqref{quadrXYZ} for $X_{\pm}, Y_{\pm}$ get swapped under the action of $\sigma$ and the quadratic equation for $Z$ is invariant under the action of $\sigma$. All three quadratic equations are invariant under the action of $\tau$. Using \eqref{sqrtsqrt}, we also find that the action of $\tau, \sigma$ on the leading terms of the discriminants of the quadratic equations \eqref{quadrXYZ} is given by 
\begin{align*}
\tau(\sqrt{5+2\sqrt{5}}) = - \sqrt{5+2\sqrt{5}}, \quad \tau(\sqrt{5-2\sqrt{5}}) = -\sqrt{5-2\sqrt{5}}, \quad \tau(\sqrt{5}) = \sqrt{5}, \\
\sigma(\sqrt{5+2\sqrt{5}}) = - \sqrt{5-2\sqrt{5}}, \quad \sigma(\sqrt{5-2\sqrt{5}}) = \sqrt{5+2\sqrt{5}}, \quad \sigma(\sqrt{5}) = -\sqrt{5}.
\end{align*}
Hence 
\begin{align*}
&\tau(X_+) = X_-, \quad \tau(X_-) = X_+, \quad \tau(Y_+) = Y_-, \quad \tau(Y_-) = Y_+, \quad \tau(Z) = Z, \\
&\sigma(X_+) = Y_-, \quad \sigma(X_-) = Y_+, \quad \sigma(Y_+) = X_+, \quad \sigma(Y_-) = X_-, \quad \sigma(Z) = Z^{-1}.
\end{align*}
It then follows that $\sigma(\Psi_{5,S,c_1}) = \tau(\Psi_{5,S,c_1}) = \Psi_{5,S,c_1}$. 
\end{proof}
The summands of $\Psi_{r,S,c_1}$ are called quantum vacua by physicists. For $r>2$, we see in the proof of the previous proposition that there is an interesting action of the Galois group on the quantum vacua.

\begin{remark}
For any smooth polarized surface $(S,H)$ with $H_1(S,\Z)=0$ and $p_g(S)>0$, and $c_1 \in H^2(S,\Z)$ such that there are no rank $r$ strictly $H$-semistable sheaves on $S$ with first Chern class $c_1$, we have that $\mathsf{Z}_{S,H,c_1}^{(r)}(q)$ is the generating series of virtual Euler characteristics $e^{\vir}(M_S^H(r,c_1,c_2)) \in \Z$. Therefore, we expect that the coefficients of 
$$
r^{2+K_S^2 - \chi(\O_S)} \Psi_{r,S,c_1}(q)
$$ 
are \emph{integers}. This appears to be a non-trivial fact which we cannot prove in general. 
If moreover the only Seiberg-Witten basic classes are $0, K_S \neq 0$ and $r = 3,4,5$, we checked integrality numerically for a list of values of $\chi(\O_S), K_S^2, K_S c_1 \in \Z$, satisfying $\chi(\O_S) \geq 2$ and $K_S^2 \geq \chi(\O_S) - 3$, up to low order in $q$. Note that $\chi(\O_S) \geq 2$ since $H_1(S,\Z) = 0$ and $p_g(S)>0$, and the inequality $K_S^2 \geq \chi(\O_S) - 3$ is satisfied by the proof of \cite[Prop.~5.3]{GK1}.\footnote{For $S$ minimal of general type, it is implied by Noether's inequality $K_S^2 \geq 2(\chi(\O_S) - 3)$.} Interestingly, we find that integrality can \emph{fail} for values of $\chi(\O_S), K_S^2, K_S c_1 \in \Z$ which do not satisfy these inequalities. Therefore, integrality is tied to the geometry. 
\end{remark}

We end this section by giving the leading term of the generating series $\mathsf{Z}_{S,H,c_1}^{(r)}(q)$ for $r =2,3,4,5$, which are Donaldson invariants without insertions. Below, we use the notation $\widetilde{\beta} = 2\beta - K_S$ for any $\beta \in H^2(S,\Z)$.
\begin{proposition} \label{prop:Don}
Assume Conjecture \ref{conjhor} holds. Let $(S,H)$ be a smooth polarized surface with $H_1(S,\Z)=0$ and $p_g(S)>0$. Suppose $c_1 \in H^2(S,\Z)$, $c_2 \in H^4(S,\Z)$ such that there are no rank $r$ strictly $H$-semistable sheaves on $S$ with Chern classes $c_1, c_2$ and
$\vd(r,c_1,c_2) = 0$. Then 
$$
\int_{[M_S^H(r,c_1,c_2)]^{\vir}} 1 = r^{2+K_S^2 - \chi(\O_S)} D_0(0)^{K_S^2} \sum_{\bfbeta \in H^2(S,\Z)^{r-1}}   \prod_{i} \epsilon_r^{i \beta_i c_1} \, \SW(\beta_i) \prod_{i \leq j} D_{ij}(0)^{\beta_i \beta_j}.
$$
Furthermore, assuming the conjectural formulae for $D_0, D_{ij}$, obtained from Conjecture \ref{conjCD} and \ref{conjver:rk2}--\ref{conjver:rk5}, the right hand side reduces to (respectively)
\begin{align*}
&2^{2-\chi(\O_S)+K_S^2} \sum_{\beta \in H^2(S,\Z)} \SW(\beta) \, (-1)^{\beta c_1}, \\
&3^{2-\chi(\O_S)+K_S^2} \sum_{(\beta_1,\beta_2) \in H^2(S,\Z)^2} \SW(\beta_1) \, \SW(\beta_2) \, \epsilon_3^{(\beta_1+2\beta_2)c_1} \, 2^{\frac{1}{4}(\widetilde \beta_1+\widetilde \beta_2)^2}, \\
&4^{2-\chi(\O_S)+K_S^2} \sum_{(\beta_1,\beta_2,\beta_3) \in H^2(S,\Z)^3} \SW(\beta_1) \, \SW(\beta_2) \, \SW(\beta_3) \, i^{(-\beta_1+2\beta_2-3\beta_3)c_1}  \\ 
&\quad \quad \quad \quad \times 2^{K_S^2+ \frac{1}{4}(\widetilde \beta_1 + \widetilde \beta_2)(\widetilde \beta_1+\widetilde \beta_3)} \,  \Big(1-\frac{1}{2} \sqrt{2} \Big)^{\frac{1}{2}\widetilde \beta_2(\widetilde \beta_1 +\widetilde \beta_3)}, \\
&5^{2-\chi(\O_S)+K_S^2} \sum_{(\beta_1,\beta_2,\beta_3,\beta_4) \in H^2(S,\Z)^4} \SW(\beta_1) \, \SW(\beta_2) \, \SW(\beta_3) \, \SW(\beta_4) \epsilon_5^{\sum_i i \beta_i c_1} (20+8 \sqrt{5})^{K_S^2} \\
&\quad \quad \quad \quad \times 2^{\frac{1}{2}\widetilde{\beta}_1\widetilde{\beta}_4 + \frac{1}{2}\widetilde{\beta}_2\widetilde{\beta}_3 - \frac{1}{4}\widetilde{\beta}_1^2- \frac{1}{4}\widetilde{\beta}_2^2- \frac{1}{4}\widetilde{\beta}_3^2- \frac{1}{4}\widetilde{\beta}_4^2} \Big(\frac{3}{2} - \frac{1}{2} \sqrt{5} \Big)^{\frac{1}{4}\widetilde{\beta}_1^2 + \frac{1}{4} \widetilde{\beta}_4^2 + \frac{1}{4}\widetilde{\beta}_1\widetilde{\beta}_4 - \frac{1}{4}\widetilde{\beta}_1\widetilde{\beta}_3 - \frac{1}{4}\widetilde{\beta}_2\widetilde{\beta}_3 - \frac{1}{4}\widetilde{\beta}_2\widetilde{\beta}_4  }  \\
&\quad \quad \quad \quad  \times\Big(\frac{7}{2} - \frac{3}{2} \sqrt{5}\Big)^{\frac{1}{4}\widetilde{\beta}_2^2 + \frac{1}{4} \widetilde{\beta}_3^2 - \frac{1}{4} \widetilde{\beta}_1 \widetilde{\beta}_2 - \frac{1}{4} \widetilde{\beta}_3 \widetilde{\beta}_4}.
\end{align*}
\end{proposition}
\begin{proof}
This follows from determining the leading terms of the universal series $D_0(q), D_{ij}(q)$, and some minor rewriting using the relations $\beta^2 = \beta K_S$ and $\SW(K_S - \beta) = (-1)^{\chi(\O_S)} \SW(\beta)$ for any Seiberg-Witten basic class $\beta \in H^2(S,\Z)$ \cite[Prop.~6.3.1, 6.3.4]{Moc}. \\

\noindent For $r=2$, we have 
$$
D_0(0) = D_{11}(0)=1. 
$$

\noindent For $r=3$, we have 
$$
D_0(0) = 2, \quad D_{11}(0) = D_{22}(0) = \frac{1}{2}, \quad D_{12}(0) = 4.
$$ 

\noindent For $r=4$, we have 
\begin{align*}
&D_0(0) = 4 + 2 \sqrt{2}, \quad D_{11}(0) = D_{33}(0) = 1 - \frac{1}{2} \sqrt{2}, \quad D_{22}(0) = 3 - 2 \sqrt{2}, \\
&D_{12}(0) = D_{23}(0) = 3 + 2 \sqrt{2}, \quad D_{13}(0) = 2.
\end{align*}

\noindent For $r=5$, we have
\begin{align*}
&D_0(0) = 20 + 8 \sqrt{5}, \quad D_{11}(0)=D_{44}(0)=\frac{3}{4} - \frac{1}{4} \sqrt{5}, \\
&D_{22}(0) = D_{33}(0) = \frac{7}{4} - \frac{3}{4} \sqrt{5},\quad D_{12}(0) = D_{34}(0) = \frac{7}{2} + \frac{3}{2} \sqrt{5}, \\
&D_{13}(0) = D_{24}(0) = \frac{3}{2} + \frac{1}{2} \sqrt{5}, \quad D_{23}(0) = 6 + 2 \sqrt{5}, \quad D_{14}(0) = 6 - 2 \sqrt{5}. \qedhere
\end{align*}
\end{proof}

\begin{remark}
For $r=2$, the formula of this proposition is a special case of the Witten conjecture. It was proved for smooth projective surfaces satisfying $b_1(S) = 0$ and $p_g(S)>0$ in \cite{GNY3}, and (under a technical assumption) for all differentiable 4-manifold $M$ with $b_1(M) = 0$, odd $b^{+}(M)$, and Seiberg-Witten simple type in \cite{FL1, FL2}. The higher rank analogue of the Witten conjecture is the Mari\~{n}o-Moore conjecture \cite{MM1, LM}. An algebro-geometric version of the Mari\~{n}o-Moore conjecture was given in \cite{GK4} and the universal constants in the conjecture were explicitly given for $r=3,4$. The conjectures in \cite{GK4} were proved in several examples using Mochizuki's formula. Recently, the first-named author determined the universal constants for any rank $r$ \cite{Got4} (see also \cite{LM}). The motivation in \cite{Got4} comes from conjectural blow-up formulae for virtual Segre and Verlinde numbers. In the present paper, Proposition \ref{prop:Don} is a consequence of $S$-duality. This provides an interesting consistency check between the conjectures of \cite{MM1, LM, GK4, Got4} on the one hand, and the conjectures of this paper on the other.
\end{remark}

\subsection{$S$-duality} \label{sec:S-dual}

Let $(S,H)$ be a smooth polarized surface with $H_1(S,\Z) = 0$ and $p_g(S)>0$. Let $r>2$ be \emph{prime} and $c_1 \in H^2(S,\Z)$ algebraic such that there are no rank $r$ strictly $H$-semistable Higgs pairs $(\cE,\phi)$ on $S$ with $c_1(\cE) = c_1$. As discussed in the introduction, Thomas vanishing \cite[Thm.~5.23]{Tho} implies that the only contributions to $\sfZ^{\SU(r)}_{S,H,c_1}$ are for ``open and closed'' components (unions of connected components) indexed by $\mu = (1^r)$ (vertical) and $\mu = (r)$ (horizontal). Moreover, the universal structure of $\mathsf{Z}^{(1^r)}_{S,H,c_1}$ is given by Theorem \ref{thm:Laarakker} and of $\mathsf{Z}^{(r)}_{S,H,c_1}$ by Conjecture \ref{conjhor}. 

\begin{definition} \label{def:smallpsi}
For any $r>1$, denote the formula of Conjecture \ref{conjhor} by 
$$
q^{-\frac{\chi(\O_S)}{2r} + \frac{r K_S^2}{24}} \psi_{S,c_1}(q^{\frac{1}{2r}}).
$$
\end{definition}

The following is elementary (e.g.~proof of \cite[Cor.~1.2]{GK3})
\begin{equation} \label{Zandpsi}
\sfZ^{(r)}_{S,H,c_1}(q) = r^{-1} q^{-\frac{\chi(\O_S)}{2r} + \frac{r K_S^2}{24}} \sum_{k=0}^{r} \epsilon_r^{k((r-1)c_1^2 + (r^2-1) \chi(\O_S))} \psi_{S,c_1}(\epsilon_r^k q^{\frac{1}{2r}}).
\end{equation}
We therefore have a complete conjectural structure formula for the $\SU(r)$ Vafa-Witten partition function. Using the definition of $\sfZ^{\SU(r)}_{S,H,c_1}$ in terms of Joyce-Song pairs in the strictly semistable case \cite{TT2}, we expect that the same expression (with the same universal functions!) holds in the strictly semistable case. This has been proved for the vertical contribution $\sfZ^{(1^r)}_{S,H,c_1}$ by the third-named author in \cite{Laa2}. Altogether, we have
\begin{align}
\begin{split} \label{VWfull}
\sfZ_{S,H,c_1}^{\SU(r)}(q) &= r^{-1} \sfZ_{S,H,c_1}^{(1^r)}(q) + r^{-1} \sfZ_{S,H,c_1}^{(r)}(q), \\
\sfZ_{S,H,c_1}^{(1^r)}(q) &= \Bigg( \frac{(-1)^{r-1}}{ r \Delta(q^r)^{\frac{1}{2}}} \Bigg)^{\chi(\O_S)} \Bigg( \frac{\Theta_{A_{r-1},0}(q)}{\eta(q)^r} \Bigg)^{-K_S^2} \\
&\quad\quad\quad\quad\times  C_0(q)^{K_S^2} \sum_{\bfbeta \in H^2(S,\Z)^{r-1}}  \delta_{c_1,\sum_i i \beta_i} \prod_{i} \SW(\beta_i) \prod_{i \leq j} C_{ij}(q)^{\beta_i \beta_j}, \\
\sfZ^{(r)}_{S,H,c_1}(q) &= r^{-1} q^{-\frac{\chi(\O_S)}{2r} + \frac{r K_S^2}{24}} \sum_{k=0}^{r} \epsilon_r^{k((r-1)c_1^2 + (r^2-1) \chi(\O_S))} \psi_{S,c_1}(\epsilon_r^k q^{\frac{1}{2r}}), \\
\psi_{S,c_1}(q^{\frac{1}{2r}})&=r^{2+K_S^2 - \chi(\O_S)} \Bigg( \frac{1}{\overline{\Delta}(q^{\frac{1}{r}})^{\frac{1}{2}}} \Bigg)^{\chi(\O_S)} \Bigg( \frac{\Theta_{{A^\vee_{r-1}},0}(q)   }{\overline{\eta}(q)^r} \Bigg)^{-K_S^2}  \\
&\quad\quad\quad\quad \times D_0(q)^{K_S^2} \sum_{\bfbeta \in H^2(S,\Z)^{r-1}}   \prod_{i} \epsilon_r^{i \beta_i c_1} \, \SW(\beta_i) \prod_{i \leq j} D_{ij}(q)^{\beta_i \beta_j},
\end{split}
\end{align}
where $\overline{\Delta}(q^{\frac{1}{r}})$, $\overline{\eta}(q)$ denote the normalizations of $\Delta(q^{\frac{1}{r}}), \eta(q)$ by their leading term (so they start with 1). Although this provides a formula for $\SU(r)$ Vafa-Witten invariants for any prime rank $r$, we are particularly interested in the rich modular structure of the universal series $C_0, C_{ij}, D_0, D_{ij}$.

We recall the mathematical statement of $S$-duality from the introduction (Conjecture \ref{intro:Sdualconj}):
\begin{conjecture}[Vafa-Witten] \label{Sdualconj} 
Let $(S,H)$ be a polarized  surface satisfying $H_1(S,\Z) = 0$ and $p_g(S)>0$. Let $r$ be prime and $c_1 \in H^2(S,\Z)$ algebraic. Then $\sfZ_{S,H,c_1}^{\SU(r)}(q)$ and $\sfZ_{S,H,c_1}^{^{L}\SU(r)}(q)$ are Fourier expansions in $q = \exp(2 \pi i \tau)$ of meromorphic functions $\sfZ_{S,H,c_1}^{\SU(r)}(\tau)$ and $\sfZ_{S,H,c_1}^{^{L}\SU(r)}(\tau)$ on $\mathfrak{H}$ satisfying
\begin{equation*} 
\mathsf{Z}^{\SU(r)}_{S,H,c_1}(-1/\tau) = (-1)^{(r-1)\chi(\O_S)} \Big( \frac{ r \tau}{i} \Big)^{-\frac{e(S)}{2}} \mathsf{Z}^{^{L}\SU(r)}_{S,H,c_1}(\tau).
\end{equation*}
\end{conjecture}

In this conjecture, we have
$$
\mathsf{Z}^{^{L}\SU(r)}_{S,H,c_1}(q) := r \sum_{w \in H^2(S,\Z_r)} \epsilon_r^{w c_1} \, \sfZ_{S,H,w}^{\SU(r)}(q).
$$
As discussed in the introduction, using Tanaka-Thomas's definition, $\sfZ_{S,H,w}^{\SU(r)}$ is obviously zero when $w$ cannot be represented by an algebraic class. On the other hand, the expression in \eqref{VWfull} does make sense (and is generally non-zero) when substituting \emph{non-algebraic} $c_1=w$ on the right hand side. Let us call this expression $\widetilde{\sfZ}_{S,H,w}^{\SU(r)}$. Then $\widetilde{\sfZ}_{S,H,w}^{\SU(r)}$ is manifestly invariant under replacing $w$ by $w + r \gamma$ for any $\gamma \in H^2(S,\Z)$. A \emph{geometric} definition of $\mathsf{Z}^{^{L}\SU(r)}_{S,H,c_1}$, using moduli spaces of twisted sheaves, is proposed in \cite{JK}. Conjecturally, the geometrically defined generating series $\mathsf{Z}^{^{L}\SU(r)}_{S,H,c_1}$ of loc.~cit.~satisfies
\begin{equation} \label{GKJconj}
\mathsf{Z}^{^{L}\SU(r)}_{S,H,c_1}(q) = r \sum_{w \in H^2(S,\Z_r)} \epsilon_r^{w c_1} \, \widetilde{\sfZ}_{S,H,w}^{\SU(r)}(q).
\end{equation}
This is proved for K3 surfaces with generic polarization in \cite[Thm.~1.5]{JK}. It is proved for a list of surfaces satisfying $H_1(S,\Z) = 0$, $p_g(S)>0$, up to certain virtual dimensions by combining \cite{BGJK} and \cite{GK1, GK3}. For the purposes of this paper, the reader can take \eqref{GKJconj} as the \emph{definition} of $\mathsf{Z}^{^{L}\SU(r)}_{S,H,c_1}$.

We now discuss how Conjecture \ref{conjCD} naturally implies part of the $S$-duality conjecture for any prime $r>2$. We will also show that our conjectural formulae for $C_0, C_{ij}, D_0, D_{ij}$ for ranks 3 and 5 indeed satisfy the $S$-duality conjecture.\footnote{The case $r=3$ was first covered in \cite{GK3}. The case $r=2$ appeared in Vafa-Witten's original paper \cite[Sect.~5]{VW} (see also \cite{GK3}).}

We require interesting lattice theoretic identities, known as \emph{flux sums}, which we learned from \cite{VW, LL}.
Since $H_1(S,\Z) = 0$, we have that $H^2(S,\Z)$ is torsion free and we consider the unimodular lattice $(H^2(S,\Z),\cup)$. Note that
$$
H^2(S,\Z) / r H^2(S,\Z) \cong H^2(S,\Z) \otimes_{\Z} \Z_r \cong H^2(S,\Z_r) 
$$
with its induced pairing. 
We denote the Betti numbers of $S$ by $b_i(S)$ and its signature by $\sigma(S)$. 
We recall that the cup product induces a non-degenerate symmetric bilinear form on $H^2(S,\R)$ and we denote its number of positive and negative eigenvalues by $b_2^+(S)$ and $b_2^-(S)$ respectively.
Then $b_2(S) = b_2^+(S) + b_2^-(S)$ and $\sigma(S) = b_2^+(S) - b_2^-(S)$. We have \cite[Eqn.~(A.16), (A.17)]{LL} (see also \cite{VW, GK3}):
\begin{proposition}[Labastida-Lozano] \label{LLlattice}
For $r>2$ prime, $c \in H^2(S,\Z)$ (not necessarily algebraic), and $m =1, \ldots, r-1$, we have
\begin{align*}
\sum_{x \in H^2(S,\Z_r)} e^{\frac{2 \pi i}{r} (x c)} &= r^{b_2(S)} \delta_{c,0}, \\
\sum_{x \in H^2(S,\Z_r)} e^{\frac{2 \pi i}{r} (x c)} e^{\frac{\pi i (r-1)}{r} m x^2} &= \epsilon(m)^{b_2(S)} r^{\frac{b_2(S)}{2}} e^{-\frac{\pi i}{8} (r-1)^2 \sigma(S)} e^{\frac{\pi i (r-1)}{r} n c^2},
\end{align*}
where $mn \equiv -1 \mod r$,
$$
\epsilon(m) := \left\{ \begin{array}{cc} \Big( \frac{m/2}{r} \Big) & \textrm{if} \ m \ \textrm{is even} \\ \Big( \frac{(m+r)/2}{r} \Big) & \textrm{if} \ m \ \textrm{is odd} \end{array} \right.
$$
and $\big( \frac{a}{r} \big)$ denotes the Legendre symbol.
\end{proposition}

We start with the following observation, which essentially states that ---in the $S$-duality transformation--- the vertical generating series and the ``$k=0$ term'' of the horizontal generating series are swapped. This phenomenon was observed for ranks 2,3 in \cite{GK3} motivated by \cite{DPS}.
\begin{proposition} \label{prop:Sdualpart}
For any prime $r>2$, Conjecture \ref{conjCD} and the relations $C_{ij}(q) = C_{r-j,r-i}(q)$ of Conjecture \ref{conj:symm} imply 
\begin{align*} 
&\sfZ_{S,H,c_1}^{(1^r)}(q) \big|_{\tau \to -\frac{1}{\tau}} = (-1)^{(r-1)\chi(\O_S)}  \Big( \frac{ r \tau}{i} \Big)^{-\frac{e(S)}{2}}  q^{-\frac{\chi(\O_S)}{2r} + \frac{r K_S^2}{24}} \sum_{w \in H^2(S,\Z_r)} \epsilon_r^{w c_1} \psi_{S,w}(q^{\frac{1}{2r}}), \\
&r^{-2} q^{-\frac{\chi(\O_S)}{2r} + \frac{r K_S^2}{24}} \psi_{S,c_1}(q^{\frac{1}{2r}}) \big|_{\tau \to -\frac{1}{\tau}} = (-1)^{(r-1)\chi(\O_S)}  \Big( \frac{ r \tau}{i} \Big)^{-\frac{e(S)}{2}} \sum_{w \in H^2(S,\Z_r)} \epsilon_r^{w c_1} \sfZ_{S,H,w}^{(1^r)}(q).
\end{align*}
\end{proposition}
\begin{proof}
We have the following transformations
\begin{align}
\begin{split} \label{eqn:requiredtrans}
&\eta(-1/\tau) = \Big( \frac{\tau}{i} \Big)^{\frac{1}{2}} \eta(\tau), \\
&\Delta(-r/\tau)^{\frac{1}{2}} = -\Big(\frac{\tau}{r}\Big)^{6} \Delta(\tau/r)^{\frac{1}{2}}, \\
&\Theta_{A_{r-1},0}(-1/\tau) = \frac{1}{\sqrt{r}} \Big( \frac{\tau}{i} \Big)^{\frac{r-1}{2}} \, \Theta_{A_{r-1}^{\vee},0}(\tau), \\
&C_0(-1/\tau) = D_0(\tau), \quad C_{ij}(\tau) = D_{ij}(-1/\tau),
\end{split}
\end{align}
where the last line is Conjecture \ref{conjCD}. The sign in the second formula comes from the transformation \cite[Thm.~3.4]{Apo}.
\begin{align} \label{etatransf}
\eta\Big( \frac{a \tau + b}{c \tau + d} \Big)^{12} =  \epsilon(a,b,c,d) (c\tau +d)^6 \eta(\tau)^{12}, \quad \Big( \begin{array}{cc} a & b \\ c & d \end{array}\Big) \in \mathrm{SL}(2,\Z)
\end{align}
where 
$$
\epsilon(a,b,c,d) := - e^{\pi i \Big( \frac{a+d}{c} + 12 s(-d,c) \Big)}.
$$
Here $s(h,k)$ denotes the Dedekind sum 
$$
s(h,k) = \sum_{r=1}^{k-1} \frac{r}{k} \Big(\frac{hr}{k} - \Big\lfloor \frac{hr}{k} \Big\rfloor - \frac{1}{2} \Big).
$$

Using \eqref{VWfull} and the transformations \eqref{eqn:requiredtrans}, we obtain
\begin{align*}
&\sfZ_{S,H,c_1}^{(1^r)}(q) \big|_{\tau \to -\frac{1}{\tau}} = (-1)^{(r-1)\chi(\O_S)}  \Big( \frac{ r \tau}{i} \Big)^{-\frac{e(S)}{2}}  r^{11 \chi(\O_S)} \Bigg( \frac{1}{\Delta(q^{\frac{1}{r}})^{\frac{1}{2}}} \Bigg)^{\chi(\O_S)} \times \\ 
&\Bigg( \frac{\Theta_{{A^\vee_{r-1}},0}(q)   }{\eta(q)^r} \Bigg)^{-K_S^2}   D_0(q)^{K_S^2} \sum_{\bfbeta \in H^2(S,\Z)^{r-1}}   \delta_{c_1,\sum_i i \beta_i} \prod_{i} \SW(\beta_i) \prod_{i \leq j} D_{ij}(q)^{\beta_i \beta_j}.
\end{align*}
On the other hand, using  the first equality of Proposition \ref{LLlattice}, we have 
\begin{align*}
 &(-1)^{(r-1)\chi(\O_S)}  \Big( \frac{ r \tau}{i} \Big)^{-\frac{e(S)}{2}}  q^{-\frac{\chi(\O_S)}{2r} + \frac{r K_S^2}{24}} \sum_{w \in H^2(S,\Z_r)} \epsilon_r^{w c_1} \psi_{S,w}(q^{\frac{1}{2r}}) = \\
&(-1)^{(r-1)\chi(\O_S)}  \Big( \frac{ r \tau}{i} \Big)^{-\frac{e(S)}{2}}  r^{11 \chi(\O_S)} \Bigg( \frac{1}{\Delta(q^{\frac{1}{r}})^{\frac{1}{2}}} \Bigg)^{\chi(\O_S)} \Bigg( \frac{\Theta_{{A^\vee_{r-1}},0}(q)   }{\eta(q)^r} \Bigg)^{-K_S^2}   \times \\
&D_0(q)^{K_S^2} \sum_{\bfbeta \in H^2(S,\Z)^{r-1}}   \delta_{c_1+\sum_i i \beta_i,0} \prod_{i} \SW(\beta_i) \prod_{i \leq j} D_{ij}(q)^{\beta_i \beta_j},
\end{align*}
which is \emph{almost} the same as $\sfZ_{S,H,c_1}^{(1^r)}(q) \big|_{\tau \to -\frac{1}{\tau}}$. We want the term $\delta_{c_1+\sum_i i \beta_i,0} = \delta_{c_1,\sum_i (r-i) \beta_i}$ to be replaced by  $\delta_{c_1,\sum_i i \beta_i}$. The relations $C_{ij}(q) = C_{r-j,r-i}(q)$ imply the relations $D_{ij}(q) = D_{r-j,r-i}(q)$ by Conjecture \ref{conjCD}. Therefore, replacing $\beta_i$ by $\beta_{r-i}$ and using the relations $D_{ij}(q) = D_{r-j,r-i}(q)$, we obtain
\begin{align*}
 &(-1)^{(r-1)\chi(\O_S)}  \Big( \frac{ r \tau}{i} \Big)^{-\frac{e(S)}{2}}  q^{-\frac{\chi(\O_S)}{2r} + \frac{r K_S^2}{24}} \sum_{w \in H^2(S,\Z_r)} \epsilon_r^{w c_1} \psi_{S,w}(q^{\frac{1}{2r}}) = \\
&(-1)^{(r-1)\chi(\O_S)}  \Big( \frac{ r \tau}{i} \Big)^{-\frac{e(S)}{2}}  r^{11 \chi(\O_S)} \Bigg( \frac{1}{\Delta(q^{\frac{1}{r}})^{\frac{1}{2}}} \Bigg)^{\chi(\O_S)} \Bigg( \frac{\Theta_{{A^\vee_{r-1}},0}(q)   }{\eta(q)^r} \Bigg)^{-K_S^2}   \times \\
&D_0(q)^{K_S^2} \sum_{\bfbeta \in H^2(S,\Z)^{r-1}}   \delta_{c_1,\sum_i i \beta_i} \prod_{i} \SW(\beta_i) \prod_{i \leq j} D_{ij}(q)^{\beta_i \beta_j},
\end{align*}
as desired. This proves the first equality of the proposition. The second equality of the proposition follows similarly (but is easier because it neither requires the relations $C_{ij}(q) = C_{r-j,r-i}(q)$ nor Proposition \ref{LLlattice}).
\end{proof}

Assuming Conjecture \ref{conjCD} holds, and using the previous proposition, the $S$-duality conjecture is equivalent to the following equality involving the horizontal generating series only 
\begin{align}
\begin{split} \label{Sdualremainder} 
&\frac{1}{r} q^{-\frac{\chi(\O_S)}{2r} + \frac{r K_S^2}{24}} \sum_{k=1}^{r-1} \epsilon_r^{k((r-1)c_1^2 + (r^2-1) \chi(\O_S))} \psi_{S,c_1}(\epsilon_r^k q^{\frac{1}{2r}}) \big|_{\tau \to -\frac{1}{\tau}} \\
&= (-1)^{(r-1)\chi(\O_S)} \Big( \frac{r \tau}{i} \Big)^{-\frac{e(S)}{2}} q^{-\frac{\chi(\O_S)}{2r} + \frac{r K_S^2}{24}} \sum_{w \in H^2(S,\Z_r)} \sum_{k=1}^{r-1} \epsilon_r^{w c_1}  \epsilon_r^{k((r-1)w^2 + (r^2-1) \chi(\O_S))} \psi_{S,w}(\epsilon_r^k q^{\frac{1}{2r}}),
\end{split}
\end{align}
where $q= \exp(2 \pi i \tau)$.
In the remainder of this section, we show that this equality holds for our conjectural formulae for $D_0, D_{ij}$ for $r=3,5$. 
This implies that our $\SU(3)$ and $\SU(5)$ Vafa-Witten partition functions indeed satisfy $S$-duality (Conjecture \ref{Sdualconj}).

By \eqref{etatransf}, we have
\begin{align*}
\Delta\Big( \frac{\tau + m}{r} \Big) \Big|_{\tau \mapsto - \frac{1}{\tau}} = \tau^{12} \Delta\Big( \frac{\tau + n}{r} \Big),
\end{align*}
for any $m,n \in \Z$ satisfying $mn \equiv -1 \mod r$ and 
\begin{align*}
\Delta\Big( \frac{\tau + m}{r} \Big)^{\frac{1}{2}} \Big|_{\tau \mapsto - \frac{1}{\tau}} = \epsilon(m,b,r,-n) \tau^{6} \Delta\Big( \frac{\tau + n}{r} \Big)^{\frac{1}{2}},
\end{align*}
where $b$ is defined by the equation $mn = -1 -br$. \\

\noindent $\boldsymbol{r=3}$. We have
\begin{align*}
\Delta\Big( \frac{\tau + 1}{3} \Big)^{\frac{1}{2}} \Big|_{\tau \mapsto - \frac{1}{\tau}}  &=\tau^{6} \Delta\Big( \frac{\tau + 2}{3} \Big)^{\frac{1}{2}}, \\
\Theta_{A_{2}^{\vee},0}(\tau+2) |_{\tau \mapsto - \frac{1}{\tau}} &= (-i) \Big( \frac{\tau}{i} \Big)   \, \Theta_{A_{2}^{\vee},0}(\tau+4), \\
\Theta_{A_{2}^{\vee},1}(\tau+2) |_{\tau \mapsto - \frac{1}{\tau}} &= (-i \epsilon_3) \Big( \frac{\tau}{i} \Big)   \, \Theta_{A_{2}^{\vee},1}(\tau+4).
\end{align*}
The last two identities follow from \eqref{thlk} and \eqref{maintrans} in Appendix \ref{sec:theta} (or see \cite{GK3}). In particular, we have
$t_{A_2^\vee,1}(\tau +2) |_{\tau \mapsto- \frac{1}{\tau}} = \epsilon_3^2 t_{A_2^\vee,1}(\tau + 4)$.
Under this transformation, the quadratic equation
$$
X^2 - 4  (t_{A_2^\vee,1}(\tau+2))^2 \, X + 4 t_{A_2^\vee,1}(\tau+2) = 0
$$
maps to
$$
(\epsilon_3^2 X)^2 - 4 (t_{A_2^\vee,1}(\tau+4))^2 \, (\epsilon_3^2 X) + 4 t_{A_2^\vee,1}(\tau+4) = 0
$$
and vice versa. In particular, $X_{\pm}(\tau+2)$ map to $\epsilon_3 X_{\pm}(\tau+4)$, and vice versa. Combining these transformations with the second equation of Proposition \ref{LLlattice}, we deduce \eqref{Sdualremainder} by a straightforward calculation (see also \cite{GK3}). Our conjectural formula for the $\SU(3)$ Vafa-Witten partition function satisfies $S$-duality. \\

\noindent $\boldsymbol{r=5}$. We have (Appendix \ref{sec:theta} and \eqref{etatransf})
\begin{align*}
\Delta\Big( \frac{\tau + 1}{5} \Big)^{\frac{1}{2}} \Big|_{\tau \mapsto - \frac{1}{\tau}} &= \tau^{6} \Delta\Big( \frac{\tau + 4}{5} \Big)^{\frac{1}{2}}, \\
\Delta\Big( \frac{\tau + 2}{5} \Big)^{\frac{1}{2}} \Big|_{\tau \mapsto - \frac{1}{\tau}}  &= - \tau^{6} \Delta\Big( \frac{\tau + 2}{5} \Big)^{\frac{1}{2}}, \\
\Delta\Big( \frac{\tau + 3}{5} \Big)^{\frac{1}{2}} \Big|_{\tau \mapsto - \frac{1}{\tau}}  &= - \tau^{6} \Delta\Big( \frac{\tau + 3}{5} \Big)^{\frac{1}{2}}, \\
\Theta_{A_{4}^\vee,0}(\tau+2) |_{\tau \mapsto - \frac{1}{\tau}} &=  \Big( \frac{\tau}{i} \Big)^2   \, \Theta_{A_{4}^{\vee},0}(\tau+2), \\
\Theta_{A_{4}^\vee,0}(\tau+4) |_{\tau \mapsto - \frac{1}{\tau}} &=  -\Big( \frac{\tau}{i} \Big)^2   \, \Theta_{A_{4}^{\vee},0}(\tau+6), \\
\Theta_{A_{4}^\vee,0}(\tau+8) |_{\tau \mapsto - \frac{1}{\tau}} &=  \Big( \frac{\tau}{i} \Big)^2   \, \Theta_{A_{4}^{\vee},0}(\tau+8), \\
\Theta_{A_{4}^{\vee},1}(\tau+2) |_{\tau \mapsto - \frac{1}{\tau}} &= \epsilon_5^4 \Big( \frac{\tau}{i} \Big)^2   \, \Theta_{A_{4}^{\vee},2}(\tau+2), \\
\Theta_{A_{4}^{\vee},1}(\tau+4) |_{\tau \mapsto - \frac{1}{\tau}} &= -\epsilon_5^2 \Big( \frac{\tau}{i} \Big)^2   \, \Theta_{A_{4}^{\vee},1}(\tau+6), \\
\Theta_{A_{4}^{\vee},2}(\tau+4) |_{\tau \mapsto - \frac{1}{\tau}} &= -\epsilon_5^3 \Big( \frac{\tau}{i} \Big)^2   \, \Theta_{A_{4}^{\vee},2}(\tau+6), \\
\Theta_{A_{4}^{\vee},1}(\tau+8) |_{\tau \mapsto - \frac{1}{\tau}} &= \epsilon_5 \Big( \frac{\tau}{i} \Big)^2   \, \Theta_{A_{4}^{\vee},2}(\tau+8).
\end{align*}
Hence $t_{A_4^\vee,1}(\tau+2) |_{\tau \mapsto - \frac{1}{\tau}} = \epsilon_5 t_{A_4^\vee,2}(\tau+2)$, $t_{A_4^\vee,1}(\tau+8) |_{\tau \mapsto  - \frac{1}{\tau}} = \epsilon_5^4 t_{A_4^\vee,2}(\tau+8)$, $t_{A_4^\vee,1}(\tau+4) |_{\tau \mapsto  - \frac{1}{\tau}} = \epsilon_5^3 t_{A_4^\vee,1}(\tau+6)$, and $t_{A_4^\vee,2}(\tau+4) |_{\tau \mapsto  - \frac{1}{\tau}} = \epsilon_5^2 t_{A_4^\vee,2}(\tau+6)$. Moreover, we have
\begin{align*}
s(\tau +2) |_{\tau \mapsto - \frac{1}{\tau}} &= - \epsilon_5^2 s(\tau +2)^{-1}, \\
s(\tau +4) |_{\tau \mapsto - \frac{1}{\tau}} &= \epsilon_5 s(\tau +6), \\
s(\tau +6) |_{\tau \mapsto - \frac{1}{\tau}} &= \epsilon_5^4 s(\tau +4), \\
s(\tau +8) |_{\tau \mapsto - \frac{1}{\tau}} &= -\epsilon_5^3 s(\tau +8)^{-1}.
\end{align*}
By \eqref{def:s} and the fact that $r(q) \in q^{1/5}(1+q\mathbb{Z}[[q]])$, the last four equations are immediate consequences of the following identities in the polynomial ring $\Q(\epsilon_5,\sqrt{5})[r]$ (where $r$ is a formal variable)
\begin{align*}
(\varphi(\varphi+r) + \epsilon_5^2(1-\varphi r))(-\varphi -\epsilon_5^2 r) - \epsilon_5^3 (\varphi + r - \varphi \epsilon_5^2(1-\varphi r))(1-\varphi \epsilon_5^2 r) &= 0, \\
(\varphi(\varphi+r) + \epsilon_5^4(1-\varphi r))(1-\varphi \epsilon_5 r) - \epsilon_5^4 (\varphi + r - \varphi \epsilon_5^4(1-\varphi r))(\varphi + \epsilon_5 r) &= 0, \\
(\varphi(\varphi+r) + \epsilon_5(1-\varphi r))(1-\varphi \epsilon_5^4 r) - \epsilon_5 (\varphi + r - \varphi \epsilon_5(1-\varphi r))(\varphi + \epsilon_5^4 r) &= 0, \\
(\varphi(\varphi+r) + \epsilon_5^3(1-\varphi r))(-\varphi -\epsilon_5^3 r) - \epsilon_5^2 (\varphi + r - \varphi \epsilon_5^3(1-\varphi r))(1-\varphi \epsilon_5^3 r) &= 0.
\end{align*}
Here $\varphi = (1 + \sqrt{5})/2$ denotes the golden ratio.
Altogether, we deduce that
\begin{align*}
X_{\pm}(\tau+2) |_{\tau \mapsto -\frac{1}{\tau}} &= \epsilon_5 Y_{\pm}(\tau+2), \quad Z(\tau+2) |_{\tau \mapsto -\frac{1}{\tau}} = Z(\tau+2),  \\
X_{\pm}(\tau+4) |_{\tau \mapsto -\frac{1}{\tau}} &= \epsilon_5^3 X_{\pm}(\tau+6), \quad Y_{\pm}(\tau+4) |_{\tau \mapsto -\frac{1}{\tau}} = \epsilon_5^2 Y_{\pm}(\tau+6), \\
Z(\tau+4) |_{\tau \mapsto -\frac{1}{\tau}} &= Z(\tau+6),  \\
X_{\pm}(\tau+8) |_{\tau \mapsto -\frac{1}{\tau}} &= \epsilon_5^4 Y_{\pm}(\tau+8), \quad Z(\tau+8) |_{\tau \mapsto -\frac{1}{\tau}} = Z(\tau+8).
\end{align*}
Combining these transformations with the second equation of Proposition \ref{LLlattice}, we deduce \eqref{Sdualremainder} from a straightforward calculation. This computation involves a relabeling of the $\beta_i$ and uses the identities $D_{ij} = D_{r-j,r-i}$. Our conjectural formulae for the $\SU(5)$ Vafa-Witten partition function satisfies $S$-duality. 

Since the case $r=2$ was already established in \cite[Sect.~5]{VW} (see also \cite[Sect.~4.7]{GK3}), we deduce:
\begin{proposition} \label{prop:Sdualcheck}
Assume Conjectures \ref{conjver:rk2}, \ref{conjver:rk3}, \ref{conjver:rk5}, \ref{conjhor}, \ref{conjCD} hold. Then $\mathsf{Z}_{S,H,c_1}^{\SU(r)}(q)$, for $r=2,3,5$, satisfies $S$-duality (Conjecture \ref{Sdualconj}). 
\end{proposition}

By \cite[Thm.~5.23]{Tho}, the generating series $\mathsf{Z}_{S,H,c_1}^{\SU(4)}(q)$ could receive contributions from the component 
$$
N_\mu^{\C^*}, \quad \mu = (2,2).
$$
Since we only have conjectural formulae for the vertical and horizontal contribution, indexed by $\mu=(1^4), (4)$, we cannot check whether it satisfies $S$-duality. Furthermore, Proposition \ref{LLlattice} requires adjustment in the non-prime case.

\subsection{Blow-up formula} \label{sec:blowup}

Denote the normalized Dedekind eta function by 
$$
\overline{\eta}(q) = \prod_{n=1}^{\infty} (1-q^n).
$$
We conjecture the following blow-up formula for $\psi_{S,c_1}$. Recall that $\psi_{S,c_1}$ was introduced in Definition \ref{def:smallpsi} (and depends on $r>1$, not necessarily prime).
\begin{conjecture} \label{conj:blowup}
Assume Conjecture \ref{conjhor} holds. Let $S$ be a smooth projective surface satisfying $b_1(S) = 0$, $p_g(S)>0$, and let $r>1$. Let $\pi : \widetilde{S} \to S$ be the blow-up of $S$ in a point. Let $c_1 \in H^2(S,\Z)$ and $\widetilde{c}_1 = \pi^* c_1 - \ell E$, where $E$ denotes the exceptional divisor and $\ell \in \Z$. Then
\begin{equation*} \label{eqn:blowup}
\psi_{\widetilde{S}, \widetilde{c}_1} = \Bigg( \frac{\Theta_{A_{r-1},\ell}(q)}{\overline{\eta}(q)^r} \Bigg) \psi_{S,c_1}.
\end{equation*}
\end{conjecture}

We recall that the Seiberg-Witten basic classes of $\widetilde{S}$ are $\pi^*\beta, \pi^*\beta+E$, where $\beta$ runs over the Seiberg-Witten basic classes of $S$, and \cite[Thm.~7.4.6]{Mor}
$$
\SW(\pi^* \beta) = \SW(\pi^* \beta + E) = \SW(\beta).
$$
We find that Conjecture \ref{conj:blowup} is equivalent to the following set of equations 
\begin{equation} \label{eqn:blowup2}
\frac{1}{r} \sum_{I \subset [r-1]} \epsilon_r^{\ell |\!| I |\!|} D_I(q)^{-1} = \frac{\Theta_{A_{r-1},\ell}(q)}{\Theta_{A_{r-1}^{\vee},0}(q)}, \quad \ell \in \Z.
\end{equation}

Assume Conjectures \ref{conjver:rk2}--\ref{conjver:rk5}, \ref{conjver:rk6}, \ref{conjver:rk7} for $C_0, C_{ij}$ hold. Consider the corresponding formulae for $D_0, D_{ij}$ (obtained via Conjecture \ref{conjCD}). By equations \eqref{TAdualasTA} and \eqref{maintrans} from Appendix \ref{sec:theta}
\begin{equation*} 
\frac{1}{r} \sum_{k=0}^{r-1} \epsilon_r^{k \ell} \Theta_{A_{r-1}^\vee,k}(q) = \Theta_{A_{r-1},\ell}(q),
\end{equation*}
from which we obtain \eqref{eqn:blowup2} for $r=2$--$7$ and $r \neq 5$. For $r=5$, \eqref{eqn:blowup2} is also true if we add the following additional conjectural relation to Conjecture \ref{conjver:rk5} 
$$
Z^{\frac{1}{2}}((X_+ Y_-)^{-1} + (X_- Y_+)^{-1}) + Z^{-\frac{1}{2}}((X_+ Y_+)^{-1} + (X_- Y_-)^{-1})  = (\beta_1\beta_2)^{-\frac{1}{2}}.
$$
This relation holds up to the order for which we determined the universal functions $C_0,C_{ij}$ in Appendix \ref{sec:data}.

Consider the setting of Conjecture \ref{conj:blowup} for any \emph{prime} rank $r>1$. Suppose that (for polarizations $H, \widetilde{H}$) there exist no rank $r$ strictly semistable sheaves on $S$, $\widetilde{S}$ with first Chern class $c_1$, $\widetilde{c}_1$ respectively. Then \eqref{Zandpsi} together with Conjecture \ref{conj:blowup} implies a blow-up formula for virtual Euler characteristics:
\begin{equation} \label{blowevir}
\sum_{c_2 \in \Z} e^{\vir}(M_{\widetilde{S}}^{\widetilde{H}}(r,\widetilde{c}_1,c_2)) \, q^{c_2 - \frac{(r-1)}{2r} \widetilde{c}_1^2} = \Bigg( \frac{\Theta_{A_{r-1},\ell}(q)}{\overline{\eta}(q)^r} \Bigg) \sum_{c_2 \in \Z} e^{\vir}(M_{S}^{H}(r,c_1,c_2)) \, q^{c_2 - \frac{(r-1)}{2r} c_1^2}.
\end{equation}
Here we used the identity
$$
\Theta_{A_{r-1},\ell}( \tau+2k) = \epsilon_r^{-k \ell^2} \Theta_{A_{r-1},\ell}(\tau), 
$$
which follows from \eqref{congruences} in Appendix \ref{sec:theta}.
Curiously, \eqref{blowevir} is \emph{identical} to blow-up formula for \emph{topological} Euler characteristics (\cite{Got3,Yos1} for arbitrary rank and \cite{LQ} for rank 2). We expect \eqref{blowevir} holds for non-prime $r>1$ as well.

Recently, the work of N.~Kuhn, O.~Leigh, and Y.~Tanaka \cite{KLT} (based on \cite{KT}) led to a proof of \eqref{blowevir} (under some slightly stronger assumptions on stability).

\appendix

\section{Multiplicative instanton invariants} \label{sec:mult}

Let $R$ be a commutative $\Q$-algebra (typically $\Q$ or a polynomial algebra over $\Q$). Fix $v \in \Z_{\geq 0}$, and $\boldsymbol{a} = (a_1, \ldots, a_N) \in \{1,2\}^N$ with $N \in \Z_{\geq 0}$. Let $(\underline{\alpha_0},\underline{\alpha_1},\ldots,\underline{\alpha_N})$ be a list of variables and let $\underline{r}$ also be a variable. Let  $\mathsf{P}_{v,\boldsymbol{a}}$ be a formal power series, with coefficients in $R(\underline{r})$, in the following formal symbols
\begin{align*}
&\pi_{\cM*}\Big( \pi_S^* \underline{\alpha_i} \cap \ch_{k}( \underline{\EE} \otimes \det(\underline{\EE})^{-1/\underline{r}})\Big), \quad \pi_{\cM*}\Big( \pi_S^* \underline{\alpha_i c_1(S)} \cap \ch_{k}( \underline{\EE} \otimes \det(\underline{\EE})^{-1 / \underline{r}})\Big), \\ 
&\pi_{\cM*}\Big( \pi_S^* \underline{\alpha_i c_2(S)} \cap \ch_{k}( \underline{\EE} \otimes \det(\underline{\EE})^{- 1 / \underline{r}})\Big), \quad \pi_{\cM*}\Big( \pi_S^* \underline{\alpha_i c_1(S)^2} \cap \ch_{k}( \underline{\EE} \otimes \det(\underline{\EE})^{-1 / \underline{r}})\Big), \\ 
&c_j(\underline{T}),
\end{align*}
where $i=0,\ldots, N$ and $k,j$ can be any non-negative integers, and $\underline{\alpha_i}, \underline{\alpha_i c_1(S)}, \ldots$, $\underline{\EE}$, $\underline{T}$ are all regarded as variables. We define degrees $\deg \, c_j(\cdot) = j$ and
\begin{align*}
\deg \,\pi_{\cM*}\Big( \pi_S^* \underline{\alpha_i} \cap \ch_{k}( \underline{\EE} \otimes \det(\underline{\EE})^{- 1 / \underline{r}})\Big) &= a_i + k-2, \\
\deg \,\pi_{\cM*}\Big( \pi_S^* \underline{\alpha_i c_1(S)} \cap \ch_{k}( \underline{\EE} \otimes \det(\underline{\EE})^{- 1  / \underline{r}})\Big) &= a_i + k - 1,  \\
\deg \,\pi_{\cM*}\Big( \pi_S^* \underline{\alpha_i c_2(S)} \cap \ch_{k}( \underline{\EE} \otimes \det(\underline{\EE})^{- 1 / \underline{r}})\Big) &= a_i + k, \\
\deg \,\pi_{\cM*}\Big( \pi_S^* \underline{\alpha_i c_1(S)^2} \cap \ch_{k}( \underline{\EE} \otimes \det(\underline{\EE})^{- 1  / \underline{r}})\Big) &= a_i + k, 
\end{align*}
where $i=0,\ldots, N$ and we set $a_0:=0$. We require that the formal power series  $\mathsf{P}_{v,\boldsymbol{a}}$ has only finitely many terms in each degree. 

Let $r >1$, let $S$ be a smooth projective surface, and $\boldsymbol{\alpha} = (\alpha_1, \ldots, \alpha_N) \in H^{*}(S,\Q)^N$ with $\deg(\alpha_i) = a_i$ for all $i=1, \ldots, N$. Furthermore, let $\cM$ be a proper Deligne-Mumford stack over $\C$, let $\EE$ be an $\cM$-flat coherent sheaf of rank $r$ on $S \times \cM$, and let $T$ be a perfect complex of constant rank on $\cM$ such that
$$
\rk(T) + \chi(\O_S) = v \geq 0.
$$
Then the evaluation $\mathsf{P}_{v,\boldsymbol{a}}(r,S,\boldsymbol{\alpha};\cM,\EE,T)$ is defined as the cohomology class on $\cM$ obtained from $\mathsf{P}_{v,\boldsymbol{a}}$ by substituting\footnote{There is some redundancy here. Obviously $\alpha_i c_2(S) = \alpha_i c_1(S)^2 = 0$ for all $i=1, \ldots, N$ for degree reasons. We nonetheless allow these classes for notational convenience.} 
\begin{align*}
&\underline{\alpha_0} = 1, \quad \underline{\alpha_0 c_1(S)} = c_1(S), \quad \underline{\alpha_0 c_2(S)} = c_2(S), \quad \underline{\alpha_0 c_1(S)^2} = c_1(S)^2,  \quad \underline{\EE} = \EE, \\ 
&\underline{r} = r, \quad \underline{\alpha_i} = \alpha_i, \quad \underline{\alpha_i c_1(S)} = \alpha_i c_1(S), \quad \underline{\alpha_i c_2(S)} = \alpha_i c_2(S), \quad \underline{\alpha_i c_1(S)^2} = \alpha_i c_1(S)^2
\end{align*}
for all $i=1, \ldots, N$, where $\pi_S : S \times \cM \to \cM$, $\pi_\cM : S \times \cM \to \cM$ now denote projections. 
When $\det(\EE)$ does not have an $r$th root, we simply define $\ch(\EE \otimes \det(\EE)^{-1 / r})$ by the right hand side of
$$
\ch(\EE \otimes \det(\EE)^{-1  / r}) = \ch(\EE) e^{- c_1(\EE) / r}.
$$
Note that this expression is invariant under replacing $\EE$ by $\EE \otimes L$ for any line bundle $L$ on $S \times \cM$. More generally, for any algebraic torus $\mathbb{T} = (\mathbb{C}^*)^{m}$, we can endow $S, \mathcal{M}$ with \emph{trivial} $\mathbb{T}$-action. Then for any choice of $\mathbb{T}$-equivariant structure on $\boldsymbol{\alpha}, \EE, T$, we obtain a $\mathbb{T}$-equivariant cohomology class on $\cM$ obtained from $\mathsf{P}_{v,\boldsymbol{a}}$ by the same substitutions as before \emph{and} viewing all Chern characters and Chern classes $T$-equivariantly. We denote this evaluation by $\mathsf{P}^{\mathbb{T}}_{v,\boldsymbol{a}}(r,S,\boldsymbol{\alpha};\cM,\EE,T)$.

Suppose we have a sequence $\{\mathsf{P}_{v,\boldsymbol{a}}\}_{v=0}^{\infty}$ of such formal power series. Let $S',S''$ be two smooth projective surfaces with $\boldsymbol{\alpha}', \boldsymbol{\alpha}''$ such that $\deg(\alpha_i') = \deg(\alpha_i'') = a_i$ for all $i=1, \ldots, N$. Let $\cM',\cM''$ be two proper Deligne-Mumford stacks over $\C$ with $\cM'$- resp.~$\cM''$-flat coherent sheaves $\EE', \EE''$ of rank $r$ on $S' \times \cM'$, $S'' \times \cM''$, and perfect complexes $T',T''$ of constant rank on $\cM',\cM''$ satisfying
$$
\rk(T') + \chi(\O_{S'}) = v' \geq 0, \quad \rk(T'') + \chi(\O_{S''})= v'' \geq 0.
$$
Then 
$$
\boldsymbol{\alpha}' \oplus \boldsymbol{\alpha}'' = (\alpha_1' \oplus \alpha_1'', \ldots, \alpha_N' \oplus \alpha_N'')
$$
are cohomology classes on the disjoint union $S' \sqcup S''$ with $\deg(\alpha_i' \oplus \alpha_i'') = a_i$ for all $i$. Denote by $j' : S' \times \cM' \times \cM'' \hookrightarrow (S' \sqcup S'') \times \cM' \times \cM''$ and $j'' : S'' \times \cM' \times \cM'' \hookrightarrow (S' \sqcup S'') \times \cM' \times \cM''$ (the base change of) the inclusions. Then we have
$$
j'_* \EE' \oplus j''_* \EE'', \quad \textrm{on \ } (S' \sqcup S'') \times \cM' \times \cM'', 
$$
where we suppressed the pull-back of $\EE'$ from $S' \times \cM'$ to $S' \times \cM' \times \cM''$ (and similarly for $\EE''$). Moreover $\rk(T' \boxplus T'') +\chi(\O_{S' \sqcup S''}) = \rk(T')+\rk(T'') + \chi(\O_{S'}) + \chi(\O_{S''}) = v'+v''$ and it makes sense to consider the evaluation
$$
\mathsf{P}_{v'+v'',\boldsymbol{a}}(r,S' \sqcup S'',\boldsymbol{\alpha}' \oplus \boldsymbol{\alpha}'';\cM' \times \cM'',j'_* \EE' \oplus j''_* \EE'', T' \boxplus T'').
$$
More generally, for a torus $\mathbb{T}$, acting trivially on $S',S'',\mathcal{M}',\mathcal{M}''$, and a choice of $\mathbb{T}$-equivariant structure on $\boldsymbol{\alpha}',\boldsymbol{\alpha}'',\EE',\EE'',T',T''$, we can consider the evaluation
$$
\mathsf{P}^{\mathbb{T}}_{v'+v'',\boldsymbol{a}}(r,S' \sqcup S'',\boldsymbol{\alpha}' \oplus \boldsymbol{\alpha}'';\cM' \times \cM'',j'_* \EE' \oplus j''_* \EE'', T' \boxplus T'').
$$

\begin{definition} \label{def:mult}
In the above setting, we say that $\{\mathsf{P}_{v,\boldsymbol{a}}\}_{v=0}^{\infty}$ is a \emph{multiplicative sequence of insertions} with coefficients in $R$ when
\begin{align*}
\mathsf{P}^{\mathbb{T}}_{v'+v'',\boldsymbol{a}}(r,S' \sqcup S'',\boldsymbol{\alpha}' \oplus \boldsymbol{\alpha}'';\cM' \times \cM'',j'_* \EE' \oplus j''_* \EE'', T' \boxplus T'') = \\
\mathsf{P}^{\mathbb{T}}_{v',\boldsymbol{a}}(r,S',\boldsymbol{\alpha}';\cM',\EE',T') \cdot \mathsf{P}^{\mathbb{T}}_{v'',\boldsymbol{a}}(r,S'',\boldsymbol{\alpha}'';\cM'',\EE'',T''),
\end{align*}
for all $(S',\boldsymbol{\alpha}',\cM',\EE',T')$, $(S'',\boldsymbol{\alpha}'',\cM'',\EE'',T'')$, $\mathbb{T}$ as above.
\end{definition}
In our applications, $T$ is always of the form $R^\mdot\hom_\pi(\EE,\EE)[1]$, where $\pi : S \times \cM \to \cM$ denotes the projection.

From now on, we restrict our attention to such multiplicative sequences when $(S,H)$ is a smooth polarized surface satisfying $b_1(S) = 0$, $p_g(S)>0$, and $\cM = \cM_S^H(r,c_1,c_2)$ is the Deligne-Mumford stack of rank $r$ oriented $H$-stable sheaves on $S$ with Chern classes $c_1,c_2$. We recall that an oriented sheaf $(\cE,\theta)$ with first Chern class $c_1$ consists of a coherent sheaf $\cE$ on $S$ together with an isomorphism $$\theta : \det(\cE) \cong \O_S(c_1).$$ This moduli space was used by Mochizuki in \cite{Moc}. Its advantage over the Gieseker-Maruyama moduli space $M := M_S^H(r,c_1,c_2)$ is that there always exists a \emph{unique} universal sheaf $\EE$ on $S \times \cM$ such that
$$
\det(\EE) \cong \pi_S^* \O_S(c_1).
$$
Note that there is a degree $\frac{1}{r} : 1$ morphism
$\cM \to M$.
We assume that there exist no rank $r$ strictly $H$-semistable sheaves on $S$ with Chern classes $c_1,c_2$ so that $\cM$ is proper and has a virtual class $[\cM]^{\vir}$. Indeed, similar to the Gieseker-Maruyama moduli space, $\cM$ has a perfect obstruction theory with virtual tangent bundle and virtual dimension given by \cite{Moc}
\begin{align*}
&T_{\cM}^{\vir}|_{[\cE]} \cong R^\mdot\Hom(\cE,\cE)_0[1], \\
&\vd := \vd(r,c_1,c_2) = 2rc_2 - (r-1)c_1^2 - (r^2-1) \chi(\O_S).
\end{align*}

\begin{conjecture} \label{conj:multinst}
Fix a multiplicative sequence of insertions $\{\mathsf{P}_{v,\boldsymbol{a}}\}_{v=0}^{\infty}$ with coefficients in a $\Q$-algebra $R$. 
For each $r > 1$, there exist formal power series
\begin{align*}
&A, \quad B, \quad \{V_i\}_{i \in \{1, \ldots, N\} \, \mathrm{s.t.} \, a_i=1}, \quad \{W_{ij}\}_{i \leq j \in \{1, \ldots, N\} \, \mathrm{s.t.} \, a_i=a_j=1}, \\ 
&\{X_i\}_{i \in \{1, \ldots, N\} \, \mathrm{s.t.} \, a_i=2}, \quad \{Y_{ij}\}_{i \in \{1, \ldots, N\} \, \mathrm{s.t.} \, a_i=1, j \in \{1, \ldots, r-1\}}, \quad \{Z_{ij}\}_{i \leq j \in \{1, \ldots, r-1\}}
\end{align*}
with $A, W_{ij}, X_i \in R \otimes_{\Q} \C [[q^2]]$, $B, V_{i}, Y_{ij}, Z_{ij} \in R \otimes_{\Q} \C [[q]]$, and only depending on $r$ and $\{\mathsf{P}_{v,\boldsymbol{a}}\}_{v=0}^{\infty}$, with the following property.
Let $(S,H)$ be a smooth polarized surface satisfying $b_1(S) = 0$ and $p_g(S)>0$. Let $c_1 \in H^2(S,\Z),c_2 \in H^4(S,\Z)$ be chosen such that there exist no rank $r$ strictly $H$-semistable sheaves on $S$ with Chern classes $c_1,c_2$. Let $\boldsymbol{\alpha} = (\alpha_1, \ldots, \alpha_N) \in H^*(S,\Q)^N$ with $\deg(\alpha_i) = a_i$ for all $i$. Then, for $ \cM:=\cM_S^H(r,c_1,c_2)$ and $\vd:=\vd(r,c_1,c_2)$, 
$$
\int_{[\cM]^{\vir}} \mathsf{P}_{\vd,\boldsymbol{a}}(r,S,\boldsymbol{\alpha};\cM,\EE,R^\mdot\hom_\pi(\EE,\EE)[1])
$$
equals the coefficient of $q^{\vd}$ of the following expression
\begin{align*}
&r \cdot  A^{\chi(\O_S)} B^{K_S^2} \prod_{i \in \{1, \ldots, N\} \atop \mathrm{s.t.} \, a_i=1} V_i^{\alpha_i K_S} \prod_{i \leq j \in \{1, \ldots, N\} \atop \mathrm{s.t.} \, a_i=a_j=1} W_{ij}^{\alpha_i \alpha_j} \prod_{i \in \{1, \ldots, N\} \atop \mathrm{s.t.} \, a_i=2} X_i^{\alpha_i} \\
&\times  \sum_{(\beta_1, \ldots, \beta_{r-1}) \in H^2(S,\Z)^{r-1}}   \prod_{i \in \{1, \ldots, r-1\}} \epsilon_r^{i \beta_i c_1}  \, \SW(\beta_i) \prod_{i \in \{1, \ldots, N\} \atop \mathrm{s.t.} \,a_i=1} \prod_{j \in \{1, \ldots, r-1\}} Y_{ij}^{\alpha_i \beta_j} \prod_{i \leq j \in \{1, \ldots, r-1\}} Z_{ij}^{\beta_i \beta_j}.
\end{align*}
\end{conjecture}

In each of the cases listed at the start of Section \ref{sec:horconj}, the invariants can be expressed in terms of a multiplicative sequence $\{\mathsf{P}_{v,\boldsymbol{a}}\}_{v=0}^{\infty}$. This is proved in \cite{GK1}--\cite{GK4,GKW}. Therefore, Conjecture \ref{conj:multinst} implies each of the conjectures in loc.~cit.~(except of course for the explicit expressions given for the universal functions in each case).\footnote{In loc.~cit.~the invariants are defined on the Gieseker-Maruyama moduli space $M$, whereas in this appendix, we work with the Deligne-Mumford stack $\cM$. However, the virtual classes are easily related via the degree $\frac{1}{r} : 1$ morphism $\cM\to M$, and the invariants only differ by a factor $r$.} As we learn from the conjectures in loc.~cit., it is crucial that we require the universal functions to have coefficients in $R \otimes_{\Q} \C$ (as opposed to just $R$).

\begin{remark} \label{differentrks:inst}
Fix a multiplicative sequence of insertions $\{\mathsf{P}_{v,\boldsymbol{a}}\}_{v=0}^{\infty}$. Then for any $r>1$, we consider the universal functions $A, B, \ldots$ of Conjecture \ref{conj:multinst}. When we want to stress their dependence on $r$, we write $A^{(r)}, B^{(r)}, \ldots$ We expect that the universal functions $Y^{(r)}_{ij}$ only depend on the ratio $i/r$ and $Z^{(r)}_{ij}$, for $i<j$, only depend on the ratios $i/r, j/r$. More precisely, we conjecture that for all $r,s>1$ we have
\begin{align*}
Y^{(rs)}_{i,js} &= Y^{(r)}_{ij}, \quad \forall i \in \{1, \ldots, N\} \, \mathrm{s.t.} \, a_i = 1, \quad \forall j \in \{1, \ldots, r-1\} \\ 
Z^{(rs)}_{is,js} &= Z^{(r)}_{ij}, \quad \forall i < j \in \{1, \ldots, r-1\}.
\end{align*}
This property was first observed in the calculations in \cite{Got4} on blow-up formulae for virtual Segre and Verlinde numbers. In the setting of Vafa-Witten theory, this property is equivalent ---via Conjecture \ref{conjCD}--- to a similar property of the universal functions $C_{ij}^{(r)}$ (Remark \ref{differentrks:mon}).
\end{remark}

Let $\T := (\C^{*})^{r-1}$ be an algebraic torus, with characters $\t_1, \ldots, \t_{r-1} \in X(\mathbb{T})$, acting trivially on a point $\pt = \Spec \C$. Consider
$$
H_{\T}^*(\mathrm{pt},\Z) = \Z[t_1^{\pm 1}, \ldots, t_{r-1}^{\pm 1}],
$$
where $t_i = c^{\T}_1(\t_i)$ are the equivariant parameters. The following (rational) characters in $X(\mathbb{T}) \otimes_{\Z} \Q$ feature in the theorem below
\begin{align*}
\mathfrak{T}_i := \t_i^{-1} \otimes \bigotimes_{j<i} \t_j^{\frac{1}{r-j}}, \quad \mathfrak{T}_r := \bigotimes_{j<r} \t_j^{\frac{1}{r-j}}, \quad T_j := c_1^{\T}(\mathfrak{T}_j),
\end{align*}
for all $i=1, \ldots, r-1$ and $j=1, \ldots, r$. Let $(S,H)$ be a smooth polarized surface. For any $\ch \in H^*(S,\Q)$, we define
$$
\chi(\ch) := \int_S \ch \cdot \td(S).
$$
Furthermore, for any $c \in H^2(S,\Z)$, we define $\chi(c) := \chi(e^c)$ and we denote the corresponding Hilbert polynomial by $P_{\ch}(m) = \chi(\ch \cdot e^{mH})$. In what follows, we write $\mathrm{Coeff}_{t^0}(f)$ for the coefficient of $t^0$ of a formal Laurent series $f(t)$.
\begin{theorem} \label{thm:multinst}
Fix a multiplicative sequence of insertions $\{\mathsf{P}_{v,\boldsymbol{a}}\}_{v=0}^{\infty}$ with coefficients in a $\Q$-algebra $R$. For any $r> 1$, there exist formal power series\footnote{In Conjecture \ref{conj:multinst}, indices took values in $\{1, \ldots, r-1\}$, whereas now in $\{1, \ldots, r\}$. We also stress that the coefficients are in $R(\!(t_1,\ldots, t_{r-1})\!)$ (as opposed to $R \otimes_{\Q} \C$).}
\begin{align*}
&\widetilde{A}, \quad \widetilde{B}, \quad \{\widetilde{V}_i\}_{i \in \{1, \ldots, N\} \, \mathrm{s.t.} \, a_i=1}, \quad \{\widetilde{W}_{ij}\}_{i \leq j \in \{1, \ldots, N\} \, \mathrm{s.t.} \, a_i=a_j=1}, \\
&\{\widetilde{X}_i\}_{i \in \{1, \ldots, N\} \, \mathrm{s.t.} a_i=2}, \quad \{\widetilde{Y}_{ij}\}_{i \in \{1, \ldots, N\} \, \mathrm{s.t.} \, a_i=1, j \in \{1, \ldots, r\}}, \quad \{\widetilde{Z}_{ij}\}_{i \leq j \in \{1, \ldots, r\}}
\end{align*}
in $1+ q^{2r} \, R(\!(t_1,\ldots, t_{r-1})\!)[[q^{2r}]]$, only depending on $r$ and $\{\mathsf{P}_{v,\boldsymbol{a}}\}_{v=0}^{\infty}$, with the following property.
Let $(S,H)$ be a smooth polarized surface satisfying $b_1(S) = 0$ and $p_g(S)>0$. Let $c_1 \in H^2(S,\Z),c_2 \in H^4(S,\Z)$ be chosen such that there exist no rank $r$ strictly $H$-semistable sheaves on $S$ with Chern classes $c_1,c_2$. Let $\boldsymbol{\alpha} = (\alpha_1, \ldots, \alpha_N) \in H^*(S,\Q)^N$ with $\deg(\alpha_i) = a_i$ for all $i$. Assume the following:
\begin{enumerate}
\item[(i)] $P_{(r, c_1, \frac{1}{2} c_1^2-c_2)}(m) / r > P_{e^{K_S}}(m)$ for all $m \gg 0$,
\item[(ii)] $\chi(r, c_1, \frac{1}{2} c_1^2-c_2) > (r-2) \chi(\O_S)$,
\item[(iii)] for any $\beta_1, \ldots, \beta_r \in H^2(S,\Z)$ satisfying $\beta_1 + \cdots +\beta_r=c_1$, such that $\beta_1, \ldots, \beta_{r-1}$ are Seiberg-Witten basic classes and $\beta_i H \leq \frac{1}{r-i} \sum_{j>i} \beta_jH$ for all $i=1, \ldots, r-1$, we have strict inequalities, i.e., $\beta_i H < \frac{1}{r-i} \sum_{j>i} \beta_jH$ for all $i=1, \ldots, r-1$.
\end{enumerate}
For any $\boldsymbol{\beta}=(\beta_1, \ldots, \beta_r) \in H^2(S,\Z)^r$, define
\begin{align*}
&\mathsf{C}(r,S,\boldsymbol{\alpha},\boldsymbol{\beta},\boldsymbol{t}):= q^{-(r-1)\sum_{i=1}^{r} \beta_i^2 + 2\sum_{i<j} \beta_i\beta_j -(r^2-1) \chi(\O_S)} \prod_{i=1}^{r-1} t_i^{\sum_{j \geq i} ( \frac{1}{2} \beta_j(\beta_j-K_S) + \chi(\O_S)  ) } \\
&\times \prod_{i<j}  (T_j - T_i)^{\chi(\beta_j-\beta_i) - \chi(\beta_j)} (T_i-T_j)^{\chi(\beta_i-\beta_j)} \\
&\times \mathsf{P}^{\T}_{0,\boldsymbol{a}}\Bigg(r,S,\boldsymbol{\alpha}; \pt, \bigoplus_{i=1}^{r} \O_S(\beta_i) \otimes \mathfrak{T}_i,R^\mdot \Hom_S\Big(\bigoplus_{i=1}^{r} \O_S(\beta_i) \otimes \mathfrak{T}_i,\bigoplus_{i=1}^{r} \O_S(\beta_i) \otimes \mathfrak{T}_i)\Big)[1] \Bigg).
\end{align*}
Then, for $ \cM:=\cM_S^H(r,c_1,c_2)$ and $\vd:=\vd(r,c_1,c_2)$, we have that
$$
\int_{[\cM]^{\vir}} \mathsf{P}_{\vd,\boldsymbol{a}}(r,S,\boldsymbol{\alpha};\cM,\EE,R^\mdot\hom_\pi(\EE,\EE)[1])
$$
is given by $\mathrm{Coeff}_{t_{1}^0} \cdots \mathrm{Coeff}_{t_{r-1}^0}$ of the coefficient of $q^{\vd}$ of the following expression
\begin{align*}
&\widetilde{A}^{\chi(\O_S)} \widetilde{B}^{K_S^2} \prod_{i \in \{1, \ldots, N\} \atop \mathrm{s.t.} \, a_i=1} \widetilde{V}_i^{\alpha_i K_S} \prod_{i \leq j \in \{1, \ldots, N\} \atop \mathrm{s.t.} \, a_i=a_j=1} \widetilde{W}_{ij}^{\alpha_i \alpha_j} \prod_{i \in \{1, \ldots, N\} \atop \mathrm{s.t.} \, a_i=2} \widetilde{X}_i^{\alpha_i} \\
&\times \sum_{\boldsymbol{\beta}} (-1)^{r-1} \mathsf{C}(r,S,\boldsymbol{\alpha},\boldsymbol{\beta},\boldsymbol{t})  \prod_{i \in \{1, \ldots, r-1\}} \SW(\beta_i) \prod_{i \in \{1, \ldots, N\} \atop \mathrm{s.t.} \, a_i=1} \prod_{j \in \{1, \ldots, r\}} \widetilde{Y}_{ij}^{\alpha_i \beta_j} \prod_{i \leq j \in \{1, \ldots, r\}} \widetilde{Z}_{ij}^{\beta_i \beta_j},
\end{align*}
where the sum is over all $\boldsymbol{\beta} = (\beta_1, \ldots, \beta_r) \in H^2(S,\Z)^r$ satisfying $\beta_1 + \cdots +\beta_r=c_1$ and $\beta_i H \leq  \frac{1}{r-i} \sum_{j>i} \beta_jH$ for all $i=1, \ldots, r-1$.
\end{theorem}
\begin{proof}
The proof is similar to \cite[Thm.~2.2, 2.3]{GK4}, so we will be brief. Fix $\{\mathsf{P}_{v,\boldsymbol{a}}\}_{v=0}^{\infty}, r,S,H,c_1,c_2,\boldsymbol{\alpha}$ as in the statement of the theorem. We claim that
\begin{equation} \label{integral}
\mathsf{P}_{\vd,\boldsymbol{a}}(r,S,\boldsymbol{\alpha};\cM,\EE,R^\mdot \hom_\pi(\EE,\EE)[1])
\end{equation}
can be written as a formal power series, with coefficients in $R$, in expressions of the form
\begin{equation} \label{simpleslant}
\pi_{\cM*}\big( \pi_S^* \theta \cap \ch_{\ell}(\EE)\big)
\end{equation}
for certain $\theta \in H^*(S,\Q)$ and $\ell \in \Z_{\geq 0}$. By Grothendieck-Riemann-Roch
$$
\ch(R^\mdot \hom_\pi(\EE,\EE)) = \pi_{\mathcal{M}!}( \ch(\EE)^{\vee} \cdot \ch(\EE) \cdot  \td(S)).
$$
Therefore, in order to establish the claim, it suffices to show that any expression of the form
$$
\pi_{\cM*}\big( \pi_S^* \gamma \cap \prod_{j=1}^{L} \ch_{\ell_j}(\EE)\big)
$$
can be rewritten as a polynomial expression in ``simple slant products'' \eqref{simpleslant}. (In fact, we only require the case $L=2$.) To see this, we consider $S^{\times L} \times \mathcal{M}$ and we denote the projections to the various factors by $\pi_i, \pi_{ij}, \pi_{ijk}, \ldots$ Then
$$
\pi_{\cM*}\big( \pi_S^* \gamma \cap \prod_{j=1}^{L} \ch_{\ell_j}(\EE)\big) = \pi_{L+1*}\big( \pi_1^* \gamma \cdot  \prod_{j=1}^{L} \pi_{j L+1}^* \ch_{\ell_j}(\EE) \cdot \pi_{1\cdots L}^* \mathrm{PD}(\Delta) \big),
$$
where $\mathrm{PD}(\Delta)$ denotes the Poincar\'e dual of the small diagonal $\Delta \subset S^{\times L}$. Substituting the K\"unneth decomposition
$$
\mathrm{PD}(\Delta) = \sum_{i_1 + \cdots + i_L = 4L - 4} \theta_1^{(i_1)} \boxtimes \cdots \boxtimes \theta_L^{(i_L)} \in \bigoplus_{i_1 + \cdots + i_L = 4L - 4} H^{i_1}(S,\Q) \times \cdots \times H^{i_L}(S,\Q),
$$
the claim easily follows from the projection formula and flat base change.

Since \eqref{integral} is a polynomial expression in ``simple slant products'' \eqref{simpleslant}, we can apply Mochizuki's formula \cite[Thm.~7.5.2]{Moc}. In order to fulfill the conditions of \cite[Thm.~7.5.2]{Moc}, we must assume (i)--(iii). It implies that we can express \eqref{integral} as $\mathrm{Coeff}_{t_{1}^0} \cdots \mathrm{Coeff}_{t_{r-1}^0}$ of an expression of the form
$$
(-1)^{r-1} \sum_{\beta_1 + \cdots + \beta_{r}=c_1 \, \mathrm{s.t.} \atop \beta_i H \leq \frac{1}{r-i} \sum_{j>i} \beta_jH \ \forall i}  \sum_{n_1 + \cdots + n_r =c_2 - \sum_{i<j} \beta_i \beta_j} \prod_{i=1}^{r-1} \SW(\beta_i) \int_{S^{[\boldsymbol{n}]}} \widetilde{\Psi}(\boldsymbol{\alpha},\boldsymbol{\beta}, \boldsymbol{n},\boldsymbol{t}),
$$
where $S^{[\boldsymbol{n}]} = S^{[n_1]} \times \cdots \times S^{[n_r]}$ and $\widetilde{\Psi}(\boldsymbol{\alpha},\boldsymbol{\beta}, \boldsymbol{n},\boldsymbol{t})$ is an explicit expression depending on $\{\mathsf{P}_{v,\boldsymbol{a}}\}_{v=0}^{\infty}$, $r$, $\boldsymbol{\alpha} = (\alpha_1, \ldots, \alpha_N)$, $\boldsymbol{\beta} = (\beta_1, \ldots, \beta_r)$, $\boldsymbol{n} = (n_1, \ldots, n_r)$, and $\boldsymbol{t} = (t_1, \ldots, t_{r-1})$. The precise formula for $\widetilde{\Psi}(\boldsymbol{\alpha},\boldsymbol{\beta}, \boldsymbol{n},\boldsymbol{t})$ is reviewed in \cite[Sect.~6.1, p.~105]{GK5}. Consider
$$
\sum_{\boldsymbol{n} \in \Z_{\geq 0}^r} q^{2r |\boldsymbol{n}| -(r-1)\sum_i \beta_i^2 + 2\sum_{i<j} \beta_i\beta_j -(r^2-1) \chi(\O_S)} \int_{S^{[\boldsymbol{n}]}} \widetilde{\Psi}(\boldsymbol{\alpha},\boldsymbol{\beta}, \boldsymbol{n}, \boldsymbol{t}),
$$
where the power is the virtual dimension. The coefficient of the lowest order term of this expression, obtained by setting $n_1 = \cdots = n_r = 0$, yields the expression for $\mathsf{C}(r,S,\boldsymbol{\alpha},\boldsymbol{\beta},\boldsymbol{t})$ defined in the statement of the theorem. We now normalize this generating series and study it separately.

Let $S$ be any \emph{possibly disconnected} smooth projective surface. Take \emph{any} $\boldsymbol{\alpha} = (\alpha_1, \ldots, \alpha_N) \in H^*(S,\Q)^N$ with $\deg(\alpha_i) = a_i$ and algebraic classes $\boldsymbol{\beta} = (\beta_1, \ldots, \beta_r) \in H^2(S,\Q)^r$. Consider the generating function
\begin{equation*} 
\mathsf{Z}_S(\boldsymbol{\alpha},\boldsymbol{\beta}, \boldsymbol{t},q) = \frac{1}{\mathsf{C}(r,S,\boldsymbol{\alpha},\boldsymbol{\beta},\boldsymbol{t})} \sum_{\boldsymbol{n} \in \Z_{\geq 0}^{r}} q^{2r |\boldsymbol{n}| -(r-1)\sum_i \beta_i^2 + 2\sum_{i<j} \beta_i\beta_j -(r^2-1) \chi(\O_S)} \int_{S^{[\boldsymbol{n}]}} \widetilde{\Psi}(\boldsymbol{\alpha}, \boldsymbol{\beta}, \boldsymbol{n}, \boldsymbol{t}).
\end{equation*}
This series has constant term 1 and is defined for any (possibly disconnected) $S$ (not necessarily satisfying $b_1(S) = 0$ or $p_g(S)>0$) and any $\boldsymbol{\alpha}, \boldsymbol{\beta}$ satisfying $\rk(\alpha_i) = a_i$ for all $i$. We claim that
\begin{equation} \label{multZinst}
\mathsf{Z}_{S' \sqcup S''}(\boldsymbol{\alpha}' \oplus \boldsymbol{\alpha}'',\boldsymbol{\beta}' \oplus \boldsymbol{\beta}'', \boldsymbol{t},q) = \mathsf{Z}_{S'}(\boldsymbol{\alpha}',\boldsymbol{\beta}', \boldsymbol{t},q) \mathsf{Z}_{S''}(\boldsymbol{\alpha}'',\boldsymbol{\beta}'', \boldsymbol{t},q),
\end{equation}
for any $S',\boldsymbol{\alpha}',\boldsymbol{\beta}'$ and $S'',\boldsymbol{\alpha}'',\boldsymbol{\beta}''$ such that $\rk(\alpha_i') = \rk(\alpha_i'') = a_i$ for all $i$. 
This follows from the fact that
$$
S^{[\boldsymbol{n}]} \cong \bigsqcup_{\boldsymbol{n} = \boldsymbol{n}' + \boldsymbol{n}''} S^{\prime [\boldsymbol{n}']} \times S^{\prime \prime [\boldsymbol{n}'']},
$$
for all $\boldsymbol{n} = (n_1, \ldots, n_r) \in \Z_{\geq 0}^r$, combined with the multiplicative property (Definition \ref{def:mult}). In particular, we use that the normalization term is multiplicative
$$
\mathsf{C}(r,S' \sqcup S'',\boldsymbol{\alpha}' \oplus \boldsymbol{\alpha}'',\boldsymbol{\beta}' \oplus \boldsymbol{\beta}'',\boldsymbol{t}) = \mathsf{C}(r,S',\boldsymbol{\alpha}',\boldsymbol{\beta}',\boldsymbol{t}) \mathsf{C}(r,S'',\boldsymbol{\alpha}'',\boldsymbol{\beta}'',\boldsymbol{t}).
$$

Combining \eqref{multZinst} with (a suitable upgrade of) the universality theorem \cite[Thm.~4.1]{EGL} yields the result. This last step, often called ``cobordism argument'', can be found in detail in \cite{GNY1, GK1,GK5, GKW} (among many other places).
\end{proof}

The proof of this theorem is amenable to computer implementation. Indeed, the universal function $\mathsf{Z}_S(\boldsymbol{\alpha},\boldsymbol{\beta}, \boldsymbol{t},q)$ is defined for any smooth projective surface $S$ and any $\boldsymbol{\alpha}, \boldsymbol{\beta}$ satisfying $\rk(\alpha_i) = a_i$ for all $i$. Therefore, it is determined on any collection of $S,\boldsymbol{\alpha},\boldsymbol{\beta}$ for which the corresponding vectors of Chern numbers $$(\chi(\O_S), K_S^2, \ldots)$$ are $\Q$-independent. As in Section \ref{sec:loc}, on a basis of toric surfaces, $\mathsf{Z}_S(\boldsymbol{\alpha},\boldsymbol{\beta}, \boldsymbol{t},q)$ can (in principle) be determined up to given order in $q$ by applying the Atiyah-Bott localization formula, which then determines the universal functions $\widetilde{A}, \widetilde{B}, \ldots$ up to some order in $q$. See \cite{GK1}--\cite{GK5,GKW} for such calculations.

\begin{remark} \label{strongmoc}
Conjecturally,  Conditions (i) and (iii) can be dropped from Theorem \ref{thm:multinst} and the sum in the formula can be replaced by the sum over \emph{all} classes $\boldsymbol{\beta} \in H^2(S,\Z)^r$ satisfying $\beta_1 + \cdots +\beta_r=c_1$. See also \cite{GNY1} and \cite{GK1}--\cite{GK5,GKW}. Furthermore, we expect that Conjecture \ref{conj:multinst} and Theorem \ref{thm:multinst} can be extended to the semistable case by using Mochizuki-Joyce-Song pairs \cite{Moc, GJT}. 
\end{remark}

\begin{remark}
Condition (ii) is essential and can be equivalently stated as 
$$
c_2 < \tfrac{1}{2}c_1(c_1-K_S) + 2\chi(\O_S).
$$ 
At first glance, this appears to be a strong restriction. However, applying $- \otimes \O_S(\ell H)$ induces an isomorphism of moduli spaces 
$$
\cM_{S}^{H}(r,c_1,c_2)\cong \cM_S^H(r,c_1+ r \ell H,c_2+(r-1)\ell Hc_1+\tfrac{1}{2}r(r-1)\ell^2 H^2).
$$ 
Note that this isomorphism leaves the invariant of Theorem \ref{thm:multinst} unchanged. This crucially depends on the fact that $\mathsf{P}_{\vd,\boldsymbol{a}}$ only involves  
$$
\ch(\EE \otimes \det(\EE)^{-1 / r}) = \ch(\EE(\ell H) \otimes \det(\EE(\ell H))^{- 1 / r}).
$$
Under the above isomorphism, Condition (ii) becomes 
$$
c_2 < \tfrac{1}{2}c_1(c_1-K_S) + 2\chi(\O_S) + (c_1 - \tfrac{r}{2} K_S) \ell H + \tfrac{r}{2} \ell^2 H^2,
$$ 
so the upper bound on $c_2$ can be made \emph{arbitrarily large} by taking $\ell \gg 0$.
\end{remark}

We state three interesting consequences of Conjecture \ref{conj:multinst}: 
\begin{itemize}
\item The invariant $\int_{[\cM]^{\vir}} \mathsf{P}_{\vd,\boldsymbol{a}}$ is independent of $H$. This would also follow from Thm.~\ref{thm:multinst} if Conditions (i) and (iii) can be dropped (Remark \ref{strongmoc}).
\item The invariant $\int_{[\cM]^{\vir}} \mathsf{P}_{\vd,\boldsymbol{a}}$ is unchanged upon replacing $c_1$ by $c_1 + r \gamma$ for any algebraic $\gamma \in H^2(S,\Z)$. This is obvious when Gieseker stability coincides with $\mu$-stability, and there are no strictly semistables (simply because $- \otimes \O_S(\gamma)$ is then an isomorphism on moduli spaces preserving the invariants).
\item Restricting to smooth projective surfaces $S$ with $K_S$ very ample and $b_1(S) = 0$, we have that $0,K_S$ are the only Seiberg-Witten basic classes, and the corresponding Seiberg-Witten invariants are $1, (-1)^{\chi(\O_S)}$ \cite[Thm.~7.4.1]{Mor}. Recall that $\chi(\O_S)$ and $K_S^2$ can be expressed in terms of topological Euler characteristic $e(S)$ and signature $\sigma(S)$. If we take $\alpha_1, \ldots, \alpha_N$ to be fixed polynomials expressions in $[K_S] \in H^*(S,\Q)$, $H = K_S$, and $c_1 = \ell K_S$ such that $\gcd(r,Hc_1)=1$, then it follows from Theorem \ref{thm:multinst} that $\int_{[\cM]^{\vir}} \mathsf{P}_{\vd,\boldsymbol{a}}$ only depends on $e(S)$ and $\sigma(S)$, and is therefore a \emph{topological} invariant.\footnote{See the proof of \cite[Thm.~1.8]{GK5} for more details, e.g.~how the assumptions of Theorem \ref{thm:multinst} can be made to hold.}  
\end{itemize}

Although Theorem \ref{thm:multinst} goes some way towards proving Conjecture \ref{conj:multinst}, the authors do not know how to evaluate the residues other than by computer calculation in examples. 
It appears that the new wall-crossing framework for virtual classes via vertex algebras developed by D.~Joyce \cite{Joy} (see also \cite{GJT}) will lead to a proof of Conjecture \ref{conj:multinst}. 

\section{$K$-theoretic refinement} \label{sec:ref}

Let $(S,H)$ be a smooth polarized surface with $H_1(S,\Z) = 0$, and let $r>0$, $c_1 \in H^2(S,\Z)$, $c_2 \in H^4(S,\Z)$.
Since $N:=N_S^H(r,c_1,c_2)$ has a $\C^*$-equivariant perfect obstruction theory, we also have a virtual structure ``sheaf''
$$
\O^{\vir}_N \in K_0^{\C^*}(N).
$$
As we learn from the work of N.~Nekrasov and A.~Okounkov \cite{NO}, one should rather consider the twisted virtual structure sheaf
$$
\widehat{\O}^{\vir}_N = \O^{\vir}_N \otimes (K_{N}^{\vir})^{\frac{1}{2}},
$$
where $(K_{N}^{\vir})^{\frac{1}{2}}$ is a square root of the virtual canonical bundle $K_N^{\vir} = \det (T_N^{\vir})^{\vee}$. Assuming, for the moment, that there are no rank $r$ strictly $H$-semistable Higgs pairs on $S$ with Chern classes $c_1, c_2$, symmetry of the perfect obstruction theory implies that there exists a canonical square root of $K_N^{\vir}$ on the fixed locus $N^{\C^*}$ \cite[Prop.~2.6]{Tho}. Then Thomas proposes to study $K$-theoretic Vafa-Witten invariants (defined by the $K$-theoretic virtual localization formula)
\begin{equation} \label{defKinv}
\chi(N, \widehat{\O}^{\vir}_N) = \chi\Big(N^{\C^*}, \frac{ \O_{N^{\C^*}}^{\vir} \otimes (K_{N}^{\vir})^{\frac{1}{2}} |_{N^{\C^*}}  } {\Lambda_{-1} (\nu^{\vir})^{\vee}   }\Big)  \in \Q(y^{\frac{1}{2}}).
\end{equation}
Here $\nu^{\vir}$ is the virtual normal bundle and $y := e^t$ with $t = c_1^{\C^*}(\mathfrak{t})$ the equivariant parameter of the $\C^*$ scaling action. 
Furthermore, for a vector bundle $V$, we denote $\Lambda_{-1} V := \sum_i (-1)^i \Lambda^i V$, and this definition extends to arbitrary $K$-theory classes as in \cite[Sect.~4]{FG}.
This expression is invariant under $y \rightarrow y^{-1}$ and reduces to $\int_{[N]^{\vir}} 1$ for $y=1$ (see \cite{Tho} for details). The contribution from the Gieseker-Maruyama moduli space $M:=M_S^H(r,c_1,c_2)$ to \eqref{defKinv} is the (signed and symmetrized) \emph{virtual Hirzebruch $\chi_{-y}$ genus} (defined in general in \cite{FG})
\begin{equation*} 
(-1)^{\vd} y^{-\frac{\vd}{2}} \chi_{-y}^{\vir}(M), \quad \vd = \vd(M).
\end{equation*}

Still assuming there are no strictly $H$-semistable objects, the $K$-theoretic $\SU(r)$ Vafa-Witten partition function is defined by 
$$
\sfZ_{S,H,c_1}^{\SU(r)}(q,y) = r^{-1} q^{-\frac{\chi(\O_S)}{2r} + \frac{r K_S^2}{24}} \sum_{c_2 \in \Z} q^{\frac{1}{2r} \vd(r,c_1,c_2)} (-1)^{\vd(r,c_1,c_2)}  \chi(N_S^H(r,c_1,c_2), \widehat{\O}^{\vir}_{N_S^H(r,c_1,c_2)}).
$$
As in the unrefined case discussed in the introduction, the generating function can be decomposed according to contributions from different components of the fixed loci indexed by sequences $\mu$
$$
\sfZ_{S,H,c_1}^{\SU(r)}(q,y) = r^{-1} \sum_{\mu} \sfZ_{S,H,c_1}^{\mu}(q,y).
$$
We make two comments:
\begin{itemize}
\item For arbitrary $S,H,r,c_1$, there may be strictly $H$-semistable objects. Then $K$-theoretic Vafa-Witten invariants, and the generating series $\sfZ_{S,H,c_1}^{\SU(r)}(q,y)$, can still be defined using moduli spaces of Joyce-Song Higgs pairs \cite{Tho}.
\item When $p_g(S) > 0$, ``Thomas vanishing'' applies just like in the unrefined case \cite{Tho}. In particular, for $r$ \emph{prime}, only horizontal and vertical components contribute 
$$
\sfZ_{S,H,c_1}^{\SU(r)}(q,y) = r^{-1} \sfZ_{S,H,c_1}^{(r)}(q,y) + r^{-1} \sfZ_{S,H,c_1}^{(1^r)}(q,y).
$$
\end{itemize}

Next, we consider the following refined lattice theta functions
\begin{align*}
\Theta_{A_{r},\ell}(q,y) &= \sum_{v \in \mathbb{Z}^r} q^{\frac{1}{2} \langle v - \ell \lambda, v - \ell \lambda \rangle} y^{\langle v, M_{A_r}^{-1} (1, \ldots, 1) \rangle}, \quad \lambda = \frac{1}{r+1}(r,r-1, \ldots, 1) \\
t_{A_{r},\ell}(q,y) &= \frac{\Theta_{A_{r},0}(q,y)}{\Theta_{A_{r},\ell}(q,y)}, \quad \ell \in \Z,
\end{align*}
where $\langle \cdot, \cdot \rangle$ is the symmetric bilinear form of the $A_r$ lattice and $M_{A_r}$ is the corresponding symmetric matrix. See Appendix \ref{sec:theta} for details on lattice theta functions. We also consider 
$$
\phi_{-2,1}(q,y) = (y^{\frac{1}{2}} - y^{-\frac{1}{2}})^2 \prod_{n=1}^{\infty} \frac{(1- y q^n)^2 (1-y^{-1} q^n)^2}{(1-q^n)^4},
$$
which is the Fourier expansion of a \emph{weak Jacobi form} of weight $-2$ and index 1 and 
$$
q = e^{2 \pi i \tau}, \quad y = e^{2 \pi i z}, \quad (\tau,z) \in \mathfrak{H} \times \C.
$$
Using the same methods as described in Section \ref{sec:GT}, the third-named author established the following universality result for the vertical contribution \cite{Laa1, Laa2}:
\begin{theorem}[Laarakker] \label{thm2:Laarakker}
For any $r>1$, there exist $C_0$, $\{C_{ij}\}_{1 \leq i \leq j \leq r-1} \in \Q(y^{\frac{1}{2}})(\!(q^{\frac{1}{2r}})\!)$ with the following property. For any smooth polarized surface $(S,H)$ satisfying $H_1(S,\Z) = 0$, $p_g(S)>0$, and $c_1 \in H^2(S,\Z)$, we have
\begin{align*}
\frac{\sfZ_{S,H,c_1}^{(1^r)}(q,y)}{(y^{\frac{1}{2}} - y^{-\frac{1}{2}})^{\chi(\O_S)}} = &\Bigg( \frac{(-1)^{r-1}}{ \phi_{-2,1}(q^r, y^r)^{\frac{1}{2}} \Delta(q^r)^{\frac{1}{2}}} \Bigg)^{\chi(\O_S)} \Bigg( \frac{\Theta_{A_{r-1},0}(q,y)}{\eta(q)^r} \Bigg)^{-K_S^2} \\
&\times  C_0(q,y)^{K_S^2} \sum_{\bfbeta \in H^2(S,\Z)^{r-1}}  \delta_{c_1,\sum_i i \beta_i} \prod_{i} \SW(\beta_i) \prod_{i \leq j} C_{ij}(q,y)^{\beta_i \beta_j}.
\end{align*}
\end{theorem}

The approach described in Section \ref{sec:GT} also works in the $K$-theoretic case. The calculations are worked out in detail in \cite{Laa1, Laa2}, where the normalized universal functions $\overline{C}_0(q,y), \overline{C}_{ij}(q,y)$ were calculated modulo $q^{15}$ for $r=2$ and modulo $q^{11}$ for $r=3$.
For this paper we also determined $\overline C_0(q,y), \overline C_{ij}(q,y)$
for $r=4,5$ modulo $q^{11}$ and for $r=6$ modulo $q^{10}$.
In this appendix, when we write $C_0, C_{ij}, \Theta_{A_r,\ell}$ (and later $D_0, D_{ij}, \Theta_{A_r^\vee,\ell}$), we always mean their refined versions depending on $(q,y)$. We also use the notation $C_I$, $I \subset [r-1]$ introduced at the beginning of Section \ref{sec:verconj}.
The following three conjectures are always verified up to the orders for which we determined the universal functions $C_0, C_{ij}$.\footnote{The first two conjectures previously appeared in \cite{GK1} ($r=2$) and \cite{GK3} ($r=3$) and were also verified, up to some order in $q$, in \cite{Laa1}.}

For $r=2$, the refinement simply consists of replacing  $t_{A_1,1}(q)$ by $t_{A_1,1}(q,y)$:
\begin{conjecture}[G\"ottsche-Kool] \label{conjverref:rk2}
For $r=2$, we have
$$
C_{\varnothing} = 1, \quad C_{\{1\}} = t_{A_1,1}.
$$
\end{conjecture}
For $r=3$, the refinement is unexpected and does not simply consist of replacing $t_{A_2,1}(q)$ by $t_{A_2,1}(q,y)$:
\begin{conjecture}[G\"ottsche-Kool] 
For $r=3$, we have
\begin{align*}
C_{\varnothing} = t_{A_2,1} X_-, \quad C_{\{1\}} = C_{\{2\}} = t_{A_2,1}, \quad C_{\{1,2\}} = t_{A_2,1} X_+,
\end{align*}
where $X_{\pm}$ are the solutions of
$$
X^2 - t_{A_2,1}(t_{A_2,1} + 3t_{A_2,1}(q,1))  X + t_{A_2,1} + 3t_{A_2,1}(q,1) = 0.
$$
\end{conjecture}

We now present a rather surprising conjectural refinement for $r=4$. Recall Ramanujan's octic continued fraction $u(q)$ in \eqref{defu}, and define
$$
J(q,y) = \frac{u(q^2)^4}{4} \frac{\phi_{-2,1}(q^2,y^2)\phi_{-2,1}(q^8,y^4)^2}{ \phi_{-2,1}(q^4,y^4)\phi_{-2,1}(q^4,y^2)^2}.
$$
Then we have
\begin{equation} \label{specJ}
J(q,1) = u(q^2)^4. 
\end{equation}

\begin{conjecture} \label{conjverref:rk4}
For $r=4$, we have
\begin{align*}
&C_{\varnothing} = \frac{Z - Z^{-1}}{t_{A_3,2}^{-1} J^{-1} - Z^{-1}}, \quad C_{\{1,3\}} = \frac{Z^{-1} - Z}{t_{A_3,2}^{-1} J^{-1} - Z}, \\
&C_{\{1\}} = C_{\{3\}} = ( J + 1)t_{A_3,1}, \quad C_{\{1,2\}} = C_{\{2,3\}} = ( 1+J^{-1} )t_{A_3,1}, \\
&C_{\{2\}} = \frac{Z - Z^{-1}}{t_{A_3,2}^{-1} Z - J}, \quad C_{\{1,2,3\}} = \frac{Z^{-1} - Z}{t_{A_3,2}^{-1} Z^{-1} - J},
\end{align*}
where $Z$ is a root of
\begin{align*}
Z - (J^{\frac{1}{2}} + J^{-\frac{1}{2}} + 2J(q,1)^{\frac{1}{2}} + 2J(q,1)^{-\frac{1}{2}}) + Z^{-1} = 0.
\end{align*}
\end{conjecture}
By \eqref{specJ}, specializing $y=1$ recovers Conjecture \ref{conjver:rk4}. 

\begin{remark}
For any $r$, $C_0(q,y), C_{ij}(q,y)$ satisfy the relations 
\begin{equation} \label{eqn:yinv}
C_0(q,y) = C_0(q,1/y), \quad C_{ij}(q,y) = C_{ij}(q,1/y).
\end{equation}
Indeed $\sfZ_{S,H,c_1}^{\SU(r)}(q,y) = \sfZ_{S,H,c_1}^{\SU(r)}(q,1/y)$ by \cite[Prop.~2.27]{Tho}. Moreover, for $S$ a smooth projective toric surface and $T$-equivariant classes $\boldsymbol{a}$, applying Atiyah-Bott localization to $\mathsf{G}_{S,\boldsymbol{a}}(q,y)$  (i.e.~the $K$-theoretic analogue of \eqref{eqn:defGS}, \eqref{eqn:GSafterAB}), one easily finds that $\mathsf{G}_{S,\boldsymbol{a}}(q,y) = \mathsf{G}_{S,\boldsymbol{a}}(q,1/y)$ from which \eqref{eqn:yinv} follows.
\end{remark}

On the horizontal branch, we conjecture the following.
\begin{conjecture}  
For any $r>1$, there exist $D_0$, $\{D_{ij}\}_{1 \leq i \leq j \leq r-1} \in \C(y^{\frac{1}{2}})[\![q^{\frac{1}{2r}}]\!]$ with the following property. For any smooth polarized surface $(S,H)$ satisfying $H_1(S,\Z) = 0$, $p_g(S)>0$, $c_1 \in  H^2(S,\Z)$, and $c_2 \in H^4(S,\Z)$ such that there are no rank $r$ strictly $H$-semistable sheaves on $S$ with Chern classes $c_1,c_2$, we have
$$
y^{-\frac{\vd}{2}}\chi_{-y}^{\vir}(M_S^H(r,c_1,c_2))
$$ 
equals the coefficient of $q^{c_2 - \frac{r-1}{2r} c_1^2 - \frac{r}{2} \chi(\O_S) + \frac{r}{24} K_S^2}$ of
\begin{align*}
&r^{2+K_S^2 - \chi(\O_S)} \Bigg( \frac{(y^{\frac{1}{2}} - y^{-\frac{1}{2}})}{\phi_{-2,1}(q^{\frac{1}{r}},y)^{\frac{1}{2}} \Delta(q^{\frac{1}{r}})^{\frac{1}{2}}} \Bigg)^{\chi(\O_S)} \Bigg( \frac{\Theta_{{A^\vee_{r-1}},0}(q,y)   }{\eta(q)^r} \Bigg)^{-K_S^2}  \\
&\quad\quad\quad\quad \times D_0(q,y)^{K_S^2} \sum_{\bfbeta \in H^2(S,\Z)^{r-1}}   \prod_{i} \epsilon_r^{i \beta_i c_1} \, \SW(\beta_i) \prod_{i \leq j} D_{ij}(q,y)^{\beta_i \beta_j}.
\end{align*}
\end{conjecture}
By the same reasoning as in Section \ref{sec:evir}, this conjecture is implied by Conjecture \ref{conj:multinst}. The only modification needed is that the formula for Euler characteristics of Hilbert schemes of points of a K3 surface \cite{Got1} should be replaced by the formula for their $\chi_y$-genera \cite{GS}.

We present the following refined analogue of Conjecture \ref{conjCD}:
\begin{conjecture} \label{conjCDref}
For any $r>1$, $C_0(q,y)$, $C_{ij}(q,y)$, $D_0(q,y)$, $D_{ij}(q,y)$ are Fourier expansions in $q = e^{2 \pi i \tau}$, $y=e^{2 \pi i z}$ of meromorphic functions $C_0(\tau,z)$, $C_{ij}(\tau,z)$, $D_0(\tau,z)$, $D_{ij}(\tau,z)$ on $\mathfrak{H} \times \C$ satisfying
$$
D_0(\tau,z) = C_0(- 1/\tau,z/\tau), \quad D_{ij}(\tau,z) = C_{ij}(- 1/\tau,z/\tau).
$$
\end{conjecture}

Then Conjectures \ref{conjverref:rk2}--\ref{conjverref:rk4}, \ref{conjCDref} imply formulae for $D_0, D_{ij}$ for $r=2$--$4$, which are obtained by making the following replacements (Appendix \ref{sec:theta}):
\begin{align*}
t_{A_{r-1},\ell} \big|_{(\tau,z) \mapsto (- 1/\tau,z/\tau)} &= t_{A_{r-1}^\vee,\ell}, \\
J \big|_{(\tau,z) \mapsto (- 1/\tau,z/\tau)} &=  \frac{\eta(q^{\frac{1}{2}})^4 \eta(q^\frac{1}{8})^8 }{\eta(q^\frac{1}{4})^{12}} \frac{\phi_{-2,1}(q^{\frac{1}{2}},y)\phi_{-2,1}(q^{\frac{1}{8}},y^{\frac{1}{2}})^2}{\phi_{-2,1}(q^{\frac{1}{4}},y)\phi_{-2,1}(q^{\frac{1}{4}},y^{\frac{1}{2}})^2 }.
\end{align*}

In \cite{GK3}, the first- and second-named authors introduced a \emph{$K$-theoretic $S$-duality} conjecture for $\mathsf{Z}_{S,H,c_1}^{\SU(r)}(q,y)$, and proved that, for $r=2,3$, the expressions in this section satisfy this conjecture. Just like in Section \ref{sec:S-dual}, we will not discuss $K$-theoretic $S$-duality for $\mathsf{Z}_{S,H,c_1}^{\SU(4)}(q,y)$, because we do not know the contribution of fixed loci indexed by the partition $(2,2) \vdash 4$. We did, however, check that the $K$-theoretic analogue of Proposition \ref{prop:Sdualpart} holds.

\section{Lattice theta functions} \label{sec:theta}

We review some facts about lattice theta functions  (e.g.~see \cite[Sect. 3.3]{GZ}). Let $\Gamma$ be a positive definite lattice of rank $r$.  We write 
$V=\Gamma\otimes_{\Z} \C$ and $V_\Q=\Gamma\otimes_{\Z}\Q$.  
We denote the bilinear form by $\langle \cdot ,\cdot \rangle$.
Denote by $X_\Gamma$ the set of meromorphic functions $f:{\mathfrak H}\times V\to \C$. 
For $(\lambda,\mu)\in V \times V$, define
\begin{align*}
f|(\lambda,\mu)(\tau,x) = q^{\frac{1}{2} \langle \lambda,\lambda\rangle} e^{2\pi i \langle \lambda,x+\mu/2\rangle} \,  f(\tau,x+\lambda\tau+\mu),
\end{align*}
where $q = \exp(2 \pi i \tau)$. Next, define
$$
f|_{r/2}S(\tau,x) = \Big( \frac{\tau}{i} \Big)^{-\frac{r}{2}} \, e^{-\pi i \langle x,x\rangle/\tau}f(-1/\tau,x/\tau).
$$
It is easy to see that 
\begin{align}\label{Sv} 
f|(\lambda,\mu)|_{r/2}S(\tau,x)= f|_{r/2}S |(\mu,-\lambda)(\tau,x).
\end{align}
The \emph{theta function of $\Gamma$} is defined by 
$$
\Theta_\Gamma(\tau,x) = \sum_{v\in \Gamma} q^{\frac{1}{2}\langle v, v\rangle} e^{2 \pi i \langle v ,x\rangle}.
$$ 
Then $\Theta_\Gamma \in X_{\Gamma}$. If $\Gamma$ has rank $r$ it is well-known that
\begin{align} \label{STh}
\sqrt{N} \, \Theta_{\Gamma}|_{r/2}S(\tau,x) =\Theta_{\Gamma^\vee}(\tau,x)=\sum_{v\in P} \Theta_{\Gamma}|(v,0)(\tau,x).
\end{align}
Here is $N$ the determinant of the matrix of the bilinear form on $\Gamma$ and
$$
\Gamma^\vee=\big\{v\in \Gamma\otimes_{\Z} \Q\bigm| \langle v,w\rangle\in \Z,  \forall \, w\in \Gamma \big\}
$$
is the dual lattice to $\Gamma$ (with bilinear form obtained by extending $\langle \cdot, \cdot \rangle$ from $\Gamma$ to $\Gamma^\vee$). For the second equality of \eqref{STh}, we assume that $\Gamma$ is integral and $P$ is a system of representatives of $\Gamma^\vee/\Gamma$.

We are interested in the lattice $A_{r}$. This is $\Z^{r}$ and, denoting the standard basis of $\Z^{r}$ by $(e_1, \ldots, e_{r})$, the symmetric bilinear form is determined by
$$
\langle e_i, e_j \rangle = \left\{ \begin{array}{cc} 2 & \textrm{if  } i=j \\ -1 & \textrm{if } |i-j|=1 \\ 0 & \textrm{otherwise.} \end{array}\right.
$$
We denote the corresponding symmetric matrix by $M_{A_{r}}$, then $\det(M_{A_{r}}) = r+1$. The dual lattice is given by 
$$
A_{r}^\vee = M_{A_{r}}^{-1}(\Z^{r}) \subset \Q^{r}
$$
with the same bilinear form. Taking the columns of $M_{A_{r}}^{-1}$ as a basis, the $A_{r}^\vee$ lattice is $\Z^{r}$ with bilinear form $\langle \cdot, \cdot \rangle^{\vee}$ determined by the matrix $M_{A_{r}^\vee} = M_{A_{r}}^{-1}$.

We define
$$
\lambda := M_{A_{r}}^{-1} (1,0, \ldots, 0) = \frac{1}{r+1}(r,r-1, \ldots,1).
$$
Then
$$
\big\{\ell\lambda\bigm| \ell=0,1,\ldots,r\big\}
$$
is a system of representatives of  $A_{r}^\vee/A_{r}$. 
With respect to the basis of column vectors of $M_{A_{r}}^{-1}$, the vector $\lambda$ corresponds to $(1,0, \ldots, 0)$.
For any $\ell \in \Z$, we define 
 \begin{align}
 \begin{split} \label{defTl}
\Theta_{A_{r},\ell}(\tau,x)&=\sum_{v\in \Z^{r}} q^{\frac{1}{2} \langle v-\ell\lambda,v-\ell\lambda \rangle} e^{2 \pi i \langle v - \ell \lambda, x \rangle}=\Theta_{A_r}|(-\ell \lambda,0)(\tau,x), \\
\Theta_{A_{r}^\vee,\ell}(\tau,x)&=\sum_{v\in \Z^{r}}q^{\frac{1}{2} \langle v,v \rangle^{\vee} }e^{2\pi i \langle v, x + \ell (1,0, \ldots, 0) \rangle^{\vee} }=\Theta_{A_r^\vee}|(0,\ell \lambda)(\tau,x). 
 \end{split}
 \end{align}
We define the following specialization
$$
\Theta_{A_{r}^\vee,\ell}(\tau,z) =\sum_{v\in \Z^{r}}q^{\frac{1}{2} \langle v,v \rangle^\vee }e^{2\pi i \langle v, (z,\ldots, z) + \ell (1,0,\ldots, 0) \rangle^\vee}.
$$
Here the vector $(1, \ldots, 1)$ is with respect to the basis of column vectors of $M_{A_{r-1}}^{-1}$, and we (accordingly) define
\begin{align*}
\Theta_{A_{r},\ell}(\tau,z) =\sum_{v\in \Z^r} q^{\frac{1}{2} \langle v-\ell\lambda,v-\ell\lambda \rangle} e^{2 \pi i \langle v - \ell \lambda, M_{A_{r}}^{-1} (z,\ldots, z) \rangle}.
\end{align*}
We denote the corresponding Nullwerte by $\Theta_{A_r,\ell}(\tau) = \Theta_{A_r,\ell}(\tau,0)$, $\Theta_{A_r^\vee,\ell}(\tau) = \Theta_{A_r^\vee,\ell}(\tau,0)$. Using the fact that 
$$
\langle (1, \ldots, 1), (1, \ldots, 1) \rangle^{\vee} = \frac{1}{12}r(r+1)(r+2),
$$ 
equations \eqref{defTl}, \eqref{STh}, \eqref{Sv} imply
\begin{align} \label{maintrans}
\Theta_{A_{r},\ell}(-1/\tau,z/\tau) = \frac{1}{\sqrt{r+1}} \Big( \frac{\tau}{i} \Big)^{\frac{r}{2}} e^{\frac{\pi i z^2}{\tau} \big( \frac{1}{12}r(r+1)(r+2) \big)} \Theta_{A_{r}^{\vee},\ell}(\tau,z).
\end{align}
One also easily verifies
\begin{equation} \label{eqn:A:ltominl}
\Theta_{A_r,\ell}(\tau,z)=\Theta_{A_r,-\ell}(\tau,-z) = \Theta_{A_r,-\ell}(\tau,z).
\end{equation}

Moreover, by the second equality of \eqref{STh}, we have 
\begin{equation} \label{keysum}
\Theta_{A_r^\vee,0}(\tau,z)=\sum_{\ell=0}^r \Theta_{A_{r},\ell}(\tau,z).
\end{equation}
As in the introduction, we write $\epsilon_{r+1} = \exp(2 \pi i/(r+1))$. Note that 
\begin{equation} \label{congruences}
\langle \lambda, \lambda \rangle \equiv - \frac{1}{r+1} \mod 1, \quad \langle \lambda,v \rangle \equiv 0 \mod 1,
\end{equation}
for all $v \in \Z^r$. Together with \eqref{defTl} and \eqref{keysum}, we deduce
\begin{equation} \label{TAdualasTA}
\Theta_{A_r^\vee,\ell}(\tau,z) =
\sum_{m=0}^r \epsilon_{r+1}^{\ell m} \, \Theta_{A_r,m}(\tau,z),
\end{equation}
which together with \eqref{eqn:A:ltominl} implies
\begin{equation} \label{eqn:Adual:ltominl}
\Theta_{A_r^\vee,\ell}(\tau,z)= \Theta_{A_r^\vee,-\ell}(\tau,z).
\end{equation}

Using \eqref{maintrans}, we also obtain 
$$
\Theta_{A_r,\ell}(\tau,z) =
\frac{1}{r+1}\sum_{m=0}^r \epsilon_{r+1}^{\ell m} \, \Theta_{A_r^\vee,m}(\tau,z).
$$
More generally, for any $k \in \Z$ we have
\begin{equation} \label{thlk}
\Theta_{A_r^\vee,\ell}(\tau+2k,z) 
=\sum_{m=0}^r \epsilon_{r+1}^{m(\ell-mk)} \, \Theta_{A_r,m} (\tau,z).
\end{equation}
It then follows from \eqref{thlk} and \eqref{maintrans} that, for any $m \in \Z$, we have
\begin{equation}
\begin{split}\label{thlkS}
\Theta_{A_r^\vee,\ell}(\tau + 2k,x)|_{r/2}S(\tau,(z, \ldots, z)) = \frac{1}{\sqrt{r+1}} \sum_{m=0}^r \Big(\sum_{n=0}^r \epsilon_{r+1}^{n(\ell+m-nk)} \Big) \Theta_{A_r,m}(\tau,z).
\end{split}
\end{equation}
We obtain identities between $\Theta_{A_r^\vee,\ell}(\tau+2k,z)$ and $\Theta_{A_r^\vee,\ell}(\tau + 2k,x)|_{r/2}S(\tau,(z,\ldots,z))$ for various values of $\ell$ and $k$ by comparing the right hand sides of equations \eqref{thlk} and \eqref{thlkS}. These identities yield the transformation properties of $\Theta_{A_{r-1}^\vee,\ell}(\tau + 2k)$, for $r=2,3,5$ used in Section \ref{sec:S-dual}.

\section{Universal series} \label{sec:data}

The normalized universal series $\overline{A}, \overline{B}, \overline{C}_{ij}$ were determined for $r=2$ modulo $q^{15}$ and $r=3$ modulo $q^{11}$ in \cite{Laa3}. 
We determined $\overline{A}, \overline{B}, \overline{C}_{ij}$ for $r=4,5,6,7$ modulo $q^{13}$. Recall that the normalization terms are given in Section \ref{sec:cob} and that $C_0,B$ are related by \eqref{def:C0Cij}. 
Below we only reproduce our data for $r \leq 5$. \\

\noindent {\bf Rank $2$.}
{\small\begin{align*}
\overline{A} =&1 + 12q^2  + 90q^4 + 520q^6 + 2535q^8  + 10908q^{10} +  42614q^{12} + 153960 q^{14} + O(q^{15}), \\ 
\overline{B} =& 1 - 4q + 7q^2 - 12q^3 + 23q^4 - 36q^5 + 56q^6 - 88q^7 + 128q^8 - 188q^9 + 273q^{10} \\
&- 384q^{11} + 536q^{12}  - 740 q^{13} + 1009 q^{14} +  O(q^{15}), \\            
\overline{C}_{11} =& 1 + 2q - q^2 - 2q^3 + 3q^4 + 2q^5 - 4q^6  - 4q^7 + 5q^8  + 8q^9 - 8q^{10} - 10q^{11} + 11q^{12} \\
&+ 12 q^{13} -15 q^{14}  + O(q^{15}).
\end{align*}}

\noindent {\bf Rank $3$.} 
{\small\begin{align*}
\overline{A} =&1 + 12q^3  + 90q^6 + 520q^9  + O(q^{11}), \\ 
\overline{B} =& 1 - \frac{9}{4}q - \frac{135}{8}q^2 + \frac{5687}{64}q^3 - \frac{21357}{128}q^4 + \frac{91395}{512}q^5 - \frac{976831}{1024}q^6 + \frac{127326087}{16384}q^7\\& - \frac{1460793393}{32768}q^8 + \frac{29764293777}{131072}q^9 - \frac{305367716529}{262144}q^{10} +O(q^{11}) , \\            
\overline{C}_{11} =&\overline{C}_{22} = 1 - \frac{7}{4}q + \frac{239}{16}q^2 - \frac{2361}{32}q^3 + \frac{100203}{256}q^4 - \frac{1106023}{512}q^5 + \frac{24745059}{2048}q^6\\& - \frac{280615001}{4096}q^7 + \frac{25731005619}{65536}q^8 - \frac{297290563675}{131072}x^9 +\frac{6913461089017
}{524288}q^{10} + O(q^{11}), \\
\overline{C}_{12}=& 1 + 5q - 7q^2 + 3q^3 + 15q^4 - 32q^5 + 9q^6 + 58q^7 - 96q^8 + 22q^9 + 149q^{10} + O(q^{11}).
\end{align*}}

\noindent {\bf Rank $4$.}
{\small\begin{align*}
\overline{A} =&1+12q^4+90q^8+520q^{12}+O(q^{13}), \\
\overline{B}= &1 - \frac{16}{9}q - \frac{644}{81}q^2 - \frac{18560}{729}q^3 + \frac{1384039}{6561}q^4 - \frac{169424}{6561}q^5 - \frac{221992196}{177147}q^6\\& + \frac{786352768}{1594323}q^7 + \frac{260413701079}{43046721}q^8 - \frac{3655109767904}{387420489}q^9 + \frac{68835523340380}{3486784401}q^{10} \\&- \frac{2818123363388416}{31381059609}q^{11} - \frac{24806213541619144}{282429536481}q^{12} + O(q^{13}),\\
\overline{C}_{11}=&\overline{C}_{33}= 1 - \frac{11}{9}q- \frac{20}{9}q^2 + \frac{29737}{729}q^3 - \frac{1005253}{6561}q^4 + \frac{2696689}{19683}q^5 + \frac{289688671}{177147}q^6\\& - \frac{16079563135}{1594323}q^7 + \frac{115191654533}{4782969}q^8 + \frac{15075699699385}{387420489}q^9 - \frac{1982510147928208}{3486784401}q^{10}\\& + \frac{2532736814456179}{1162261467}q^{11} - \frac{474057101483673346}{282429536481}q^{12} + O(q^{13}), \\
\overline{C}_{22}=& 1 - \frac{32}{9}q+ \frac{1058}{81}q^2 - \frac{15808}{729}q^3 - \frac{83617}{6561}q^4 + \frac{4834592}{19683}q^5 - \frac{150128758}{177147}q^6 \\&+ \frac{1226683328}{1594323}q^7 + \frac{283859838659}{43046721}q^8 - \frac{14627747102272}{387420489}q^9 + \frac{285981994115234}{3486784401}q^{10} \\&+ \frac{4003447693683712}{31381059609}q^{11} - \frac{471239600612088452}{282429536481}q^{12} + O(q^{13}), \\
\overline{C}_{12}=&\overline{C}_{23}= 1 + \frac{32}{9}q+ \frac{290}{81}q^2 - \frac{8768}{729}q^3 + \frac{15199}{6561}q^4 + \frac{2519200}{59049}q^5 - \frac{22602614}{177147}q^6\\& + \frac{399387712}{1594323}q^7 + \frac{2430989377}{14348907}q^8 - \frac{1587952009664}{387420489}q^9 + \frac{57170732310946}{3486784401}q^{10} \\&- \frac{636801299852288}{31381059609}q^{11} - \frac{34921739613648004}{282429536481}q^{12} + O(q^{13}),\\
\overline{C}_{13}=& 1 + \frac{2}{9}q + \frac{29}{3}q^2 - 50q^3 + \frac{1595}{9}q^4 - \frac{1642}{3}q^5 + 1564q^6 - \frac{37316}{9}q^7 + \frac{30863}{3}q^8 \\&- 24248q^9 + \frac{493144}{9}q^{10} - \frac{357934}{3}q^{11} + 251467q^{12} + O(q^{13}).
\end{align*}}

\noindent {\bf Rank $5$}
\small{\begin{align*}
\overline{A}=& 1 + 12q^5 + 90q^{10} + O(q^{13}),\\
\overline{B}=& 1 - \frac{25}{16}q - \frac{6425}{1152}q^2 - \frac{328325}{36864}q^3 - \frac{107959975}{2654208}q^4 + \frac{19714416547}{42467328}q^5 - \frac{682323694025}{3057647616}q^6 \\&- \frac{157438303448125}{195689447424}q^7 - \frac{42125816659848475}{14089640214528}q^8 - \frac{1222748753871559175}{225434243432448}q^9\\& + \frac{91589140507399236041}{1803473947459584}q^{10} - \frac{1482872905939874912425}{57711166318706688}q^{11}\\& - \frac{999082311255603435795425}{12465611924840644608}q^{12} + O(q^{13}),\\
\overline{C}_{11}=&\overline{C}_{44}= 1 - \frac{53}{48}q - \frac{455}{768}q^2 - \frac{214105}{55296}q^3 + \frac{137066833}{1769472}q^4 - \frac{32262250883}{127401984}q^5+ \frac{504641046613}{2038431744}q^6 \\& - \frac{35290739977297}{146767085568}q^7 + \frac{55324953373355585}{9393093476352}q^8 - \frac{21106479712125095327}{676302730297344}q^9 \\&+ \frac{747314378419516014391}{10820843684757504}q^{10} - \frac{21332084310732182836789}{259700248434180096}q^{11}\\& + \frac{3766630856661994639083391}{8310407949893763072}q^{12} + O(q^{13}),\\
\overline{C}_{22}=&\overline{C}_{33}= 1 - \frac{45}{16}q + \frac{1223}{1152}q^2 + \frac{92759}{4096}q^3 - \frac{4997165}{98304}q^4 - \frac{342060809}{4718592}q^5+ \frac{1507052138695}{3057647616}q^6\\&  - \frac{2357002948433}{21743271936}q^7 - \frac{64722230546795459}{14089640214528}q^8 + \frac{106028028804927959}{8349416423424}q^9\\& - \frac{4606809888515467877}{5410421842378752}q^{10} - \frac{2800069757219943079717}{57711166318706688}q^{11}\\& - \frac{638603945799703782159017}{12465611924840644608}q^{12} + O(q^{13}),\\
\overline{C}_{12}=&\overline{C}_ {34}=1 + \frac{155}{48}q + \frac{3631}{2304}q^2 + \frac{205649}{55296}q^3 - \frac{28529015}{1769472}q^4 - \frac{440005171}{42467328}q^5 + \frac{576927703355}{6115295232}q^6\\& - \frac{6282551916719}{146767085568}q^7 - \frac{17020188466524013}{28179280429056}q^8 +  \frac{1064059615856431819}{676302730297344}q^9 \\&+  \frac{1284173036440886659}{1202315964973056}q^{10} -  \frac{990167550880548509473}{86566749478060032}q^{11} \\&+  \frac{97471092940503634774661}{24931223849681289216}q^{12} + O(q^{13}),\\
\overline{C}_{13}=&\overline{C}_{24}= 1 +  \frac{7}{48}q +  \frac{15395}{2304}q^2 -  \frac{89765}{6144}q^3 +  \frac{114167459}{5308416}q^4 -  \frac{780769727}{14155776}q^5 +  \frac{1132411382279}{6115295232}q^6\\& -  \frac{9159845550413}{16307453952}q^7 +  \frac{49255281600761635}{28179280429056}q^8 -  \frac{454493672610745723}{75144747810816}q^9 \\&+  \frac{705748856184446680013}{32462531054272512}q^{10} -  \frac{6408400678095411411923}{86566749478060032}q^{11}\\& +  \frac{5938469859966951248118893}{24931223849681289216}q^{12} + O(q^{13}),\\
\overline{C}_{14}=& 1 +  \frac{1}{12}q +  \frac{7}{36}q^2 +  \frac{1777}{108}q^3 -  \frac{26873}{324}q^4 +  \frac{147607}{972}q^5 +  \frac{397291}{2916}q^6 -  \frac{11472803}{8748}q^7 \\&+  \frac{52508239}{26244}q^8 +  \frac{97770619}{19683}q^9 -  \frac{6130025513}{236196}q^{10} +  \frac{7730522729}{354294}q^{11}\\& +  \frac{309256092337}{2125764}q^{12} + O(q^{13}),\\
\overline{C}_{23}=& 1 +  \frac{9}{4}q +  \frac{1241}{144}q^2 -  \frac{487}{64}q^3 +  \frac{3809}{256}q^4 -  \frac{928975}{9216}q^5 +  \frac{1403185}{4096}q^6 -  \frac{53708245}{49152}q^7\\& +  \frac{768441731}{196608}q^8 -  \frac{3604410743}{262144}q^9 +  \frac{149919390323}{3145728}q^{10} -  \frac{6272767971295}{37748736}q^{11} \\&+  \frac{9726080833057}{16777216}q^{12} + O(q^{13}).
\end{align*}

\end{document}